%% file: main.tex
\RequirePackage{fix-cm}
\documentclass[smallextended]{svjour3}
\usepackage{amsmath,amssymb}
\usepackage{graphicx,enumerate,bbold}
\usepackage{mathptmx}      
\DeclareMathAlphabet{\mathcal}{OMS}{cmsy}{m}{n}
\usepackage{makecell}
\usepackage[notrig]{physics}
\usepackage{microtype}
\usepackage{mathtools} 
\usepackage{tikz}
\usepackage[multidot]{grffile} 

\usepackage[colorlinks=true, pdfstartview=FitV, linkcolor=blue, citecolor=blue, urlcolor=blue]{hyperref}

\numberwithin{equation}{section}
\numberwithin{figure}{section}

\smartqed  

\input{LatexDefinitions.tex}

\begin{document}
\title{Assignment Flows for Data Labeling on Graphs: Convergence and Stability\thanks{This work has also been stimulated by the Heidelberg Excellence Cluster STRUCTURES, funded by the DFG under Germany's Excellent Strategy EXC-2181/1 - 390900948.}\thanks{This version of the article has been accepted for publication, after peer review but is not the Version of Record and does not reflect post-acceptance improvements, or any corrections. The Version of Record is available online at: \url{http://dx.doi.org/10.1007/s41884-021-00060-8}.}}
\author{Artjom Zern \and Alexander Zeilmann \and Christoph Schn\"{o}rr}
\authorrunning{A.~Zern \and A.~Zeilmann \and C.~Schn\"{o}rr}
\institute{Artjom Zern, Alexander Zeilmann, Christoph Schn\"{o}rr \at Image and Pattern Analysis Group \\ Heidelberg University, Germany \\ \email{artjom.zern@gmail.com} %
}
\date{Received: date / Accepted: date}

\maketitle

\begin{abstract}
The assignment flow recently introduced in the \textit{J.~Math.~Imaging and Vision} 58/2 (2017) constitutes a high-dimensional dynamical system that evolves on a statistical product manifold and performs contextual labeling (classification) of data given in a metric space.
Vertices of an underlying corresponding graph index the data points and define a system of neighborhoods.
These neighborhoods together with nonnegative weight parameters define the regularization of the evolution of label assignments to data points, through geometric averaging induced by the affine e-connection of information geometry.
From the point of view of evolutionary game dynamics, the assignment flow may be characterized as a large system of replicator equations that are coupled by geometric averaging.

This paper establishes conditions on the weight parameters that guarantee convergence of the continuous-time assignment flow to integral assignments (labelings), up to a negligible subset of situations that will not be encountered when working with real data in practice.
Furthermore, we classify attractors of the flow and quantify corresponding basins of attraction.
This provides convergence guarantees for the assignment flow which are extended to the discrete-time assignment flow that results from applying a Runge-Kutta-Munthe-Kaas scheme for the numerical geometric integration of the assignment flow.
Several counter-examples illustrate that violating the conditions may entail unfavorable behavior of the assignment flow regarding contextual data classification.
\end{abstract}

\maketitle
\tableofcontents

\input{Introduction}
\input{AF-Properties}
\input{NumericalExamples}
\input{Conclusion}

\appendix
\input{Appendix}

\clearpage

\section*{ORCID iDs}
\noindent
Artjom Zern \includegraphics[height=3mm]{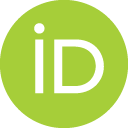} \url{https://orcid.org/0000-0002-0412-3337}\\
Alexander Zeilmann \includegraphics[height=3mm]{ORCIDiD_icon128x128} \url{https://orcid.org/0000-0002-8119-0349}\\
Christoph Schn\"{o}rr \includegraphics[height=3mm]{ORCIDiD_icon128x128} \url{https://orcid.org/0000-0002-8999-2338}

\bibliographystyle{spmpsci}
\bibliography{Literature}
\end{document}

%% file: LatexDefinitions.tex
\newcommand{\bitem}{\begin{itemize}}
\newcommand{\eitem}{\end{itemize}}
\newcommand{\mc}[1]{\mathcal{#1}}

\newcommand{\N}{\mathbb{N}}
\newcommand{\R}{\mathbb{R}}
\newcommand{\C}{\mathbb{C}}

\newcommand{\EE}{\mathbb{E}}

\newcommand{\bpm}{\begin{pmatrix}}
\newcommand{\epm}{\end{pmatrix}}
\newcommand{\bsm}{\left(\begin{smallmatrix}}
\newcommand{\esm}{\end{smallmatrix}\right)}
\newcommand{\T}{\top}

\newcommand{\ol}[1]{\overline{#1}}
\newcommand{\la}{\langle}
\newcommand{\ra}{\rangle}

\newcommand{\veps}{\varepsilon}

\newcommand{\w}{\omega}

\newcommand{\gdw}{\Leftrightarrow}

\newcommand{\eins}{\mathbb{1}}

\DeclareMathOperator{\Diag}{Diag}

\DeclareMathOperator{\rint}{rint}

\DeclareMathOperator{\argmax}{arg max}
\DeclareMathOperator{\supp}{supp}

\DeclareMathOperator{\Exp}{Exp}

\DeclareMathOperator{\im}{im}

\renewcommand{\vec}{\operatorname{vec}}
\newcommand{\GL}{\mathrm{GL}}
\newcommand{\D}{\mathrm{d}}
\newcommand{\Wstar}{\overline{\mathcal{W}}^*}

\newcommand{\baryW}{\mathbb{1}_{\mathcal{W}}}

%% file: Introduction.tex

\section{Introduction}\label{sec:Introduction}
\subsection{Problem and Motivation}

Metric \textit{data labeling} denotes the task to assign to each data point of a given finite set $\mc{F}_{I}=\{f_{i}\colon i\in I\}\subset \mc{F}$ in a metric space $(\mc{F},d_{\mc{F}})$ a unique \textit{label} (a.k.a.~\textit{prototype}, \textit{class representative}) from another given set $\mc{F}^{\ast}_{J}=\{f^{\ast}_{j}\colon j\in J\}\subset\mc{F}$. The data indices $i\in I$ typically refer to positions $x_{i} \in \R^{d}$ in space of in space-time $[0,T]\times\R^{d}$. Accordingly, one associates with the data a graph $G=(I,E)$ where the set of nodes $I$ indexes the data and the edge set $E\subset I\times I$ represents a neighborhood system. A basic example is provided by the data of a digital image observed on a regular grid graph $G$ in which case the data space $\mc{F}$ may be a color space, a high-dimensional Euclidean space like in multispectral imaging, or the positive-definite matrix manifold like in diffusion tensor medical imaging.

Data labeling provides a dramatic reduction of given data as Figure \ref{fig:motivation} illustrates. In addition, it is a crucial step for data \textit{interpretation}. Basic examples include the analysis of traffic scenes \cite{Cordts:2016aa}, of medical images or of satellite images in remote sensing.

\begin{figure}[htpb]
\begin{center}
\begin{tabular}{ccc}
\includegraphics[width=0.29\textwidth]{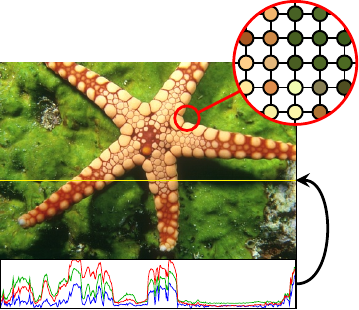} &
\includegraphics[width=0.29\textwidth]{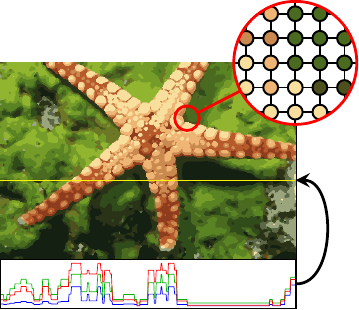} &
\includegraphics[width=0.29\textwidth]{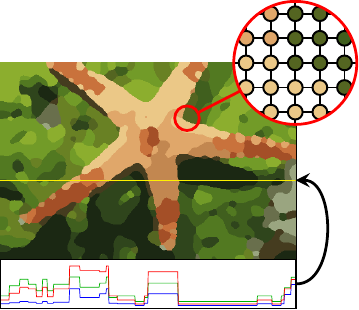}
\end{tabular}
\end{center}
\caption{
Data labeling on a graph through the assignment flow: Values from a finite set (so-called \textit{labels}) are assigned to a given vector-valued function so as to preserve its spatial structure on a certain spatial scale. \textsc{left:} Input data. \textsc{center:} Data labeled at a fine spatial scale. \textsc{right:} Data labeled at a coarser spatial scale. Scale is determined by the size $|\mc{N}_{i}|,\,i \in I$ of neighborhoods $\mc{N}_{i}$~\eqref{eq:def-Ni} that couple the individual dynamics~\eqref{eq:AF-i} (see Section~\ref{sec:AF0}). Here, uniform weight parameters $\Omega$~\eqref{eq:def-Omega} were used.
} \label{fig:motivation}
\end{figure}

The assignment flow approach \cite{Astroem2017} provides a mathematical framework for the design of dynamical systems that perform metric data labeling. This approach replaces established variational methods to image segmentation \cite{Chan:2006aa} as well as discrete Markov random fields for image labeling \cite{Kappes:2015aa} by smooth dynamical systems that facilitate the design of hierarchical systems for large-scale numerical data analysis. In addition, it can be extended to unsupervised scenarios \cite{Zern:2020ab} where the labels $\mc{F}_{J}^{\ast}$ can be adapted to given data or even learned from the data itself \cite{Zisler:2020aa}. We refer to the survey \cite{Schnorr:2020aa} for further discussion and related work.

Interpretation of data is generally not possible without an inductive bias towards prior expectations and application-specific knowledge. In connection with image labeling, such knowledge is represented by regularization parameters that influence label assignments by controlling the assignment flow. Figure \ref{fig:uniform-versus-weighted} provides an illustration. Nowadays, such parameters are learned directly from data. Due to the inherent smoothness, assignment flows can be conveniently used to accomplish this machine learning task \cite{Huhnerbein:2021th,Zeilmann:2021wt,Zeilmann:2021ul}.

\begin{figure}[htpb]
\centering
    \begin{tabular}{c@{\hskip2pt}c@{\hskip2pt}c@{\hskip2pt}c}
    	\includegraphics[width=0.23\textwidth]{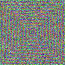}&
	\includegraphics[width=0.23\textwidth]{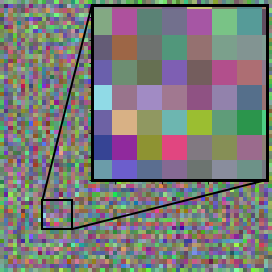}&
	\includegraphics[width=0.23\textwidth]{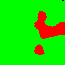}&
	\includegraphics[width=0.23\textwidth]{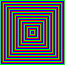}\\
{\scriptsize (a) noisy input image} &
{\scriptsize (b) close-up view} &
\makecell{\scriptsize (c) labeling with \\ \scriptsize \textit{uniform} weights} &
\makecell{\scriptsize (d) labeling with \\ \scriptsize \textit{adaptive} weights}
    \end{tabular}
\caption{
The assignment flow depends on the parameters $\Omega$~\eqref{eq:def-Omega}. (a),(b): Input data corrupted with noise. (c) Labeling with uniform weights. (d) Labeling with nonuniform weights  removes the noise, exploits spatial context and determines correct label assignments. Stability and asymptotic behavior of the assignment flow based on feasible parameters $\Omega$ are studied in this paper.
}
\label{fig:uniform-versus-weighted}
\end{figure}

From a more distant point of view, deep networks and learning \cite{Goodfellow:2016aa} prevail in machine learning. Besides their unprecedented performance in applications, current deep learning architectures are also known to be susceptible to data perturbations leading to unpredictable erroneous outputs \cite{Elad:2017aa,Finlayson:2019aa,Heaven:2019aa}. Our aim, therefore, is to prove stability properties of assignment flows under suitable assumptions on the regularization parameters, together with the guarantee that labelings, i.e.~\textit{integral} assignments, are computed for any data at hand. 

Section \ref{sec:Objectives} further details the scope of this paper after introducing the assignment flow approach in the next section.

\subsection{Assignment Flow}\label{sec:AF0}
The \textit{assignment flow} has been introduced by~\cite{Astroem2017} for the labeling of arbitrary data given on a graph $G=(I,E)$. 
It is defined by the system of nonlinear ODEs
\begin{equation}\label{eq:AF-i}
\dot W_{i} = R_{W_{i}} S_{i}(W),\qquad
W_{i}(0) = \tfrac{1}{n} \eins_n,\quad i \in I,
\end{equation}
whose solutions $W_i(t)$ evolve on the elementary Riemannian manifold $(\mc{S},g)$ given by the relative interior $\mc{S}=\rint(\Delta_{n})$ of the probability simplex
\begin{equation}
\Delta_{n} = \Big\{p \in \R^{n} \colon \sum_{j=1}^{n} p_{j} = \la\eins_{n},p\ra = 1,\ p \geq 0 \Big\}.
\end{equation}
Here, $n = |J|$ denotes the number of labels and $\eins_n = (1,\dotsc,1)^{\T} \in \R^n$ is the vector of ones.
The tangent space of $\mc{S}$ at any point $p \in \mc{S}$ is given by
\begin{equation}\label{eq:def-T0}
T_{0} = \{v \in \R^{n} \colon \la\eins_n,v\ra = 0\},
\end{equation}
and the Riemannian structure on $\mc{S}$ is defined by the Fisher-Rao metric
\begin{equation} \label{eq:def-FR}
g_{p}(u,v) = \sum_{j=1}^n \frac{u_{j} v_{j}}{p_{j}},\quad p \in \mc{S},\quad
u,v \in T_{0}.
\end{equation}

The basic idea underlying the approach~\eqref{eq:AF-i} is that each vector $W_{i}(t)$ converges within $\mc{S}$ to an $\veps$-neighborhood of some vertex (unit vector) $e_{j}$ of $\Delta_{n}$, that is
\begin{equation}
\forall \veps > 0\colon\qquad \|W_{i}(T)-e_{j}\| \leq
\veps,
\end{equation}
for sufficiently large $T = T(\veps)>0$. This enables to assign a unique label (class index) $j$ to the data point observed at vertex $i \in I$ by trivial rounding:
\begin{equation}\label{eq:local-rounding}
j = \underset{l \in \{1,\dotsc,n\}}{\argmax}~W_{i l}.
\end{equation}

In the following, we give a complete definition of the vector field defining the assignment flow \eqref{eq:AF-i}.
The linear mapping $R_{W_{i}}$ of~\eqref{eq:AF-i} will be called \textit{replicator matrix}. 
It is defined by
\begin{equation}\label{eq:def-Rp}
R_{p} \colon \R^{n} \to T_{0},\quad
R_{p} = \Diag(p)-p p^{\T},\quad p \in \mc{S}.
\end{equation}
Regarding the orthogonal projection onto $T_0$ given by
\begin{equation}\label{eq:def-Pi0-T0}
\Pi_{0} \colon \R^{n} \to T_{0},\qquad
\Pi_{0} = I_{n}-\tfrac{1}{n}\eins_{n}\eins_{n}^{\T}
\end{equation}
with $I_n$ denoting the identity matrix, the replicator matrix satisfies
\begin{equation}\label{eq:Rp-Pi0}
R_{p} = R_{p}\Pi_{0} = \Pi_{0} R_{p},\qquad \forall p \in \mc{S}.
\end{equation}
Further, we will use the exponential map and its inverse
\begin{subequations}\label{eq:def-exp-invexp}
\begin{align}
\exp_{p} &\colon \R^{n} \to \mc{S}, &
\exp_{p}(v) &= \frac{p e^{v}}{\la p, e^{v}\ra},\quad
p \in \mc{S},
\label{eq:def-exp} \\ \label{eq:def-invexp}
\exp_{p}^{-1} &\colon \mc{S} \to T_{0}, &
\exp_{p}^{-1}(q) &= \Pi_{0}\log\frac{q}{p},
\end{align}
\end{subequations}
where multiplication, division, exponentiation and logarithm of vectors is meant componentwise.
We call this map `exponential' for simplicity. In fact, definition~\eqref{eq:def-exp} is the explicit expression of the relation
\begin{equation}
\exp_{p} = \Exp_{p} \circ R_{p},
\end{equation}
where $\Exp \colon \mc{S} \times T_{0} \to \mc{S}$ is the exponential map corresponding to the affine e-connection of information geometry; see~\cite{Amari:2000aa,Ay2017} and~\cite{Schnorr:2020aa} for details. A straightforward calculation shows that the differential of $\exp_{p}$ at $v$ is
\begin{equation}\label{eq:dexp}
d\exp_{p}(v) = R_{\exp_{p}(v)},
\end{equation}
where the right-hand side is defined by~\eqref{eq:def-Rp} and~\eqref{eq:def-exp}.

The behavior of the assignment flow~\eqref{eq:AF-i}, essentially rests upon the \textit{coupling} of the local systems through the mappings $S_{i}$ within local neighborhoods
\begin{equation}\label{eq:def-Ni}
\mc{N}_{i} = \{i\} \cup \{k \in I \colon i \sim k\},\qquad i \in I,
\end{equation}
corresponding to the adjacency relation $E \subseteq I \times I$ of the underlying graph $G$. 
These couplings are parameterized by nonnegative \textit{weights}
\begin{equation}\label{eq:def-Omega}
\Omega = \{\w_{ik}\}_{k \in \mc{N}_{i}, i \in I}.
\end{equation}

Considering the \textit{assignment manifold}
\begin{equation}\label{eq:def-mcW}
\mc{W} = \mc{S} \times\dotsb\times \mc{S}, \qquad(\text{$|I|$ times})
\end{equation}
the \textit{similarity map} $S \colon \mc{W} \to \mc{W}$ is defined by
\begin{subequations}\label{eq:def-Si}
\begin{align}
S_{i}\colon \mc{W} \to \mc{S},\qquad
S_{i}(W) &= \Exp_{W_{i}}\Big(\sum_{k \in \mc{N}_{i}} \w_{ik} \Exp_{W_{i}}^{-1}\big(L_{k}(W_{k})\big)\Big), \quad i \in I
\label{eq:def-Si-a} \\ \label{eq:def-Li}
L_{i} \colon \mc{S} \to \mc{S},\qquad
L_{i}(W_{i}) &= \exp_{W_{i}}(-D_{i}),\quad i \in I.
\end{align}
\end{subequations}
It regularizes the assignment vectors $W_{i} \in \mc{S}$ depending on the parameters~\eqref{eq:def-Omega}, for given input data in terms of distance vectors $D_{i} \in \R^n$ storing the distances $D_{ij} = d_{\mc{F}}(f_i, f_j^*)$ between data points $f_i \in \mc{F}_{I}$ and prototypes $f_j^* \in \mc{F}_{J}^*$.
Denoting the barycenter of $\mc{S}$ with $\eins_{\mc{S}} = \frac{1}{n} \eins_{n}$, the defining relation~\eqref{eq:def-Si-a} can be rewritten in the form~\cite[Lemma 3.2]{Savarino:2021wt}
\begin{equation} \label{eq:def-Si-b}
S_{i}(W) = \exp_{\eins_{\mc{S}}}\Big(\sum_{k \in \mc{N}_{i}}\w_{ik}\big(\exp_{\eins_{\mc{S}}}^{-1}(W_{k})-D_{k}\big)\Big),\quad i \in I.
\end{equation}

In view of~\eqref{eq:def-mcW}, all the mappings in \eqref{eq:def-Rp}, \eqref{eq:def-exp-invexp} and \eqref{eq:def-Si} naturally generalize from $\mc{S}$ to $\mc{W}$ and from $T_{0}$ given by~\eqref{eq:def-T0} to
\begin{equation}
\mc{T}_{0} = T_{0} \times \dotsb \times T_{0}. \qquad(\text{$|I|$ times})
\end{equation}
For example,
\begin{equation}
\exp_{W}(V) = \big(\exp_{W_{1}}(V_{1}),\dotsc,\exp_{W_{|I|}}(V_{|I|})\big)^{\T}.
\end{equation}
We also denote the barycenter of $\mc{W}$ with $\eins_{\mc{W}} = ( \eins_{\mc{S}}, \dotsc, \eins_{\mc{S}})^\T$.
Accordingly, collecting all equations of~\eqref{eq:AF-i}, the assignment flow reads
\begin{equation}\label{eq:AF}
\dot W = R_{W} S(W),\qquad W(0)=\eins_{\mc{W}}.
\end{equation}

\subsection{Objectives}\label{sec:Objectives}
The first goal of this paper is to analyze the asymptotic behavior of the assignment flow~\eqref{eq:AF-i} depending on the parameters~\eqref{eq:def-Omega}.
It was conjectured~\cite[Conjecture 1]{Astroem2017} that, for data in `general position' as they are typically observed in real scenarios (e.g.\ no symmetry due to additive noise), the assignment flow converges to an integral labeling at every pixel, as described above in connection with~\eqref{eq:local-rounding}.
We confirm this conjecture in this paper under suitable assumptions on the parameters $\Omega$.
To this end, we use a reparametrization of the assignment flow and clarify the convergence of the reparameterized flow to equilibria and their stability.

The second goal of this paper concerns the same question regarding the \textit{time-discrete} assignment flow that is generated by a scheme for numerically integrating~\eqref{eq:AF-i}. Depending on what scheme is chosen, properties of the resulting flow may differ from properties of the \textit{time-continuous} flow~\eqref{eq:AF-i}. Indeed, the authors of~\cite{Astroem2017} adopted a numerical scheme from~\cite{losert1983dynamics} which, when adapted and applied to~\eqref{eq:AF-i}, was shown in~\cite{bergmann2017iterative} to always converge to a constant solution, i.e.\ a \textit{single} label is assigned to every pixel no matter which data are observed. Even though numerical experiments strongly indicate that this undesirable asymptotic behavior is irrelevant in practice, because it only occurs when $W(t_{k})$ is so close to the boundary of the closure of the underlying domain such that it cannot be reproduced with the usual machine accuracy, such behavior---nonetheless---is unsatisfactory from the mathematical viewpoint.

In this paper, therefore, we consider the simplest numerical scheme that was recently devised and studied in~\cite{Zeilmann:2018aa}, to better take into account the geometry underlying the assignment flow~\eqref{eq:AF-i} than the numerical scheme adopted in~\cite{Astroem2017}.  We show under suitable assumptions on the parameters $\Omega$, that the time-discrete assignment flow generated by such a proper numerical scheme cannot exhibit the pathological asymptotic behavior mentioned above.

\subsection{Related Work}
The assignment flow approach emerged from classical methods (variational methods, discrete Markov random fields) to image segmentation and labeling. We refer to \cite{Schnorr:2020aa} for further discussion. The approach can take into account any differentiable data likelihood, and all discrete decisions like the formation of spatial regions at a certain scale are done by integrating the flow numerically. The inherent smoothness of the approach compares favorably to discrete schemes for image segmentation, like region growing schemes \cite{Nock:2004vj}, in particular regarding the  learning of parameters for incorporating prior knowledge. In particular, spatial regularization can be performed independently of the metric model of the data at hand. This is not the case for segmentation based on spectral clustering \cite{Shi:2000aa} as discussed in detail and demonstrated by \cite{Zisler:2020aa}.

From a more distant viewpoint, our results may be also of interest in the field of evolutionary game dynamics~\cite{Hofbauer:2003aa,Sandholm:2010aa}. The corresponding basic dynamical system has the form
\begin{equation}\label{eq:replicator}
\dot p = p \big(f(p)-\EE_{p}[f(p)] \eins_{n}\big),\qquad p(0) \in \Delta_{n},
\end{equation}
where the first multiplication on the right-hand side is done componentwise, the expectation is given by $\EE_{p}[f(p)]=\la p, f(p)\ra$ and $p(t)$
evolves on $\Delta_{n}$.
The differential equation~\eqref{eq:replicator} is known as \textit{replicator equation}. It constitutes a Riemannian gradient flow with respect to the Fisher-Rao metric if $f = \nabla F$ derives from a potential $F$.
It is well known that depending on what `affinity function' $f \colon \Delta_{n} \to \R^{n}$ is chosen, a broad range of dynamics may occur, even for linear affinities $p \mapsto A p,\; A \in \R^{n \times n}$ (see e.g.~\cite{Bomze:2002aa}). Other choices give even rise to chaotic dynamics (see e.g.~\cite{Galla:2013aa}).
By comparison, the explicit form of Eq.~\eqref{eq:AF-i} reads
\begin{equation}\label{eq:dot-Wi-explicit}
\dot W_{i} = W_{i}\big(S_{i}(W)-\EE_{W_{i}}[S_{i}(W)]\eins_{n}\big),\qquad i \in I,
\end{equation}
where $S_{i}(W)$ \textit{couples} a possibly very large number $m = |I|$ of replicator equations of the form~\eqref{eq:dot-Wi-explicit}, as explained above in connection with~\eqref{eq:def-Omega}. The mapping $S_{i}$ does \textit{not} derive from a potential, however, but can be related to a potential after a proper reparametrization and under a symmetry assumption on the parameters~\eqref{eq:def-Omega}~\cite{Savarino:2021wt}. We refer to~\cite{Schnorr:2020aa} for a more comprehensive discussion of the background and further work related to the assignment flow~\eqref{eq:AF-i}.

\subsection{Organization}
The assignment flow and its basic properties (limit points, convergence, stability) are established in Section~\ref{sec:AF-Properties}. We briefly examine in Section~\ref{ssec:ConvergencePropertiesLAF} also properties of a simplified approximate version of the assignment flow, that can be linearly parametrized on the tangent space, which is convenient for data-driven estimation of suitable weight parameters~\cite{Huhnerbein:2021th}. In Section~\ref{sec:numerics}, we extend these results to the discrete-time assignment flow that is obtained by applying the simplest numerical scheme for geometric integration of the assignment flow, as worked out in~\cite{Zeilmann:2018aa}. Numerical examples demonstrate that violating the conditions established in Section~\ref{sec:AF-Properties} may lead to various behaviors of the assignment flow, all of which are unfavorable as regards data classification.
Some lengthy proofs have been relegated to Appendix~\ref{sec:Proofs}. We conclude in Section~\ref{sec:Conclusion}.

\subsection{Basic Notation}
We set $[n]=\{1,2,\dotsc,n\}$ for any $n \in \N$ and denote by $|S|$ the cardinality of any finite set $S$. Throughout this paper, $m$ and $n$ will denote the number of vertices of the underlying graph $G=(I,E)$ and the number of classes indexed by $J$, respectively,
\begin{equation}
m=|I|,\qquad n = |J|.
\end{equation}
The set $\mc{W} = \mc{S} \times \dots \times \mc{S}$ \eqref{eq:def-mcW} is called \textit{assignment manifold}, where $\mc{S} = \rint(\Delta_n)$ is the relative interior of the probability simplex $\Delta_n$.
$\mc{S}$ and $\mc{W}$, respectively, are equipped with the Fisher-Rao metric \eqref{eq:def-FR} and hence are Riemannian manifolds.
Points of $\mc{W}$ are row-stochastic matrices denoted by $W =(W_{1},\dotsc,W_{m})^{\T} \in \mc{W}$ with row vectors (also called subvectors) $W_{i} \in \mc{S},\, i \in I$ and with components $W_{ij},\, j \in J$.
The same notation is adopted for the image $S(W)$ of the mapping $S \colon \mc{W} \to \mc{W}$ defined by~\eqref{eq:def-Si}.
We denote the set of nonnegative reals by $\R_{\geq 0}$.
Parameters~\eqref{eq:def-Omega} form a matrix $\Omega \in \R_{\geq 0}^{m \times m}$. The subvectors of $\Omega S$ are denoted by $(\Omega S)_{i},\, i \in I$.

$\eins_n=(1,1,\dotsc,1)^{\T} \in \R^n$ denotes the vector with all entries equal to $1$ and $e_{i}=(0,\dotsc,0,1,0,\dotsc,0)^{\T}$ is the $i$th unit vector. 
The dimension of $e_{i}$ will be clear from the context. 
$\eins_{\mc{S}} = \frac{1}{n}\eins_{n}$ denotes the barycenter of $\mc{S}$ (uniform categorical distribution). Similarly, $\eins_{\mc{W}}$ with subvectors $(\eins_{\mc{W}})_{i}=\eins_{\mc{S}},\,i \in I$ denotes the barycenter of the assignment manifold $\mc{W}$.
$I_{n}$ denotes the identity matrix of dimension $n \times n$. 

The closure of $\mc{W}$ is denoted by
\begin{equation}\label{eq:def-olmcW}
\ol{\mc{W}}=\Delta_{n} \times\dotsb\times\Delta_{n}
\end{equation}
and the set of integral assignments (labelings) by
\begin{equation}\label{eq:def-mcWstar}
\Wstar = \ol{\mathcal{W}} \cap \{ 0, 1 \}^{m \times n}.
\end{equation}
Each subvector $W_{i}$ of a point $W \in \Wstar$ is a unit vector $W_{i}=e_{j}$ for some $j \in J$.

The support of a vector $v \in \R^n$ is denoted by $\supp(v) = \{ i \in [n] \colon v_i \neq 0 \}$. $\la x, y\ra$ denotes the Euclidean inner product of vectors $x, y$ and $\la A, B\ra=\tr(A^{\T} B)$ the inner product of matrices $A, B$. The spectral (or operator) norm of a matrix $A$ is denoted by $\|A\|_{2}$.
For two matrices of the same size, $A \odot B$ denotes the Hadamard (entry-wise) matrix product. For $A \in \R^{m \times n}$, $B \in \R^{p \times q}$, the matrix $A \otimes B \in \R^{mp \times nq}$ denotes the Kronecker product of matrices with submatrices $A_{ij} B \in \R^{p \times q},\; i \in [m],\, j \in [n]$ (cf.~e.g.~\cite{Van-Loan:2000aa}). $\mc{N}(A)$ and $\mc{R}(A)$ denote the nullspace and the range of the linear mapping represented by $A \in \R^{m \times n}$. For strictly positive vectors with full support, like $p \in \mc{S}$ with $\supp(p)=[n]$, the entry-wise division of a vector $v \in \R^{n}$ by $p$ is denoted by $\frac{v}{p}$. Likewise, we set $p v = (p_{1} v_{1},\dotsc,p_{n} v_{n})^{\T}$.
The exponential \textit{function} and the logarithm apply componentwise to vectors, i.e.\ $e^{v}=(e^{v_{1}},\dotsc,e^{v_{n}})^{\T}$ and $\log p = (\log p_{1},\dotsc,\log p_{n})^{\T}$. For large expression as arguments, we also write
\begin{equation}
e^{v} = \exp(v),
\end{equation}
which should not be confused with the exponential map~\eqref{eq:def-exp-invexp} that is always written with subscript.
$\Diag(p)$ denotes the diagonal matrix with the components of the vector $p$ on its diagonal.

%% file: AF-Properties.tex

\section{Properties of the Assignment Flow}\label{sec:AF-Properties}
\subsection{Representation of the Assignment Flow}\label{sec:AF}
The following parametrization of the assignment flow will be convenient for our analysis.
\begin{proposition}[$S$-parametrization {\cite[Proposition~3.6]{Savarino:2021wt}}]\label{prop:S-AF}
The assignment flow~\eqref{eq:AF} is equivalent to the system
\begin{subequations} \label{eq:prop_Sparam}
\begin{alignat}{2}
\dot{S} &= R_{S}(\Omega S), \qquad& S(0) &= \exp_{\baryW}(-\Omega D),
\label{eq:prop_sflow} \\ \label{eq:prop_wflow}
\dot{W} &= R_{W} S, \qquad &W(0) &= \baryW.
\end{alignat}
\end{subequations}
More precisely, $W(t),\,t \geq 0$ solves~\eqref{eq:AF} if and only if it solves~\eqref{eq:prop_Sparam}.
\end{proposition}
The difference between~\eqref{eq:AF} and~\eqref{eq:prop_Sparam} is that the latter representation separates the dependencies on the data $D$ and the assignments $W$: The given data $D$ completely determines $S(t)$ through the initial condition of~\eqref{eq:prop_sflow}, and $S(t)$ completely determines the assignments $W(t)$ by~\eqref{eq:prop_wflow}.
In what follows, our focus will be on how the parameters $\Omega$ affect $S(t)$ and $W(t)$.
\begin{remark}[$S$-flow]
We call \textit{$S$-flow} system~\eqref{eq:prop_sflow} and its solution $S(t)$ in the remainder of this paper and use the short-hand $F$ for the vector field, i.e.~
\begin{equation}\label{eq:def-S-flow-F}
\dot S = F(S) = R_{S}(\Omega S),\qquad S(0) = S_{0} \in \mc{W}.
\end{equation}
\end{remark}

A direct consequence of the parametrization~\eqref{eq:prop_Sparam} is the following.
\begin{proposition} \label{prop:W-from-S}
Let $S(t),\; t \geq 0$ solve~\eqref{eq:prop_sflow}. Then the solution to~\eqref{eq:prop_wflow} is given by
\begin{equation} \label{eq:W-from-S}
W(t) = \exp_{\baryW} \left( \int_{0}^{t} S(\tau) \,\D\tau \right) = \exp_{\baryW} \left( \int_{0}^{t} \Pi_{0}S(\tau) \,\D\tau \right).
\end{equation}
\end{proposition}
\begin{proof}
Set $I_{S}(t)=\int_{0}^{t} S(\tau)\dd\tau$. Then $W(t) = \exp_{\eins_{\mc{W}}}\big(I_{S}(t)\big)$ and
\begin{equation}
\dot W(t) = d\exp_{\eins_{\mc{W}}}\big(I_{S}(t)\big)\big[\dot I_{S}(t)\big] \overset{\eqref{eq:dexp}}{=} R_{\exp_{\eins_{\mc{W}}}(I_{S}(t))}\big(\dot I_{S}(t)\big)
= R_{W(t)}\big(S(t)\big).
\end{equation}
The second equation of~\eqref{eq:W-from-S} follows from the first equation of~\eqref{eq:Rp-Pi0}.\qed
\end{proof}
%

Transferring the assignment flow~\eqref{eq:AF} to the tangent space $\mc{T}_{0}$
and linearizing the ODE leads to the linear assignment flow~\cite[Prop.~4.2]{Zeilmann:2018aa}
\begin{alignat}{1}\label{eq:LAF}
    \dot{V} &= R_{\widehat{S}}(\Omega V) + B, \quad V(0) = 0, \quad V \in \mc{T}_0,
\end{alignat}
with fixed $\widehat{S} \in \mc{W}$ and $B \in \mc{T}_{0}$.

We note that both the $S$-flow~\eqref{eq:def-S-flow-F} and the linear assignment flow~\eqref{eq:LAF} are defined by similar vector fields on the tangent space $\mc{T}_{0}$.
Ignoring the constant term $B$ in~\eqref{eq:LAF}  that can be represented by
using a corresponding initial point (see Lemma~\ref{lmm:additive-to-initial}), the difference concerns the parameters $S$ and $\widehat S$
of the replicator matrix:
In the linear assignment flow, this parameter $\widehat S$ is fixed, whereas in the $S$-flow,
it changes with the flow.
Notice that `linear' refers to the linearity of the ODE~\eqref{eq:LAF} on the tangent space.
The corresponding lifted flow~\eqref{eq:LAF-lifted} on the assignment manifold
is still nonlinear (cf.~\cite[Def.~4.1]{Zeilmann:2018aa}).

Convergence properties of the $S$-flow and the linear assignment flow are
analyzed in the following sections.

\subsection{Existence and Uniqueness}
We establish global existence and uniqueness of both the $S$-flow and the assignment flow and examine to what extent the former determines the latter.
\begin{proposition}[global existence and uniqueness]
The solutions $W(t), S(t)$ to~\eqref{eq:prop_Sparam} are unique and globally exist for $t \geq 0$.
\end{proposition}
\begin{proof}
The hyperplanes $\{ S \colon \sum_{j} S_{ij} = 1 \}$ for $i \in I$ and $\{ S \colon S_{ij} = 0 \}$ for $i \in I$, $j \in J$ are invariant with respect to the flow~\eqref{eq:def-S-flow-F}.
Hence, $S(t)$ stays in $\mathcal{W} \subset \ol{\mathcal{W}}$ (cf.~\cite{Hofbauer:2003aa}) and therefore exists for all $t \in \R$ by~\cite[Corollary 2.16]{Teschl:2012aa}.
Equation~\eqref{eq:W-from-S} then implies the existence of $W(t)$ for all $t \in \R$.
The uniqueness of the solutions follow by the local Lipschitz continuity of the right-hand
side of~\eqref{eq:def-S-flow-F} and~\eqref{eq:AF}, respectively.\qed
\end{proof}
\begin{remark}$\text{ }$\label{rem:replicator}
\begin{enumerate}[(a)]
\item It is clear in view of the representation~\eqref{eq:prop_Sparam} that the domain $\mc{W}$ of the $S$-flow and consequently the domain of the assignment flow, too, can be extended to $\ol{\mc{W}}$, and we henceforth assume this to be the case.
Furthermore, the domain of the $S$-flow can be extended to an open set $U$ with $\ol{\mc{W}} \subset U \subseteq \R^{m \times n}$. In the latter case, although the existence for all $t \geq 0$ is no longer guaranteed, this simplifies the stability analysis of equilibria $S^* \in \ol{\mc{W}}$, as we will see in Section~\ref{sec:convergence-stability}.
\item The assignment flow shares with replicator equations in general (cf.~\cite{Hofbauer:2003aa}) that it is invariant with respect to the boundary $\partial\ol{\mc{W}}$: Due to the multiplication with $R_{S}$ and $R_{W}$, respectively, both $S(t)$ and $W(t)$ cannot leave the corresponding facet of $\partial\ol{\mc{W}}$ whenever they reach it.
\end{enumerate}
\end{remark}
Next, we examine what convergence of $S(t)$ close to $\partial\ol{\mc{W}}$ implies for $W(t)$.
\begin{proposition}\label{prop:limit-W}
Let
\begin{equation}\label{eq:def-Voronoi-cell}
\mc{V}_{j} = \big\{p \in \Delta_{n}\colon p_{j} > p_{l},\; \forall l \in [n]\setminus\{j\}\big\},\quad j \in [n]
\end{equation}
denote the Voronoi cells of the vertices of $\Delta_{n}$ in $\Delta_{n}$ and suppose $\lim_{t \to \infty} S_{i}(t) = S_{i}^{\ast} \in \Delta_{n}$, for any $i \in I$. Then the following assertions hold.
\begin{enumerate}[(a)]
\item If $S_{i}^{\ast} \in \mc{V}_{j^{\ast}(i)}$ for some label (index) $j^{\ast} = j^{\ast}(i) \in J$, then there exist constants $\alpha_{i}, \beta_{i} > 0$ such that
\begin{subequations}\label{eq:prop-limit-W-a}
\begin{gather}
\|W_{i}(t)-e_{j^{\ast}(i)}\|_{1} \leq \alpha_{i} e^{-\beta_{i} t},\quad \forall t \geq 0.
\intertext{
In particular,
}
\lim_{t \to \infty} W_{i}(t) = e_{j^{\ast}(i)}.
\end{gather}
\end{subequations}
\item
One has
\begin{multline}\label{eq:S-flow-fast-concergence}
\int_{0}^{\infty}\|S_{i}(t)-S_{i}^{\ast}\|_{1}\dd t <\infty
\\ \implies\quad
\lim_{t \to \infty} W_{i}(t) = W_{i}^{\ast} \quad\text{with}\quad
\supp(W_{i}^{\ast})=\argmax_{j \in J} S_{ij}^{\ast}.
\end{multline}
\end{enumerate}
\end{proposition}
\begin{proof}
See Appendix~\ref{sec:Proofs}.
\end{proof}
Proposition~\ref{prop:limit-W}(a) states that if any subvector of the $S$-flow converges to a Voronoi cell~\eqref{eq:def-Voronoi-cell}, then the corresponding subvector of $W(t)$ converges exponentially fast to the corresponding integral assignment.

Proposition~\ref{prop:limit-W}(b) handles the case when the limit point $S_{i}^{\ast}$ lies at the border of adjacent Voronoi cells, that is the set $\argmax_{j \in J} S_{ij}^{\ast}$ is \textit{not} a singleton. In this case, one can only state that $W_{i}(t)$ converges to some (possibly nonintegral) point $W_{i}^{\ast}$ without being able to predict precisely this limit based on $S_{i}^{\ast}$ alone. In contrast to (a), we also have to assume that the convergence of the $S$-flow is fast enough---see the hypothesis of~\eqref{eq:S-flow-fast-concergence}. This assumption is reasonable, however, because it is satisfied whenever $S_{i}^{\ast}$ is subvector of a \textit{hyperbolic} equilibrium point of the $S$-flow (cf.~Remark~\ref{remark:exp-convergence} below).
\begin{example}
We briefly demonstrate what may happen when the assumption of~\eqref{eq:S-flow-fast-concergence} is violated. Suppose $S_{i}(t)$ and $S_{i}^{*}$ are given by
\begin{equation}
S_i(t) = \begin{pmatrix} \frac{1}{2} - \frac{1}{t+1} \\ \frac{1}{2} - \frac{2}{t+1} \\ \frac{3}{t+1} \end{pmatrix} \ \longrightarrow\ S_i^* = \begin{pmatrix} \frac{1}{2} \\ \frac{1}{2} \\ 0 \end{pmatrix} \quad \text{for} \quad t \rightarrow \infty.
\end{equation}
The first component of $S_i(t)$ converges faster than the second component. Since $\|S_{i}(t) - S_{i}^{*}\|_{1} = \frac{6}{t+1}$, the convergence rate assumption of~\eqref{eq:S-flow-fast-concergence} does not hold. Calculating $W_i(t)$ due to~\eqref{eq:W-from-S} gives
\begin{equation}
W_i(t) = \frac{1}{1 + \frac{1}{t+1} + (t+1)^4 e^{-\frac{1}{2}t}} \begin{pmatrix} 1 \\ \frac{1}{t+1} \\ (t+1)^4 e^{-\frac{1}{2}t} \end{pmatrix} \longrightarrow W_i^* = \begin{pmatrix} 1 \\ 0 \\ 0 \end{pmatrix} \enskip \text{for $t \rightarrow \infty$},
\end{equation}
i.e.\ $W_i(t)$ still converges, but we have $\supp(W_{i}^{*}) \subsetneq \argmax_{j \in J} S_{ij}^{*}$ unlike the statement of~\eqref{eq:S-flow-fast-concergence}.
This example also shows that, in the case of Proposition~\ref{prop:limit-W}(b), the limit $W_{i}^{*}$ may depend on the \textit{trajectory} $S_{i}(t)$, rather than only on the limit point $S_{i}^{*}$ as in case (a).
\end{example}
Proposition~\ref{prop:limit-W} makes explicit that the $S$-flow largely determines the asymptotic behavior of $W(t)$. The next section, therefore, focuses on the $S$-flow~\eqref{eq:def-S-flow-F} and on its dependency on the parameters $\Omega$.

\subsection{Convergence to Equilibria and Stability} \label{sec:convergence-stability}
In this section, we characterize equilibria, their stability, and convergence properties of the $S$-flow~\eqref{eq:def-S-flow-F}. Quantitative estimates of the basin of attraction to exponentially stable equilibria will be provided, too.

\subsubsection{Characterization of Equilibria and Their Stability}

We show in this section under mild conditions that only \textit{integral} equilibrium points $S^{\ast} \in \ol{\mc{W}}^{\ast}$ can be stable.
\begin{proposition}[equilibria] \label{prop:crit-points}
Let $\Omega \in \R^{n \times n}$ be an arbitrary matrix.
\begin{enumerate}[(a)]
\item A point $S^* \in \ol{\mathcal{W}}$ is an equilibrium point of the $S$-flow~\eqref{eq:def-S-flow-F} if and only if
\begin{equation} \label{eq:equilibrium}
(\Omega S^*)_{ij} = \langle S_{i}^*, (\Omega S^*)_{i} \rangle, \qquad \forall j \in \supp(S_i^*),\qquad \forall i \in I,
\end{equation}
i.e., the subvectors $(\Omega S^*)_{i}$ are constant on $\supp S_i^*$, for each $i \in I$.
\item Every point $S^* \in \Wstar$ is an equilibrium point of the $S$-flow~\eqref{eq:prop_sflow}.
\item Let $J_{+} \subseteq J$ be a non-empty subset of indices, and let $\eins_{J_{+}} \in \R^{n}$ be the corresponding indicator vector with components $(\eins_{J_{+}})_{j}=1$ if $j \in J_{+}$ and $(\eins_{J_{+}})_{j}=0$ otherwise. Then $S^*= \tfrac{1}{|J_{+}|} \eins_{m} \eins_{J_{+}}^\top$ is an equilibrium point. In particular, the barycenter $\eins_{\mathcal{W}}=\tfrac{1}{n} \eins_{m} \eins_{n}^\top$ corresponding to $J_{+}=J$ is an equilibrium point.
\end{enumerate}
\end{proposition}
\begin{proof}\
\begin{enumerate}[(a)]
\item Each equation of the system~\eqref{eq:def-S-flow-F} has the form
\begin{equation}
\dot{S}_{ij} = S_{ij} \big( (\Omega S)_{ij} - \langle S_i, (\Omega S)_{i} \rangle \big),\quad i \in I,\; j \in J.
\end{equation}
$\dot S_{ij}=0$ implies $S_{ij}=S_{ij}^{\ast} \neq 0$ if $j\in \supp(S_{i}^{\ast})$ and that the term in the round brackets is zero, which is~\eqref{eq:equilibrium}.
\item The replicator matrix~\eqref{eq:def-Rp} satisfies $R_{e_{j}} \equiv 0,\; \forall j \in J$. This implies $R_{S^*} = 0$ and in turn $R_{S^*} (\Omega S^*) = 0$.
\item Since $\Omega S^* = \frac{1}{|J_+|} (\Omega \eins_{m}) \eins_{J_+}^\top$, the subvectors $(\Omega S^*)_i,\,i \in I$ are constant on $J_+ = \supp S_i^*$, which implies by (a) that $S^*$ is an equilibrium point.
\end{enumerate} \qed
\end{proof}
\begin{remark}
The set of equilibria characterized by Proposition~\ref{prop:crit-points} (b) and (c) may not exhaust the set of all equilibrium points for a general parameter matrix $\Omega$.
However, we will show below that, under certain mild conditions, any such additional equilibrium points must be unstable.
\end{remark}
Next, we study the stability of equilibrium points.
\begin{lemma}[Jacobian] \label{prop:jacobian}
Let $F(S)$ denote the vector field defining the $S$-flow~\eqref{eq:def-S-flow-F}. Then, after stacking $S$ row-wise,  the Jacobian matrix of $F$ is given by
\begin{equation}\label{eq:Jacobian-S-Flow}
\frac{\partial F}{\partial S} = \begin{pmatrix} B_1 & & \\ & \ddots & \\ & & B_{m} \end{pmatrix} + \begin{pmatrix} R_{S_{1}} & & \\ & \ddots & \\ & & R_{S_{m}} \end{pmatrix} \cdot \Omega \otimes I_{n}
\end{equation}
with block matrices $B_i = \Diag\big( (\Omega S)_i \big) - \langle S_i, (\Omega S)_i \rangle I_{n} - S_i (\Omega S)_i^\top$ and $R_{S_i}$ given by~\eqref{eq:def-Rp}.
\end{lemma}
\begin{proof}
The subvectors of $F$ have the form
\begin{equation}
F_i(S) = R_{S_i} (\Omega S)_i = \big( \Diag(S_i) - S_i S_i^\top \big) (\Omega S)_i,\qquad i \in I.
\end{equation}
Hence
\begin{subequations}\label{eq:Jacobian-dFi}
\begin{align}
\D F_i(S)[T] &= \tfrac{\D}{\D t} F_i(S+t T) |_{t=0} \\
	&= \big( \Diag(T_i) - T_i S_i^\top - S_i T_i^\top \big) (\Omega S)_i + R_{S_i} (\Omega T)_i \\
	&= \big( \Diag\big( (\Omega S)_i \big) - \langle S_i, (\Omega S)_i \rangle I_{n} - S_i (\Omega S)_i^\top \big) T_i + R_{S_i} (\Omega T)_i
	\\
	&= B_{i} T_{i} + R_{S_i} (\Omega T)_i.
\end{align}
\end{subequations}
We have $\D F(S)[T] = \tfrac{\partial F}{\partial S} \vec(T)$ with $\vec(T) \in \R^{m n}$ denoting the vector that results from stacking the row vectors (subvectors) of $T$. Comparing both sides of this equation, with the block matrices of the left-hand side given by~\eqref{eq:Jacobian-dFi}, implies~\eqref{eq:Jacobian-S-Flow}. \qed
\end{proof}

\begin{proposition}[eigenvalues of the Jacobian] \label{prop:Jacobian-eigen}
Let $S^* \in \overline{\mathcal{W}}$ be an equilibrium point of the $S$-flow~\eqref{eq:def-S-flow-F}, i.e.\ $F(S^{\ast}) = R_{S^{\ast}}(\Omega S^{\ast}) = 0$. Then regarding the spectrum $\sigma\big( \tfrac{\partial F}{\partial S}(S^*) \big)$, the following assertions hold.
\begin{enumerate}[(a)]
\item A subset of the spectrum is given by
\begin{equation}\label{eq:spectrum-a}
\sigma\big( \tfrac{\partial F}{\partial S}(S^*) \big) \supseteq \bigcup_{i \in I} \big\{ - \langle S_i^*, (\Omega S^*)_i \rangle \big\} \cup \big\{ (\Omega S^*)_{ij} - \langle S_i^*, (\Omega S^*)_i \rangle \big\}_{j \in J \setminus \supp(S_i^*)}.
\end{equation}
This relation becomes an equation if $S^*$ is integral, i.e.\ $S^* \in \Wstar$.
In the latter case, the eigenvectors are given by
\begin{equation} \label{eq:eigv-S-int}
e_i e_{j^*(i)}^\top \in \R^{m \times n}, \quad e_i (e_{j^*(i)} - e_j)^\top \in \mathcal{T}_0, \quad \forall j \in J \setminus \{ j^*(i) \}, \quad \forall i \in I.
\end{equation}
\item If $S^* = \tfrac{1}{|J_{+}|} \eins_{m} \eins_{J_{+}}^\top$ with $J_{+}\subseteq J$ and $|J_{+}| \geq 2$, then
\begin{equation} \label{eq:spectrum-bary}
\sigma\big( \tfrac{\partial F}{\partial S}(S^*) \big) = \bigcup_{i \in I} \Big\{ -\tfrac{(\Omega \eins_{m})_{i}}{|J_{+}|} \Big\} \cup \bigcup_{\lambda \in \sigma(\Omega)} \big\{ \tfrac{\lambda}{|J_{+}|} \big\}.
\end{equation}
\item Assume the parameter matrix $\Omega$ with elements $\w_{ii},\, i \in I$ on the main diagonal, is nonnegative. If $S_{i}^* \not\in \{0,1\}^{n}$ and $\omega_{ii} > 0$ hold for some $i \in I$, then the Jacobian matrix has at least one eigenvalue with positive real part.
The real and imaginary part of the corresponding eigenvector lie in
\begin{equation}
\mathcal{T}_{+} = \big\{ V \in \mathcal{T}_0 \colon \supp(V) \subseteq \supp(S^*) \big\}.
\end{equation}
\end{enumerate}
\end{proposition}
\begin{proof}
See Appendix~\ref{sec:Proofs}.
\end{proof}
Next, we apply Proposition~\ref{prop:Jacobian-eigen} and the stability criteria stated in Appendix~\ref{sec:Appendix} in order to classify the equilibria of the $S$-flow.
\begin{corollary}[stability of equilibria] \label{cor:stability-S}
Let $\Omega$ be a nonnegative matrix with positive diagonal entries. Then, regarding the equilibria $S^* \in \overline{\mathcal{W}}$ of the $S$-flow~\eqref{eq:def-S-flow-F}, the following assertions hold.
\begin{enumerate}[(a)]
\item $S^* \in \ol{\mc{W}}^{\ast}$ is exponentially stable if, for all $i \in I$,
\begin{multline} \label{eq:stable}
(\Omega S^*)_{i j} < (\Omega S^*)_{i j^*(i)} \quad \text{for all } j \in J \setminus \{ j^*(i) \} \\ \text{with} \quad \{ j^*(i) \} = \argmax_{j \in J}~S_{ij}^*.
\end{multline}
\item $S^* \in \ol{\mc{W}}^{\ast}$ is unstable if, for some $i \in I$,
\begin{multline}
(\Omega S^*)_{i j} > (\Omega S^*)_{i j^*(i)} \quad \text{for some } j  \in J \setminus \{ j^*(i) \} \\ \text{with} \quad \{ j^*(i) \} = \argmax_{j \in J}~S_{ij}^*.
\end{multline}
\item All equilibrium points $S^* \not\in \ol{\mc{W}}^{\ast}$ are unstable.
\end{enumerate}
\end{corollary}
\begin{proof}$\text{ }$
\begin{enumerate}[(a)]
\item
We apply Theorem~\ref{thm:stability}(a) that provides a condition for stability of the $S$-flow, regarded as flow on an open subset of $\R^{m\times n}$. Since the stability also holds on subsets, this shows stability of the $S$-flow on $\ol{\mc{W}}$.

By Proposition~\ref{prop:Jacobian-eigen}(a), the spectrum of $\tfrac{\partial F}{\partial S}(S^*)$, for $S^{\ast} \in \ol{\mc{W}}^{\ast}$, is given by the right-hand side of~\eqref{eq:spectrum-a} and, since $\Omega$ is nonnegative, is clearly negative if condition~\eqref{eq:stable} holds.
\item
We take eigenvectors into account and invoke Proposition~\ref{prop:unstable-cone}(b).
The eigenvectors are given by~\eqref{eq:eigv-S-int}, and if the eigenvalue $\lambda = (\Omega S^*)_{ij} - (\Omega S^*)_{i j^*(i)}$ is positive, then the corresponding eigenvector $V = e_i (e_{j^*(i)} - e_j)^\top \in \mathcal{T}_0$ is tangent to $\ol{\mathcal{W}}$.
By Proposition~\ref{prop:unstable-cone}(b), there exists an open truncated cone $\mathcal{C} \subset \R^{m\times n}$ such that $\delta \cdot V \in \mathcal{C}$, for sufficiently small $\delta > 0$, and the $S$-flow~\eqref{eq:prop_sflow} is repelled from $S^*$ within $S^* + \mathcal{C}$.
Since $V \in \mathcal{T}_0$, the (relatively) open subset $(S^* + \mathcal{C}) \cap \ol{\mathcal{W}} \subset \ol{\mathcal{W}}$ is non-empty. This shows the instability of $S^*$.
\item
By the assumption on $\Omega$, there is an eigenvalue with positive real part due to Proposition~\ref{prop:Jacobian-eigen}(c), 
and the real and imaginary part of the corresponding eigenvector lie in $\mathcal{T}_{+} \subseteq \mathcal{T}_0$.
So the argument of (b) applies here as well using the real part of the eigenvector.
\end{enumerate} \qed
\end{proof}
\begin{remark}[selection of stable equilibria]\label{rem:stability}
For $S^{\ast}$ to be exponentially stable, Corollary~\ref{cor:stability-S}(a) requires that every averaged subvector $(\Omega S^{\ast})_{i}$ has the same component as maximal component, as does the corresponding subvector $S_{i}^{\ast}$. This means that the $\Omega$-weighted average of the vectors $S_{j}^{\ast}$ within the neighborhood $j \in \mc{N}_{i}$ lies in the Voronoi-cell $\mc{V}_{j^{\ast}(i)}$~\eqref{eq:def-Voronoi-cell} corresponding to $S_{i}^{\ast}$.

Thus, Corollary~\ref{cor:stability-S} provides a mathematical and intuitively plausible definition of `spatially coherent' segmentations of given data, that can be determined by means of the assignment flow. This also demonstrates how the label (index) selection mechanism of the replicator equations~\eqref{eq:dot-Wi-explicit}, whose spatial \textit{coupling} defines the assignment flow~\eqref{eq:AF}, works from the point of view of evolutionary dynamics~\cite{Sandholm:2010aa} when using the similarity vectors $S_{i}(W)$~\eqref{eq:def-Si} as `affinity measures'.
\end{remark}

\subsubsection{\texorpdfstring{Convergence of the $S$-flow to Equilibria}{Convergence of the S-flow to Equilibria}}

We make the basic \textbf{assumption} that the parameter matrix $\Omega$ has the form
\begin{equation}\label{eq:sym-Omega}
\Omega = \Diag(w)^{-1} \widehat{\Omega} \quad \text{with} \quad w \in \R_{>0}^{m} \qquad \text{and} \qquad \widehat{\Omega}^\top = \widehat{\Omega} \in \R^{m \times m}.
\end{equation}
Matrices of the form~\eqref{eq:sym-Omega} include as special cases parameters satisfying
\begin{subequations}
\begin{align}
\Omega &= \Omega^{\T},
&&(\text{symmetric weights})
\label{eq:Omega-symmetric} \\ \label{eq:Omega-normalized}
w &= \widehat\Omega \eins_{m}.
&&(\text{normalized weights})
\end{align}
\end{subequations}
An instance of $\Omega$ satisfying~\eqref{eq:Omega-normalized} are nonnegative uniform weights with symmetric neighborhoods, i.e.
\begin{equation}\label{eq:uniform-weights}
\omega_{ik} = \tfrac{1}{| \mathcal{N}_i |},\quad
\forall k \in \mc{N}_{i}
\qquad\qquad\text{and}\qquad\qquad
k \in \mathcal{N}_i \quad\gdw\quad
i \in \mathcal{N}_k.
\end{equation}
Note that in the following basic convergence theorem, neither $\Omega$ nor $\widehat{\Omega}$ is assumed to be row-stochastic or nonnegative.
\begin{theorem}[convergence to equilibria]\label{thm:convergence}
Let $\Omega$ be of the form~\eqref{eq:sym-Omega}. Then the $S$-flow~\eqref{eq:def-S-flow-F} converges to an equilibrium point $S^{\ast} = S^{\ast}(S_{0}) \in \overline{\mathcal{W}}$, for any initial value $S_0 \in \mathcal{W}$.
\end{theorem}
\begin{proof}
See Appendix~\ref{sec:Proofs}.
\end{proof}
\begin{proposition} \label{prop:null-set}
Let $\Omega$ be nonnegative with positive diagonal entries, and let ${S^* \in \ol{\mathcal{W}}}$ be an equilibrium point of the $S$-flow~\eqref{eq:def-S-flow-F} which satisfies one of the instability criteria of Corollary~\ref{cor:stability-S} (b) or (c). Then the set of starting points $S_0 \in \mathcal{W}$ for which the $S$-flow converges to $S^*$ has measure zero in $\mathcal{W}$.
\end{proposition}
\begin{proof}
By~\cite{kelley1966stable}, there exists a center-stable manifold $\mathcal{M}_{\text{cs}}(S^*)$ which is invariant under the $S$-flow and tangent to $E_{\text{c}} \oplus E_{\text{s}}$ at $S^*$. Here, $E_{\text{c}}$ and $E_{\text{s}}$ denote the center and stable subspace of $\tfrac{\partial F}{\partial S}(S^*)$, respectively.
Any trajectory of the $S$-flow converging to $S^*$ lies in $\mathcal{M}_{\text{cs}}(S^*)$.
Therefore, it suffices to show that the dimension of the manifold $\mathcal{M}_{\text{cs}}(S^*) \cap \mathcal{W}$ is smaller than the dimension of $\mathcal{W}$. Note that $\mathcal{M}_{\text{cs}}(S^*) \cap \mathcal{W}$ is a manifold since both $\mathcal{M}_{\text{cs}}(S^*)$ and $\mathcal{W}$ are invariant under the $S$-flow.
We have
\begin{align}
\begin{split}
\dim\big( \mathcal{M}_{\text{cs}}(S^*) \cap \mathcal{W} \big) 
	&= \dim\big( (E_{\text{c}} \oplus E_{\text{s}}) \cap \mathcal{T}_0 \big) = \dim(\mathcal{T}_0) - \dim( E_{\text{u}} \cap \mathcal{T}_0 ) \\
	&= \dim(\mathcal{W}) - \dim( E_{\text{u}} \cap \mathcal{T}_0 ),
\end{split}
\end{align}
where $E_{\text{u}}$ denotes the unstable subspace of $\tfrac{\partial F}{\partial S}(S^*)$.
Since $\tfrac{\partial F}{\partial S}(S^*)$ has an eigenvalue with positive real part and a corresponding eigenvector lying in $\mathcal{T}_0$ (cf.~proof of Corollary~\ref{cor:stability-S}), we have $\dim( {E_{\text{u}} \cap \mathcal{T}_0} ) \geq 1$ and therefore
$\dim\!\big( \mathcal{M}_{\text{cs}}(S^*) \cap \mathcal{W} \big) \leq \dim(\mathcal{W}) - 1$. \qed
\end{proof}
\begin{remark}[consequences for the assignment flow] \label{remark:exp-convergence}
If $S^* \in \ol{\mc{W}}$ is a \emph{hyperbolic} equilibrium point, then the $S$-flow locally behaves as its linearization near $S^*$ by the Hartman-Grobman theorem~\cite[Section 2.8]{perko2013differential}.
Since a linear flow can only converge with an exponential convergence rate, this is also the case for the $S$-flow~\eqref{eq:def-S-flow-F}\footnote{Note that this follows by the H\"older continuity of the homeomorphism in the Hartman-Grobman theorem. The H\"older continuity is shown in~\cite{belitskii2009grobman}.}.
More precisely, if the $S$-flow converges to a hyperbolic equilibrium $S^* \in \ol{\mathcal{W}}$ then there exist $\alpha, \beta > 0$ such that
$\| S(t) - S^* \| \leq \alpha e^{-\beta t}$ irrespective of whether $S^*$ is stable or not.
A direct consequence is $\int_0^\infty \| S_i(t) - S_i^* \|_{1} \D t < \infty$ for all $i \in I$, i.e., assumption of Proposition~\ref{prop:limit-W}(b) automatically holds if $S^*$ is hyperbolic.
\end{remark}
\begin{theorem}\label{thm:measure-0}
Let $\Omega$ be a nonnegative matrix with positive diagonal entries.
Then the set of starting points $S_0 \in \mathcal{W}$ for which the $S$-flow~\eqref{eq:def-S-flow-F} converges to a nonintegral equilibrium $S^* \in \ol{\mathcal{W}}$, has measure zero in $\mathcal{W}$.
\end{theorem}
\begin{proof}
Let $\mathcal{E} = \{ S^* \in \ol{\mc{W}} \colon F(S^*) = 0 \}$ denote the set of all equilibria of the \mbox{$S$-flow} in $\ol{\mathcal{W}}$, which is a compact subset of $\ol{\mathcal{W}}$.
If $\mathcal{E}$ contains only isolated points, i.e., $\mathcal{E}$ is finite, then the statement follows from Proposition~\ref{prop:null-set}. In order to take also into account nonfinite sets $\mc{E}$ of equilibria, we apply the more general~\cite[Theorem~9.1]{fenichel1979geometric}. Some additional notation is introduced first.

For any index set $\mc{J} \subseteq I \times J$, set
\begin{equation}
\mathcal{E}_{\mathcal{J}} = \big\{ S^* \in \mathcal{E} \colon \supp(S^*) = \mathcal{J} \big\} \subset \mathcal{E}.
\end{equation}
The set $\mathcal{E}_{\mathcal{J}}$ is the relative interior of a convex polytope and therefore a manifold of equilibria.
This follows from the observation that the equilibrium criterion~\eqref{eq:equilibrium} is a set of linear equality constraints for $S^* \in \ol{\mc{W}}$, given by
\begin{equation}
\left.
\begin{aligned}
(\Omega S^*)_{ij} - (\Omega S^*)_{il} &=0 &\quad \forall j,l &\in \supp(S_i^*) \\
S_{ij}^* &= 0&\quad \forall j &\in J \setminus \supp(S_i^*)
\end{aligned}
\quad\right\}, \qquad \forall i \in I.
\end{equation}
Further, define for $n_{\text{s}},n_{\text{c}},n_{\text{u}} \in \N\cup\{0\}$ with $n_{\text{s}} + n_{\text{c}} + n_{\text{u}} = m n$ the set
\begin{equation}
\mathcal{E}_{(n_{\text{s}},n_{\text{c}},n_{\text{u}})} = \big\{ S^* \in \mathcal{E} \colon \dim E_{\text{s}}(S^*) = n_{\text{s}},\; \dim E_{\text{c}}(S^*) = n_{\text{c}},\; \dim E_{\text{u}}(S^*) = n_{\text{u}} \big\},
\end{equation}
where $E_{\text{c}}(S^*)$, $E_{\text{s}}(S^*)$ and $E_{\text{u}}(S^*)$ denote the center, stable and unstable subspace of $\tfrac{\partial F}{\partial S}(S^*)$.
This set can be written as countable union of compact sets. This can be seen as follows. The map
\begin{equation}\label{eq:proof-nullset-eigenmap}
\mathcal{E} \rightarrow \big\{ x \in \R^{m n} \colon x_1 \leq x_2 \leq \dots \leq x_{m n} \big\}, \qquad S^* \mapsto \Re \Big( \lambda \big( \tfrac{\partial F}{\partial S}(S^*) \big) \Big),
\end{equation}
where $\lambda(\cdot)$ denotes the vector of eigenvalues, is a continuous map on a compact set and therefore proper, i.e., preimages of compact sets under the map \eqref{eq:proof-nullset-eigenmap} are compact.
It is clear that the set $U_{\text{s}} \times U_{\text{c}} \times U_{\text{u}}$ with
\begin{subequations}
\begin{align}
U_{\text{s}} &= \big\{ x \in \R^{n_{\text{s}}} \colon x_1 \leq \dots \leq x_{n_s} < 0 \big\}, \\
U_{\text{c}} &= \big\{ x \in \R^{n_{\text{c}}} \colon x = 0 \big\}, \\
U_{\text{u}} &= \big\{ x \in \R^{n_{\text{u}}} \colon 0 < x_1 \leq \dots \leq x_{n_u} \big\}
\end{align}
\end{subequations}
can be written as countable union of compact sets. The preimage of this set under the map~\eqref{eq:proof-nullset-eigenmap} is $\mathcal{E}_{(n_{\text{s}},n_{\text{c}},n_{\text{u}})}$.

To complete the proof, we now argue similar to the proof of Proposition~\ref{prop:null-set}: the existence of nontrivial unstable subspaces for nonintegral equilibria implies that the center-stable manifold has a smaller dimension.

Let $\mathcal{J}$ be the support of any nonintegral equilibrium and let $\mathcal{E}_{(n_{\text{s}},n_{\text{c}},n_{\text{u}})}$ be such that $\mathcal{E}_{\mathcal{J}} \cap \mathcal{E}_{(n_{\text{s}},n_{\text{c}},n_{\text{u}})} \neq \emptyset$.
As seen in the proof of Corollary~\ref{cor:stability-S}(c), we have $E_{\text{u}}(S^*) \cap \mc{T}_0 \neq \{0\}$ for any $S^* \in \mathcal{E}_{\mathcal{J}}$, i.e.\ $n_{\text{u}} \geq 1$. Since both $\mathcal{E}_{\mathcal{J}}$ and $\mathcal{E}_{(n_{\text{s}},n_{\text{c}},n_{\text{u}})}$ can be written as countable union of compact sets, this is also the case for their intersection, i.e., we have
\begin{equation}\label{eq:proof-cover-Kl}
\mathcal{E}_{\mathcal{J}} \cap \mathcal{E}_{(n_{\text{s}},n_{\text{c}},n_{\text{u}})} = \bigcup_{l \in \N} K_l
\end{equation}
with $K_l \subseteq \mathcal{E}_{\mathcal{J}}$ compact.
For any $l \in \N$, there exists a center-stable manifold $\mathcal{M}_{\text{cs}}(K_l)$ containing $K_l$, which is invariant under the $S$-flow and tangent to $E_{\text{c}}(S^*) \oplus E_{\text{s}}(S^*)$ at any $S^* \in K_l$~\cite[Theorem~9.1]{fenichel1979geometric}.
Any trajectory of the $S$-flow converging to a point $S^* \in K_l$ lies in $\mathcal{M}_{\text{cs}}(K_l)$.
Hence, analogous to the proof of Proposition~\ref{prop:null-set}, we have
\begin{equation}
\dim\big( \mathcal{M}_{\text{cs}}(K_l) \cap \mc{W} \big) = \dim(\mc{W}) - \dim(E_{\text{u}}(S^*) \cap \mc{T}_0) \leq \dim(\mc{W}) - 1,
\end{equation}
with any $S^* \in K_l$, i.e., $\mathcal{M}_{\text{cs}}(K_l) \cap \mc{W}$ has measure zero in $\mc{W}$.
The countable union $\bigcup_{l \in \N} \mathcal{M}_{\text{cs}}(K_l) \cap \mc{W}$, which contains all trajectories converging to an equilibrium $S^* \in \mathcal{E}_{\mathcal{J}} \cap \mathcal{E}_{(n_{\text{s}},n_{\text{c}},n_{\text{u}})}$, has measure zero as well. Since there are only finitely many such sets $\mathcal{E}_{\mathcal{J}} \cap \mathcal{E}_{(n_{\text{s}},n_{\text{c}},n_{\text{u}})}$, this completes the proof. \qed
\end{proof}

In view of Theorem \ref{thm:measure-0}, the following Corollary that additionally takes into account assumption \eqref{eq:sym-Omega}, is obvious.
\begin{corollary}[Convergence to integral assignments]
Let $\Omega$ be a nonnegative matrix with positive diagonal entries which also fulfills the symmetry assumption~\eqref{eq:sym-Omega}. Then the set of starting points $S_0 \in \mathcal{W}$, for which the $S$-flow~\eqref{eq:def-S-flow-F} does not converge to an integral assignment $S^* \in \Wstar$, has measure zero.
If $\Omega$ is additionally invertible, then the set of distance matrices $D \in \R^{m \times n}$ for which the $S$-flow does not converge to an integral assignment has measure zero as well.
\end{corollary}

\subsubsection{Basins of Attraction}
Corollary~\ref{cor:stability-S} says that, if a point $S^* \in \Wstar$ satisfies the stability criterion~\eqref{eq:stable}, then there exists an open neighborhood of $S^*$ such that the $S$-flow emanating from this neighborhood will converge to $S^*$ with an exponential convergence rate.
The subsequent proposition quantifies this statement by describing the convergence in balls around the equilibria which are contained in the corresponding basin of attraction.
\begin{proposition} \label{prop:attraction}
Let $\Omega$ be a nonnegative matrix with positive diagonal entries, and
let $S^* \in \Wstar$ satisfy~\eqref{eq:stable}. Furthermore, set
\begin{multline}\label{eq:attraction-AS*}
A(S^*) \coloneqq \bigcap_{i \in I} \bigcap_{j \neq j^*(i)} \big\{ S \in \R^{m \times n} \colon (\Omega S)_{ij} < (\Omega S)_{i j^*(i)} \big\} \\ \text{with} \quad \{ j^*(i) \} = \argmax_{j \in J}~S_{ij}^*,
\end{multline}
which is an open convex polytope containing $S^*$. Finally, let $\varepsilon > 0$ be small enough such that
\begin{equation}\label{eq:attraction-B-veps}
B_{\varepsilon}(S^*) \coloneqq \big\{ S \in \ol{\mathcal{W}}  \colon \max_{i \in I} \| S_{i} - S_{i}^* \|_{1} < \varepsilon \big\} \subset \big(A(S^*) \cap \ol{\mathcal{W}}\big).
\end{equation}
Then, regarding the $S$-flow~\eqref{eq:def-S-flow-F}, the following holds:
If $S(t_0) \in B_{\varepsilon}(S^*)$ for some point in time $t_0$, then $S(t) \in B_{\varepsilon}(S^*)$ for all $t \geq t_0$ and $\lim_{t \rightarrow \infty} S(t) = S^*$.
Moreover, we have
\begin{subequations}\label{eq:atttaction-rate}
\begin{align}\label{eq:attraction-Gronwall}
\| S_i(t) - S_i^* \|_{1} \leq \| S_i(t_0) - S_i^* \|_{1} \cdot e^{-\beta_i (t - t_0)}, \quad \forall i \in I,
\intertext{where}\label{eq:beta-i-attraction}
\beta_i = \min_{S \in \ol{B_{\delta}(S^*)}
 \cap \ol{\mathcal{W}}}~S_{i j^*(i)} \cdot \min_{j \neq j^*(i)}~\big( (\Omega S)_{i j^*(i)} - (\Omega S)_{ij} \big) > 0
\end{align}
\end{subequations}
and $\delta > 0$ is chosen small enough such that $S(t_{0}) \in \ol{B_{\delta}(S^{\ast})} \subset B_{\veps}(S^{\ast})$.
\end{proposition}
\begin{proof}
For each $i \in I$, we have with $S_{i}^{\ast}=e_{j^{\ast}(i)}$
\begin{subequations}
\begin{align}
\begin{split}
\frac{\mathrm{d}}{\mathrm{d}t}& \| S_{i} - S_{i}^* \|_{1} \\
	&\stackrel{\hphantom{\eqref{eq:Si-diff-ell-1}}}{= }\frac{\mathrm{d}}{\mathrm{d}t} \Big( 1 - S_{i j^*(i)} + \sum_{j \neq j^*(i)} S_{i j} \Big)  \qquad(\text{using}\; \sum_{j \in [n]} S_{ij}=1)
\end{split} \\ \label{eq:Si-diff-ell-1}
	&\stackrel{\hphantom{\eqref{eq:Si-diff-ell-1}}}{=} \frac{\mathrm{d}}{\mathrm{d}t}(2 - 2 S_{i j^*(i)})
	\\
	&\stackrel{\substack{\hphantom{\eqref{eq:Si-diff-ell-1}} \\ \eqref{eq:def-S-flow-F}}}{=} - 2 S_{i j^*(i)} \big( (\Omega S)_{i j^*(i)} - \langle S_i, (\Omega S)_i \rangle \big) \\
	&\stackrel{\hphantom{\eqref{eq:Si-diff-ell-1}}}{\leq} -2 S_{i j^*(i)} \Big( (\Omega S)_{i j^*(i)} - S_{i j^*(i)} (\Omega S)_{i j^*(i)} - \max_{j \neq j^*(i)}~(\Omega S)_{ij} \sum_{j \neq j^*(i)} S_{ij} \Big) \\
	&\stackrel{\hphantom{\eqref{eq:Si-diff-ell-1}}}{=} -2 S_{i j^*(i)} (1 - S_{i j^*(i)}) \Big( (\Omega S)_{i j^*(i)} - \max_{j \neq j^*(i)}~(\Omega S)_{ij} \Big) \\
	&\stackrel{\eqref{eq:Si-diff-ell-1}}{=}
	- S_{i j^*(i)} \| S_{i} - S_{i}^* \|_{1} \cdot \min_{j \neq j^*(i)}~\big( (\Omega S)_{i j^*(i)} - (\Omega S)_{ij} \big).
\end{align}
\end{subequations}
Choosing $\delta > 0$ such that $S(t_0) \in \ol{B_{\delta}(S^*)} \subset B_{\varepsilon}(S^*)$, it follows that $\beta_{i}$ given by~\eqref{eq:beta-i-attraction} is positive. Consequently
\begin{equation}
\frac{\mathrm{d}}{\mathrm{d}t} \| S_{i} - S_{i}^* \|_{1} \leq -\beta_i \| S_{i} - S_{i}^* \|_{1}
\end{equation}
and by Gronwall's Lemma~\eqref{eq:attraction-Gronwall} holds.  Hence, $\max_{i \in I} \| S_{i} - S_{i}^* \|_{1}$ monotonically decreases as long as $S(t) \in \overline{B_{\delta}(S^*)}$.
This guarantees that $S(t)$ stays in $\overline{B_{\delta}(S^*)} \subset B_{\varepsilon}(S^*)$ and converges toward $S^*$. \qed
\end{proof}
Note that if $S(t)$ is close to $S^*$, then the convergence rate~\eqref{eq:atttaction-rate} of $S(t)$ is approximately governed by
\begin{equation} \label{eq:convergence-rate}
\beta_i \approx \min_{j \neq j^*(i)}~\big( (\Omega S^*)_{i j^*(i)} - (\Omega S^*)_{ij} \big).
\end{equation}

Proposition~\ref{prop:attraction} provides a criterion for terminating the numerical integration of the $S$-flow and subsequent `safe' rounding to an integral solution. For this purpose, the following proposition provides an estimate of $\veps$ defining~\eqref{eq:attraction-B-veps}.
\begin{proposition} \label{prop:eps_est}
Let $S^* \in \Wstar$ satisfy~\eqref{eq:stable}. A value $\varepsilon > 0$ that is sufficient small for the inclusion~\eqref{eq:attraction-B-veps} to hold, is given by
\begin{equation} \label{eq:eps_est}
\varepsilon_{\mathrm{est}} = \min_{i \in I}~\min_{j \neq j^*(i)}~2 \cdot \frac{(\Omega S^*)_{i j^*(i)} - (\Omega S^*)_{ij}}{(\Omega \eins_{m} )_{i} + (\Omega S^*)_{i j^*(i)} - (\Omega S^*)_{ij}} > 0.
\end{equation}
\end{proposition}
\begin{proof}
Let $S \in \ol{\mathcal{W}}$ be a point such that
\begin{equation}
\max_{i \in I} \| S_i - S_i^* \|_1 < \varepsilon = \varepsilon_{\mathrm{est}}.
\end{equation}
We have to show that $S \in A(S^{\ast})$, with $A(S^{\ast})$ given by~\eqref{eq:attraction-AS*}.

Since $\| S_i - S_i^* \|_1 = 2 - 2 S_{i j^*(i)}$, we have
\begin{subequations}
\begin{align}
S_{i j^*(i)} &> 1 - \frac{\varepsilon}{2}, \qquad\qquad
S_{i j} \leq \sum_{l \neq j^*(i)} S_{il} = 1 - S_{i j^*(i)} < \frac{\varepsilon}{2}, \quad \forall j \neq j^*(i).
\end{align}
\end{subequations}
Hence, for any $i \in I$ and any $j \neq j^*(i)$, we get with $j^{\ast}(k),\,k \in I$ similarly defined as $j^{\ast}(i)$ in~\eqref{eq:stable}, 
\begin{subequations}
\begin{align}
(\Omega S&)_{i j^*(i)} - (\Omega S)_{ij}
\overset{\eqref{eq:def-Omega}}{=}
\sum_{k \in\mc{N}_{i}} \omega_{i k} S_{k j^*(i)} - \sum_{k  \in \mc{N}_{i}} \omega_{i k} S_{kj}
\\
	&= \sum_{\substack{k \in \mc{N}_{i} \\ \mathclap{j^*(k) = j^*(i)}}} \omega_{i k} \overbrace{ S_{k j^*(i)} }^{> 1 - \tfrac{\varepsilon}{2}}
	      + \sum_{\substack{k \in \mc{N}_{i} \\ \mathclap{j^*(k) \neq j^*(i)}}} \omega_{i k} \overbrace{ S_{k j^*(i)} }^{\geq 0}
	- \sum_{\substack{k  \in \mc{N}_{i} \\ \mathclap{j^*(k) = j}}} \omega_{i k} \overbrace{ S_{kj} }^{\leq 1}
	      - \sum_{\substack{k  \in \mc{N}_{i} \\ \mathclap{j^*(k) \neq j}}} \omega_{i k} \overbrace{ S_{kj} }^{< \tfrac{\varepsilon}{2}},
 \intertext{and by dropping the second nonnegative summand,}
	&> \Big(1 - \frac{\varepsilon}{2} \Big) \sum_{\substack{k \in \mc{N}_{i} \\ \mathclap{j^*(k) = j^*(i)}}} \omega_{i k} - \sum_{\substack{k  \in \mc{N}_{i} \\ \mathclap{j^*(k) = j}}} \omega_{i k} - \frac{\varepsilon}{2} \sum_{\substack{k  \in \mc{N}_{i} \\ \mathclap{j^*(k) \neq j}}} \omega_{i k}
\intertext{and using the subvectors of $S^{\ast}$ are unit vectors,}
	&= \Big(1 - \frac{\varepsilon}{2} \Big) (\Omega S^*)_{i j^*(i)} - (\Omega S^*)_{ij} - \frac{\varepsilon}{2} \big( ( \Omega \eins_{m} )_{i} - (\Omega S^*)_{ij} \big) \\
	&= (\Omega S^*)_{i j^*(i)} - (\Omega S^*)_{ij} - \frac{\varepsilon}{2} \Big( ( \Omega \eins_{m} )_{i} + (\Omega S^*)_{i j^*(i)} - (\Omega S^*)_{ij} \Big) \\
	&\overset{\mathclap{\eqref{eq:eps_est}}}{\geq}\; 0.
\end{align}
\end{subequations}
This verifies $S \in A(S^*)$. \qed
\end{proof}

Figure~\ref{fig:attraction-region} illustrates the sets $A(S^*)$ and $B_{\varepsilon}(S^*)$ defined by~\eqref{eq:attraction-AS*} and~\eqref{eq:attraction-B-veps}, for some examples in the simple case of two data points and two labels.
The beige and green regions in the left panel illustrate that the condition $S(t_0) \in A(S^*)$ neither guarantees that the $S$-flow converges to $S^*$ nor to stay in $A(S^*)$.
This demonstrates the need for the sets $B_{\varepsilon}(S^*)$, shown as shaded squares in Figure~\ref{fig:attraction-region}. Note that $B_{\varepsilon}(S^*) \neq \emptyset$ only if $S^* \in A(S^*) \neq\emptyset$, i.e.\ if the stability condition~\eqref{eq:stable} is fulfilled.

\begin{figure}[htb]
\begin{center}
\begin{tabular}{ccc}
\quad $\Omega = \begin{pmatrix} 0.55 & 0.45 \\ 0.25 & 0.75 \end{pmatrix}$ &
\quad $\Omega = \begin{pmatrix} 0.5 & 0.5 \\ 0.5 & 0.5 \end{pmatrix}$ &
\quad $\Omega = \begin{pmatrix} 0.25 & 0.75 \\ 0.75 & 0.25 \end{pmatrix}$ \\[2ex]
\includegraphics[width=0.29\textwidth]{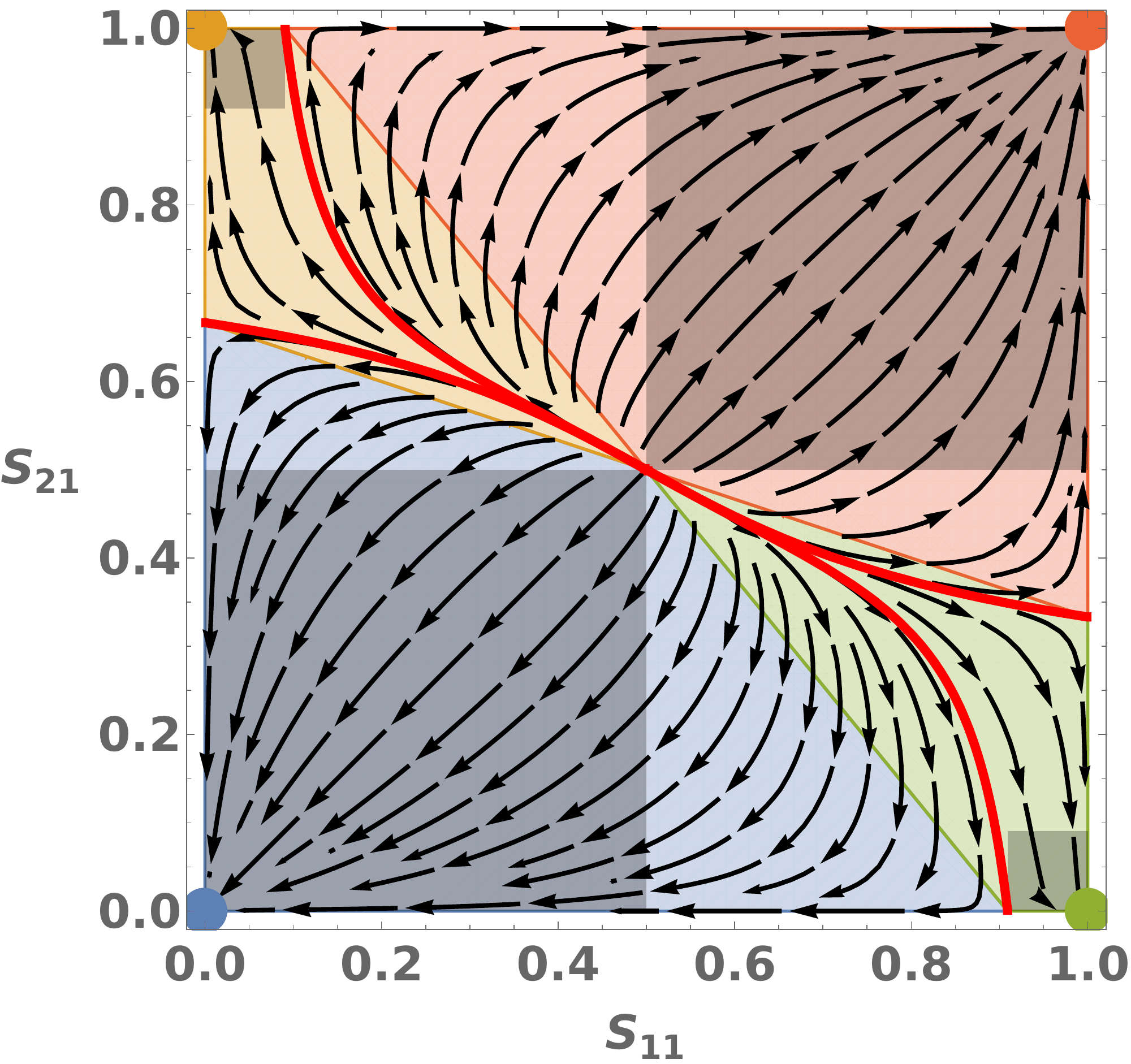} &
\includegraphics[width=0.29\textwidth]{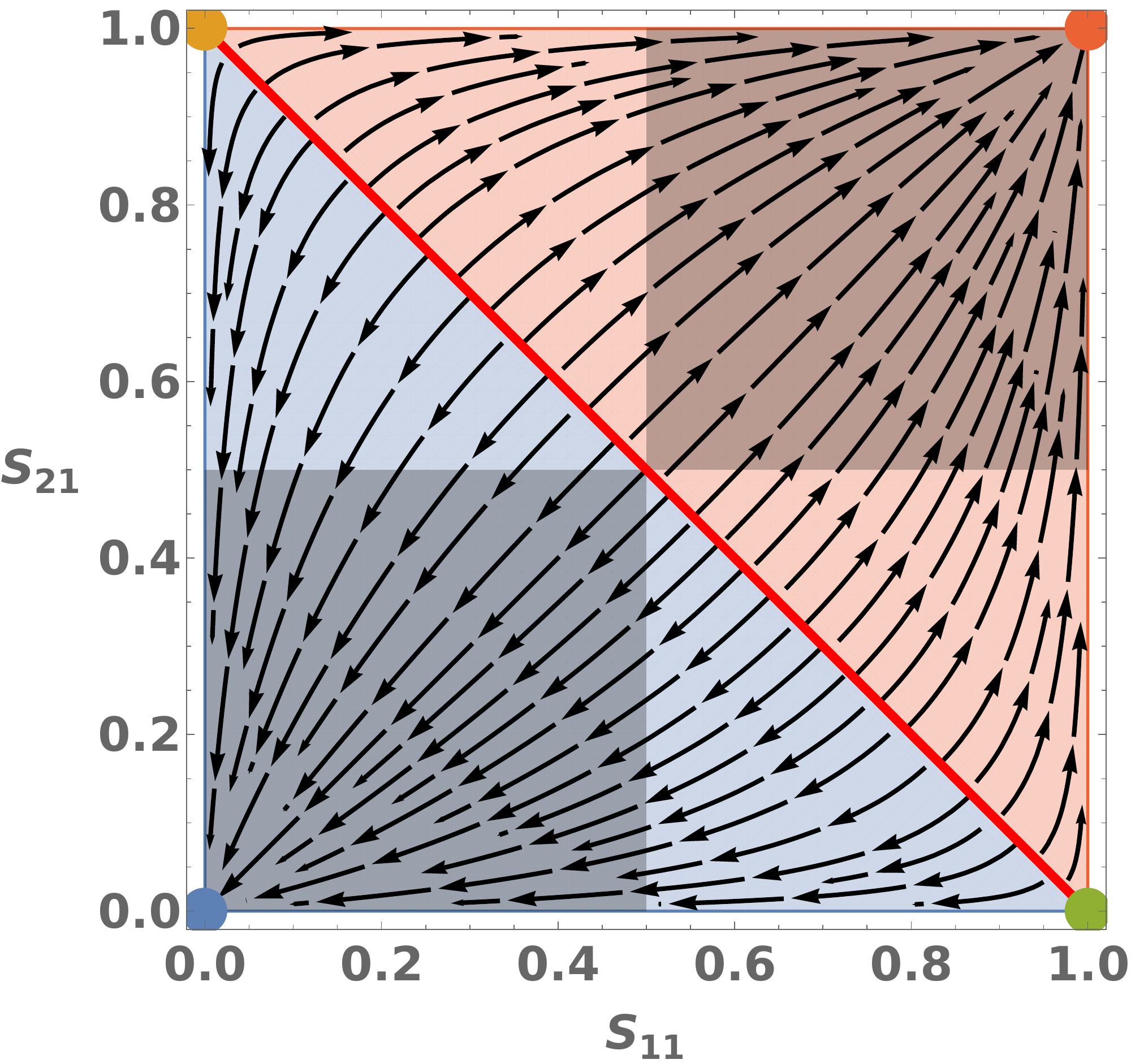} &
\includegraphics[width=0.29\textwidth]{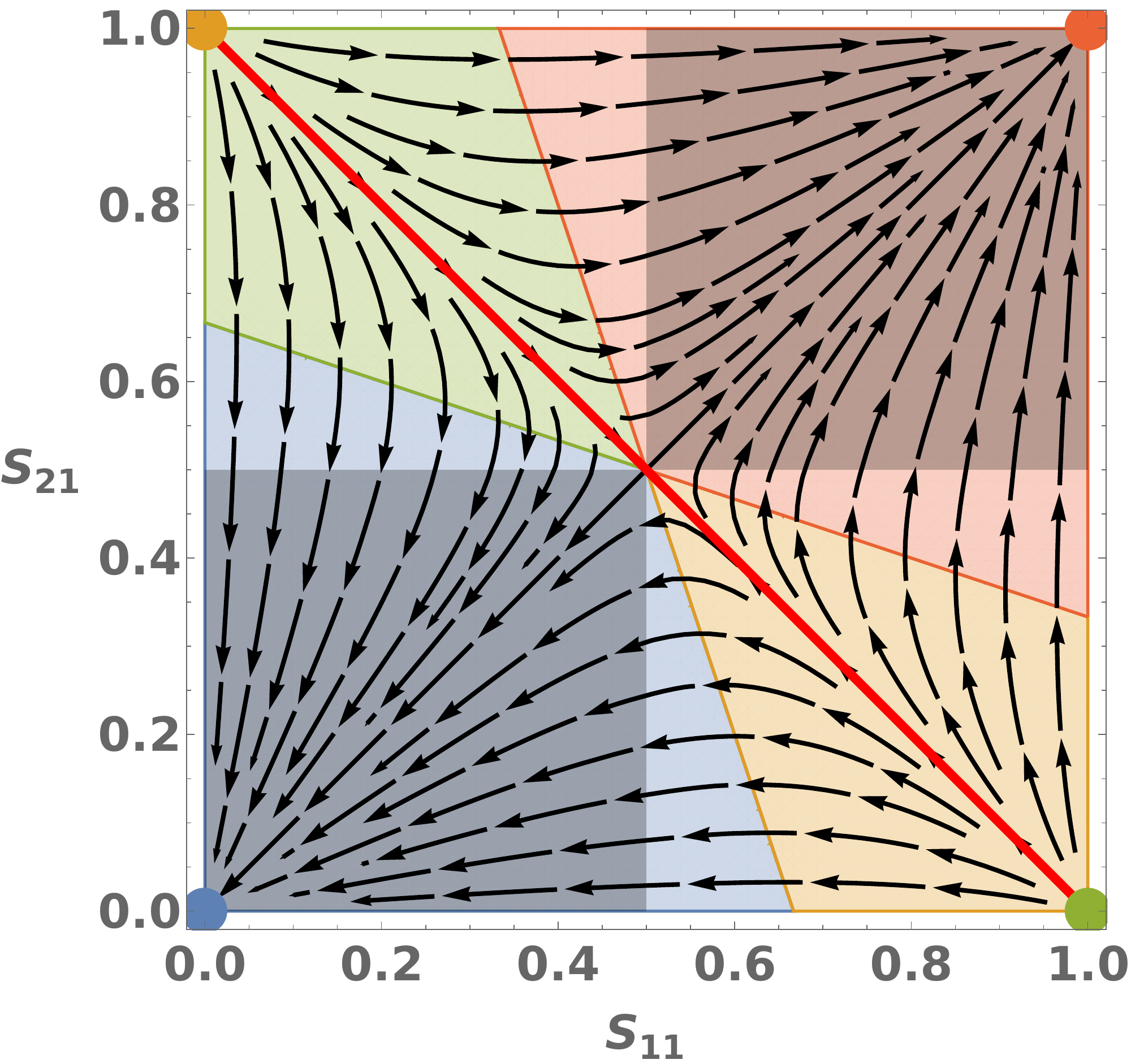}
\end{tabular}
\end{center}
\caption[Illustration of the approximation of the basins of attraction for the case \boldmath$|I|=|J|=2$.]{
\textbf{Illustration of the approximation of the basins of attraction for the case \boldmath$|I|=|J|=2$.}
The plots show the phase portrait of the $S$-flow~\eqref{eq:def-S-flow-F} for three different row-stochastic matrices $\Omega$. The four points $S^* \in \Wstar$ marked with
$\{
\tikz[color={rgb,255:red,94; green,129; blue,181},baseline=-0.625ex]{\fill (0,0) circle(0.75ex);},
\tikz[color={rgb,255:red,225; green,156; blue,36},baseline=-0.625ex]{\fill (0,0) circle(0.75ex);},
\tikz[color={rgb,255:red,142; green,176; blue,50},baseline=-0.625ex]{\fill (0,0) circle(0.75ex);},
\tikz[color={rgb,255:red,235; green,98; blue,53},baseline=-0.625ex]{\fill (0,0) circle(0.75ex);}
\}$,
the corresponding sets $A(S^*)$~\eqref{eq:attraction-AS*} are shown as colored regions, and the balls $B_{\varepsilon}(S^*)$~\eqref{eq:attraction-B-veps} around the equilibria for which convergence to the equilibria is guaranteed are shown as shaded squares, with $\varepsilon = \varepsilon_{\text{est}}(S^*, \Omega)$ from~\eqref{eq:eps_est}.
Finally, the boundary between the basins of attraction is marked with a thick red curve.
In the center and right panel, only the constant labelings $S^{\ast} \in \{\begin{psmallmatrix}0&1\\0&1\end{psmallmatrix},\begin{psmallmatrix}1&0\\1&0\end{psmallmatrix}\}
$ fulfill the stability criterion~\eqref{eq:stable}, i.e.\ $S^* \in A(S^*)$.
As for the other two points $S^* \in \Wstar$, we have either $A(S^*) = \emptyset$ (center panel) or $S^* \not\in A(S^*) \neq \emptyset$ (right panel).} \label{fig:attraction-region}
\end{figure}

If \textit{uniform} weights $\Omega$ are used for averaging, then the estimate~\eqref{eq:eps_est} can be cast into a simple form that no longer depends on  $S^*$.
\begin{corollary} \label{prop:eps_unif}
Let $\Omega$ defined by~\eqref{eq:def-Omega} be given by uniform weights $\omega_{ik} = \tfrac{1}{|\mathcal{N}_i |}$, ${k \in \mc{N}_{i}}$, ${i \in I}$. 
Then the value $\veps > 0$ that achieves the inclusion~\eqref{eq:attraction-B-veps} can be chosen as
\begin{equation} \label{eq:eps_unif}
\varepsilon_{\mathrm{unif}} = \frac{2}{1 + \max_{i \in I} |\mathcal{N}_i|} > 0.
\end{equation}
\end{corollary}
\begin{proof}
Let $j^{\ast}(i)$ be defined as in~\eqref{eq:stable}.
We have

\begin{subequations}
\begin{align}
(\Omega S^*)_{i j^*(i)} - (\Omega S^*)_{ij} 
	&= \frac{| \{ k \in \mathcal{N}_i \colon j^*(k) = j^*(i) \} | - | \{ k \in \mathcal{N}_i \colon j^*(k) = j \} |}{| \mathcal{N}_i |} \label{eq:eps-unif-int}\\
	&\geq \frac{1}{| \mathcal{N}_i |} > 0
\end{align}
\end{subequations}
by assumption and integrality of the numerator in \eqref{eq:eps-unif-int}.
Monotonicity of the function $x \mapsto \frac{x}{1+x}$ implies
\begin{equation}
2 \cdot \frac{(\Omega S^*)_{i j^*(i)} - (\Omega S^*)_{ij}}{1 + (\Omega S^*)_{i j^*(i)} - (\Omega S^*)_{ij}} \geq 2 \cdot \frac{\frac{1}{| \mathcal{N}_i |}}{1 + \frac{1}{| \mathcal{N}_i |}} = \frac{2}{1 + | \mathcal{N}_i |}
\end{equation}
and hence $\varepsilon_{\mathrm{unif}} \leq \varepsilon_{\mathrm{est}}$, with $\varepsilon_{\mathrm{est}}$ given by~\eqref{eq:eps_est}.
The assertion, therefore, follows from Proposition~\ref{prop:eps_est}. \qed
\end{proof}
%


\subsection{Convergence Properties of the Linear Assignment Flow}\label{ssec:ConvergencePropertiesLAF}

This section analyzes the convergence of the linear assignment flow to equilibria
and limit points.
To apply the standard theory, we rewrite the matrix-valued ($V \in \R^{m \times n}$)
equation of the linear assignment flow~\eqref{eq:LAF} into a vector-valued
($V \in \R^{m n}$) one, using again $V$, for simplicity.

Equation~\eqref{eq:LAF} then takes the form
\begin{subequations}\label{eq:LAF-2}
\begin{align}\label{eq:LAF-2-dot-V}
    \dot{V} &= A V + b, \quad V(0) = 0,
    \intertext{where} \label{eq:LAF-2-A}
    A &= R_{\widehat{S}}(\Omega \otimes I_n).
\end{align}
\end{subequations}
Note that matrix $A$ is exactly the second summand in the Jacobian~\eqref{eq:Jacobian-S-Flow} of the \mbox{$S$-flow}.
The first summand of~\eqref{eq:Jacobian-S-Flow} is due to the dependence of the replicator matrix on the flow. The linear assignment flow~\eqref{eq:LAF} ignores this dependency by assuming $\widehat S \in \mc{W}$ to be \textit{fixed}.

The following Lemma says that under the assumption $b \in \mathcal{R}(A)$, the asymptotic properties of~\eqref{eq:LAF-2-dot-V} can be inferred from the homogeneous system.
\begin{lemma}\label{lmm:additive-to-initial}
    Let $\Psi_{A,b,V_0}(t)$ denote the flow of the dynamical system~\eqref{eq:LAF-2}
    but with initial condition $V(0) = V_0$ and assume $b \in \mathcal{R}(A)$.
    Then the equation $\Psi_{A,b,V_0}(t) = \Psi_{A,0,V_0+A^+b}(t)-A^+b$ holds, where
    $A^+$ denotes the pseudoinverse of $A$.
\end{lemma}
\begin{proof}
    For $b \in \mathcal{R}(A)$ we have $A A^+ b = b$ and therefore with Duhamel's formula~\cite[p.~72]{Teschl:2012aa}
    \begin{subequations}
        \begin{alignat}{1}
            \Psi_{A,b,V_0}(t) &= e^{t A} V_0 + \int_0^t e^{(t-\tau) A} b\ \D \tau = e^{t A} V_0 + \int_0^t e^{(t-\tau) A} A\ \D \tau A^+ b \\
            &= e^{t A} V_0 + (e^{t A} - I_n) A^+ b = \Psi_{A,0,V_0+A^+b}(t)-A^+b.
        \end{alignat}
    \end{subequations} \qed
\end{proof}

As the translation of the flow by  $-A^+b$ does not change the convergence properties
(except for translating the equilibria), we can focus on the corresponding homogeneous system
\begin{alignat}{1}\label{eq:LAF-homogeneous}
    \dot{V} = A V,\quad V(0) = V_0.
\end{alignat}
Using the eigensystem of $A$, the solution to~\eqref{eq:LAF-homogeneous} can be represented in the following well-known way.
\begin{lemma}\label{lmm:LAF-trajectory}
    Let $A$ be a diagonalizable matrix with eigenvalues $\lambda_i$,
    corresponding eigenvectors $v_i$.
    Further let $V_0 = \sum_i c_i v_i$ with $c_i \in \R$.
    The solution of the linear dynamical system~\eqref{eq:LAF-homogeneous} can
    be written as
    \begin{alignat}{1}
        V(t) = \sum_i c_i e^{\lambda_i t} v_i.
    \end{alignat}
    Without loss of generality, let $\lambda_1$ be the dominant eigenvalue, i.e.\ the eigenvalue with
    maximal real part.
    If $\lambda_1$ is unique and $c_1 \neq 0$, then
    \begin{alignat}{1}
        \lim_{t \to \infty} V(t) = \lim_{t \to \infty} c_1 e^{\lambda_1 t} v_1.
    \end{alignat}
    The hyperplane of initial values with $c_1 = 0$ separates two half-spaces
    which are the regions of attraction for the limit points in the directions $v_{1}$
    and $-v_{1}$, respectively.
\end{lemma}

Lemma~\ref{lmm:LAF-trajectory} implies the following properties of the system~\eqref{eq:LAF-homogeneous}.

\begin{proposition}\label{prop:LAF-A-1}
Any linear dynamical system of the form~\eqref{eq:LAF-homogeneous} with diagonalizable $A$
has the following properties
    \begin{enumerate}[(a)]
        \item If $A$ has an eigenvalue with positive real part, then any finite equilibrium is unstable and the set of initial points converging to these equilibria is a null set.
        \item If all eigenvalues of $A$ are real, then the trajectory does not spiral
            around a subspace through the origin infinitely often, i.e.\ $0$ neither is a spiral sink nor a spiral source.
        \item The set of equilibria is the nullspace $\mc{N}(A)$.
        \item The stable (resp. unstable) manifold is spanned by the
            eigenvectors of $A$ corresponding to eigenvalues with
            negative (resp. positive) real part. All initial points which do not belong to the center-stable manifold
            diverge to infinity.
    \end{enumerate}
\end{proposition}
The following proposition complements Proposition~\ref{prop:LAF-A-1} by examining the spectrum of the matrix $A$.
\begin{proposition}\label{prop:LAF-A-2}
    Let $A = R_{\widehat{S}}(\Omega \otimes I_n)$ be the system matrix of the linear assignment flow~\eqref{eq:LAF-2-dot-V}. Then the following holds.
    \begin{enumerate}[(a)]
        \item If the diagonal of $\Omega$ is nonnegative and contains at least
            one positive element, the matrix $A$ has at least one eigenvalue
            with positive real part.
            This means that all finite equilibria are unstable.
        \item If $\Omega$ has the form~\eqref{eq:sym-Omega} (i.e.\ $\Omega$ is a
            row-wise positive scaling of a symmetric matrix),
            then $A$ has only real eigenvalues.
            As a consequence, any initial value converges either to a finite
            equilibrium or to a fixed limit point at infinity.
        \item If $\Omega$ is invertible, then $\rank(A) = m(n-1)$.
            Furthermore, $\mc{N}(A)$ is spanned by the vectors
            $\{e_i \otimes \eins_n \colon i \in I\}$ and the restriction
            $A|_{\mc{T}_{0}}$ is invertible.
            Thus, 0 is the only finite equilibrium.
        \item If $\Omega$ is invertible and positive definite, then the $m(n-1)$ nonzero
            eigenvalues of $A$ are positive as well.
            Consequently, the restriction $A|_{\mc{T}_{0}}$ is positive definite
            and any initial value, except for the origin, diverges to infinity.
    \end{enumerate}
\end{proposition}
\begin{proof}
    \begin{enumerate}[(a)]
        \item Because the trace of $A$ is positive---cf.~\eqref{eq:ProofAtLeastOnePositiveEigenvalue}---$A$
            must have at least one positive eigenvalue.
            The statement on the stability of the equilibria follows from Proposition~\ref{prop:LAF-A-1}(a).
        \item Using the notation $A \sim B$ for the similarity
            of the matrices $A$ and $B$, we have
            \begin{subequations}
                \begin{align}
                    A &= R_{\widehat{S}} (\Omega \otimes I_n)
                    \overset{\eqref{eq:sym-Omega}}{=} R_{\widehat{S}} (\Diag(w)^{-1}\widehat{\Omega} \otimes I_n) \\
                    &= R_{\widehat{S}} (\Diag(w)^{-1} \otimes I_n) (\widehat{\Omega} \otimes I_n) \\
                    &= R_{\widehat{S}} (\Diag(w) \otimes I_n)^{-1} (\widehat{\Omega} \otimes I_n) \\
                    \begin{split}
                    &\sim (\Diag(w) \otimes I_n)^{\frac{1}{2}} R_{\widehat{S}} (\Diag(w) \otimes I_n)^{-\frac{1}{2}} \\
                    &\qquad \cdot (\Diag(w) \otimes I_n)^{-\frac{1}{2}} (\widehat{\Omega} \otimes I_n) (\Diag(w) \otimes I_n)^{-\frac{1}{2}} 
                    \end{split} \\
                    &= R_{\widehat{S}} (\Diag(w) \otimes I_n)^{-\frac{1}{2}} (\widehat{\Omega} \otimes I_n) (\Diag(w) \otimes I_n)^{-\frac{1}{2}} \\
                    &\sim R_{\widehat{S}}^{\frac{1}{2}} (\Diag(w) \otimes I_n)^{-\frac{1}{2}} (\widehat{\Omega} \otimes I_n) (\Diag(w) \otimes I_n)^{-\frac{1}{2}} R_{\widehat{S}}^{\frac{1}{2}},
                \end{align}
            \end{subequations}
            where
            $R_{\hat{S}}^{\frac{1}{2}}$ denotes the symmetric positive semidefinite
            square root of $R_{\hat{S}}$.
            The last matrix is symmetric and therefore all of the matrices above only have
            real eigenvalues.
            By Proposition~\ref{prop:LAF-A-1}(b), the system converges either to a finite
            equilibrium or towards a fixed point at infinity.
        \item We have $\rank(A) = \rank(R_{\widehat{S}}(\Omega \otimes I_n)) = \rank(R_{\widehat{S}}) = m(n-1)$,
            which yields the first statement.
            The second statement follows from
            \begin{equation}
            	{R_{\widehat{S}}(\Omega \otimes I_n)(e_i \otimes \eins_n)} = {R_{\widehat{S}}(\Omega e_i \otimes I_n \eins_n)} = {R_{\widehat{S}}(\Omega e_i \otimes \eins_n)} = 0, 
	   \end{equation}
	   since ${R_{\widehat{S}_{i}}\eins_{n}=0}$, ${\forall i \in I}$.
            With Proposition~\ref{prop:LAF-A-1}(c) we conclude that 0 is the only finite equilibrium.
        \item $R_{\widehat{S}}$ is positive semidefinite and we have
        \begin{equation}
            \sigma(R_{\widehat{S}}(\Omega \otimes I_n)) = {\sigma((\Omega \otimes I_n)^\frac{1}{2}R_{\widehat{S}}(\Omega \otimes I_n)^\frac{1}{2})}.
        \end{equation}
            Hence, by Sylvester's law, the matrices
            $(\Omega \otimes I_n)^\frac{1}{2}R_{\widehat{S}}(\Omega \otimes I_n)^\frac{1}{2}$
            and $R_{\widehat{S}}$ have the same inertia.
            Thus, the center-stable manifold contains only the origin.
            Proposition~\ref{prop:LAF-A-1}(d) yields divergence to infinity.
    \end{enumerate} \qed
\end{proof}

\begin{remark}
    If $\Omega$ is not a row-wise positive scaling of a
    symmetric matrix, the resulting matrix $A$ may have complex eigenvalues.
    This can be seen for the choice
    \begin{alignat}{1}
        \hat{S} = \frac{1}{2}\begin{pmatrix}
            1 & 1 \\ 1 & 1
        \end{pmatrix},\quad
        \Omega = \frac{1}{2}\begin{pmatrix}
            1 & 1 \\ -1 & 1
        \end{pmatrix},
    \end{alignat}
    for which the matrix $A$ has the eigenvalues $\sigma(A) = \{\frac{1}{2}+\frac{1}{2} i, \frac{1}{2}-\frac{1}{2} i, 0, 0 \}$.
    Note that $\Omega$ is a row-wise scaling of a symmetric matrix but not a
    row-wise \textit{positive} scaling.

    The same matrix $\hat{S}$ and the matrix
    \begin{alignat}{1}
        \Omega = \frac{1}{2}\begin{pmatrix}
            -1 & 1 \\ 1 & -1
        \end{pmatrix}
    \end{alignat}
    yield only nonpositive eigenvalues $\sigma(A) = \{-\frac{1}{2}, 0, 0, 0 \}$.

    For uniform positive weights~\eqref{eq:uniform-weights}, $\Omega$ has nonpositive eigenvalues.
    The existence of the eigenvalue $0$ depends on the size of the graph and the
    size of the neighborhood.
    If $\Omega$ is a randomly chosen or a matrix of the
    form~\eqref{eq:sym-Omega} and estimated from data, it generally has negative eigenvalues.
\end{remark}

To analyze the asymptotic behavior of the lifted flow
\begin{equation}\label{eq:LAF-lifted}
W(t) = \Exp_{W_{0}}\big(V(t)\big)
= \exp_{W_{0}}\Big(\frac{V(t)}{W_{0}}\Big),
\end{equation}
it is enough to lift the line in direction of the maximal eigenvector to the
assignment manifold, as examined next.
\begin{lemma}
    Let $v$ be a vector which has its maximal entries at the positions
    $\{i_1, \dots i_k\} = \argmax_i v_i$.
    Then the line in direction $v$ lifted at $p \in \mc{S}$ converges to a
    specific point on a $k$-dimensional face of $\mc{S}$ given by
    \begin{alignat}{1}
        \lim_{t \to \infty} \exp_p(t v) = \frac{1}{\sum_{l \in [k]} p_{i_l}} \sum_{l \in [k]} p_{i_l} e_{i_l}.
    \end{alignat}
    In particular, if $v$ has a unique maximal entry, $\lim_{t \to \infty} \exp_p(t v)$
    converges to the corresponding unit vector.
\end{lemma}
\begin{proof}
    Set $v_{\max} = \max_i v_i$ and consider
    $\exp_p(t v) = \exp_p(t (v - v_{\max} \eins_{n})) = \frac{p e^{t(v - v_{\max} \eins_{n})}}{\langle p, e^{t(v - v_{\max} \eins_{n})} \rangle}$.
    In the numerator, every entry which does not correspond to a maximal entry of
    $v$ converges to $0$ for $t \to \infty$, whereas the other entries converge
    to the corresponding entry in $p$.
    The denominator normalizes the expression, which yields the result. \qed
\end{proof}

Applying this lemma to each vertex in $I$, we get the following
statement on the convergence of the lifted linear assignment flow to integral
assignments.
\begin{corollary}\label{cor:integral-assignment-LAF}
    Under the assumptions of
    Lemma~\ref{lmm:LAF-trajectory},
    if $\frac{v_1}{W_0}$ has a unique maximal entry for
    each vertex, then the lifted flow~\eqref{eq:LAF-lifted} converges to an
    integral assignment.
\end{corollary}

Because $W_0$ and the dominant eigenvector of $A$ depend on real data in practice, the assumptions of Corollary~\ref{cor:integral-assignment-LAF} are typically satisfied.

We conclude this section by comparing the convergence properties of the $S$-flow
to those of the linear assignment flow.
\begin{remark}[$S$-flow vs.~linear assignment flow]
    If $\Omega$ is nonnegative on the diagonal with at least one positive entry,
    the Jacobian matrices of the $S$-flow (at nonintegral points) and the
    Jacobian matrix of the linear assignment flow, i.e.\ $A$, have at least
    one eigenvalue with positive real part (see Proposition~\ref{prop:Jacobian-eigen}(c) and Proposition~\ref{prop:LAF-A-2}(a)).
    Thus, for both flows and such an $\Omega$, the nonintegral equilibria are
    unstable (Corollary~\ref{cor:stability-S}(c) and Proposition~\ref{prop:LAF-A-2}(a)).

    Theorem~\ref{thm:convergence} and Proposition~\ref{prop:LAF-A-2}(b) state
    that for both flows a sufficient condition for convergence is that $\Omega$
    has the form~\eqref{eq:sym-Omega}.
    Let $\Omega$ have both properties, i.e.\ nonnegative on the diagonal with at least
    one positive entry and row-wise positive scaling of a symmetric matrix.
    Then, the set of initial values converging to a nonintegral point is negligible (Proposition~\ref{prop:null-set}, Theorem~\ref{thm:measure-0} and Proposition~\ref{prop:LAF-A-1}(a)).

    For a given initial value, the two flows generally converge to different limit
    points and their regions of attraction generally look different.
    However, for small finite time-points, the linear assignment flow approximates
    the assignment flow and (after the appropriate transformation) the $S$-flow
    very well~\cite{Zeilmann:2018aa}.
\end{remark}

%% file: NumericalExamples.tex

\section{Discretization, Numerical Examples and Discussion}
\label{sec:numerics}

\subsection{Discretization, Geometric Integration}

We confine ourselves to the simplest geometric scheme worked out by~\cite{Zeilmann:2018aa} for numerically integrating the assignment flow~\eqref{eq:AF}. Applying this scheme to the $S$-flow~\eqref{eq:def-S-flow-F} that has the same structure as~\eqref{eq:AF}, yields the iteration
\begin{equation}\label{eq:Euler-scheme}
S^{(t+1)} = F_h( S^{(t)} ),\qquad
F_h(S) = \exp_{S}(h \Omega S),\qquad h > 0,\; t \in \N_{0},
\end{equation}
where $h$ denotes a fixed step size and the iteration step $t$ represents the points of time $t h$.

The following proposition shows that using this numerical method is `safe' in the sense that, by setting $h$ to a sufficiently small value, the approximation of the continuous-time solution $S(t)$ by the sequence $\big(S(t h)\big)_{t \geq 0}$ generated by
\eqref{eq:Euler-scheme} can become arbitrarily accurate.

\begin{proposition}\label{prop:approx-discr-euler}
Let $L > 0$ be the Lipschitz constant of the mapping $F$~\eqref{eq:def-S-flow-F} defining the $S$-flow.
Then there exists a constant $C > 0$ such that the solution $S(t)$ to the $S$-flow~\eqref{eq:def-S-flow-F} and the sequence $\big(S(t h)\big)_{t \geq 0}$ generated by
\eqref{eq:Euler-scheme} satisfy the relation
\begin{equation} \label{eq:approx-discr-euler}
\big\| S(t h) - S^{(t)} \big\| \leq \frac{C}{2 L} h e^{(t+1) L h},\qquad \forall t \in \N.
\end{equation}
\end{proposition}
\begin{proof}
See Appendix~\ref{sec:Proofs}.
\end{proof}

Proposition~\ref{prop:attraction} asserts the existence of regions of attraction for stable equilibria $S^{\ast}\in\ol{\mc{W}}$ of the continuous-time $S$-flow~\eqref{eq:def-S-flow-F}. The following proposition extends this assertion to the discrete-time $S$-flow~\eqref{eq:Euler-scheme}.
\begin{proposition}\label{prop:region-attraction-Euler}
Let $\Omega$, $S^* \in \Wstar$, $A(S^*)$ and $B_{\varepsilon}(S^*)$ be as in Proposition~\ref{prop:attraction}.
Then, for the sequence $(S^{(t)})_{t \in \N}$ generated by~\eqref{eq:Euler-scheme}, the following holds.
If $S^{(t_0)} \in B_{\varepsilon}(S^*)$ for some time point $t_0 \in \N$, then $S^{(t)} \in B_{\varepsilon}(S^*)$ for all $t \geq t_0$ and $\lim_{t \rightarrow \infty} S^{(t)} = S^*$.
Moreover, we have
\begin{equation}\label{eq:rate-Euler}
\big\| S_i^{(t)} - S_i^* \big\|_{1} \leq \big\| S_i^{(t_0)} - S_i^* \big\|_{1} \cdot \gamma_i^{t - t_0}
\end{equation}
with $\gamma_i \in (0,1)$, for each $i \in I$.
\end{proposition}
\begin{proof}
Let
\begin{equation}
\beta_i = \beta_i(S) \coloneqq \min \big\{ (\Omega S)_{i j^*(i)} - (\Omega S)_{ij} \big\}_{j \neq j^*(i)}.
\end{equation}
For $S \in A(S^*)$, we have $\beta_i(S) > 0$ and with $S_{i}^{\ast}=e_{ij^{\ast}(i)}$, $F_{h,i}(S)\in\Delta_{n}$,
\begin{subequations}
\begin{align}
\big\| F_{h,i}(S) - S_i^* \big\|_1
	&= 2 - 2 F_{h, i j^*(i)}(S) \\
	&= 2 - 2 \frac{ S_{i j^*(i)} }{ S_{i j^*(i)} + \sum_{j \neq j^*(i)} S_{ij} e^{h (\Omega S)_{ij} - h (\Omega S)_{i j^*(i)}} } \\
	&\leq 2 - 2 \frac{ S_{i j^*(i)} }{ S_{i j^*(i)} + (1 - S_{i j^*(i)} ) e^{-h \beta_i} } \\
	&= \| S_i - S_i^* \|_1 \underbrace{\frac{ e^{-h \beta_i} }{ S_{i j^*(i)} + (1 - S_{i j^*(i)}) e^{-h \beta_i} }}_{< 1}.
\end{align}
\end{subequations}
Choosing $\delta > 0$ with $S^{(t_0)} \in \overline{B_{\delta}(S^*)} \subset B_{\varepsilon}(S^*)$, we set
\begin{equation}
\gamma_i = \max_{S \in \ol{B_{\delta}(S^*)}}~\frac{ e^{-h \beta_i(S)} }{ S_{i j^*(i)} + (1 - S_{i j^*(i)}) e^{-h \beta_i(S)} } \in (0,1).
\end{equation}
and thus get $\| F_{h,i}(S) - S_i^* \|_1 \leq \gamma_i \| S_i - S_i^* \|_1$ for
$S \in B_{\delta}(S^*)$, which implies $F_h(\ol{B_{\delta}(S^*)}) \subseteq \ol{B_{\delta}(S^*)} \subset B_{\varepsilon}(S^*)$ and the exponential convergence rate~\eqref{eq:rate-Euler} of $S^{(t)}$. \qed
\end{proof}
\subsection{Numerical Examples, Discussion}

We illustrate in this section by a range of counter-examples that violating assumption~\eqref{eq:sym-Omega} can make the assignment flow behave quite differently from what the assertions of Section~\ref{sec:AF-Properties} predict.
In fact, we use violations of the assumptions as a guiding principle for constructing alternative asymptotic behavior (Section~\ref{sec:3x3counter-examples}).

In addition, we briefly discuss the influence of the parameter matrix $\Omega$ on the spatial shape of labelings returned by the assignment flow. Finally, we illustrate that our results on the region of attraction of the $S$-flow towards labelings turns the termination criterion proposed by~\cite{Astroem2017} into a mathematically sound one, provided a proper geometrical scheme is used for numerically integrating the assignment flow.

\subsubsection{Vanishing Diagonal Averaging Parameters}\label{sec:vanishing-diagonal}

We consider a small dynamical system that violates the basic assumption of Corollary~\ref{cor:stability-S}, that all diagonal entries of the parameter matrix $\Omega$ of the $S$-flow~\eqref{eq:def-S-flow-F} are positive. As a consequence, an entire line of nonintegral points $S^{\ast}$ is locally attracting the flow.
\begin{example}\label{ex:nonpos-diag}
Let $m=|I| = 3$ and $n=|J| = 2$, and let the parameters of the $S$-flow~\eqref{eq:def-S-flow-F} be given by the row-stochastic matrix
\begin{equation}\label{eq:ex-nonpos-diag-Omega}
\Omega = \{\omega_{ik}\}_{k \in \mc{N}_{i}, i \in I} = \frac 14 \begin{pmatrix} 0 & 2 & 2 \\ 1 & 2 & 1 \\ 1 & 1 & 2 \end{pmatrix}.
\end{equation}
One easily checks that any point $S^{\ast}$ on the line $\mc{L}$
\begin{equation}\label{eq:ex-nonpos-diag-L}
\mathcal{L} = \left\{ \begin{pmatrix} p & 1-p \\ 1 & 0 \\ 0 & 1 \end{pmatrix} \colon p \in [0,1] \right\} \subset \ol{\mathcal{W}}
\end{equation}
is an equilibrium of the $S$-flow satisfying $F(S^{\ast})=0$. In particular, this includes nonintegral points with $p \in (0,1)$. The eigenvalues of the Jacobian are given by
\begin{equation}
\sigma\big( \tfrac{\partial F}{\partial S}(S^*) \big) = \big\{ 0, -\tfrac 12, -\tfrac{p+2}{4}, -\tfrac p2, -\tfrac{1-p}{2}, -\tfrac{3-p}{4} \big\} \subset \R_{\leq 0}
\end{equation}
and are nonpositive. The phase portrait depicted by Figure~\ref{fig:example_nonpos_diag} illustrates that $\mc{L}$ locally attracts the flow.

This small example demonstrates that violation of the basic assumption---here,
specifically, $\w_{11}$ of~\eqref{eq:ex-nonpos-diag-Omega} is \textit{not} positive---leads to $S$-flows with properties not covered by the results of Section~\ref{sec:AF-Properties}.
Note that Theorem~\ref{thm:measure-0} is also based on this assumption and does not apply to the present example: there is an open set of starting points $S_{0} \in \mc{W}$ for which the $S$-flow converges to nonintegral equilibria $S^{\ast} \in \ol{\mathcal{W}}$.

Recalling Corollary~\ref{cor:integral-assignment-LAF}, we see that for
the linear assignment flow~\eqref{eq:LAF} continuous sets on the boundary of the assignment
manifold, like line $\mc{L}$ in Figure~\ref{fig:example_nonpos_diag}, cannot be limit points.
\end{example}
\begin{figure}
\begin{center}
\includegraphics[trim=2cm 0cm 0.5cm 0.75cm, clip, width=0.35\textwidth]{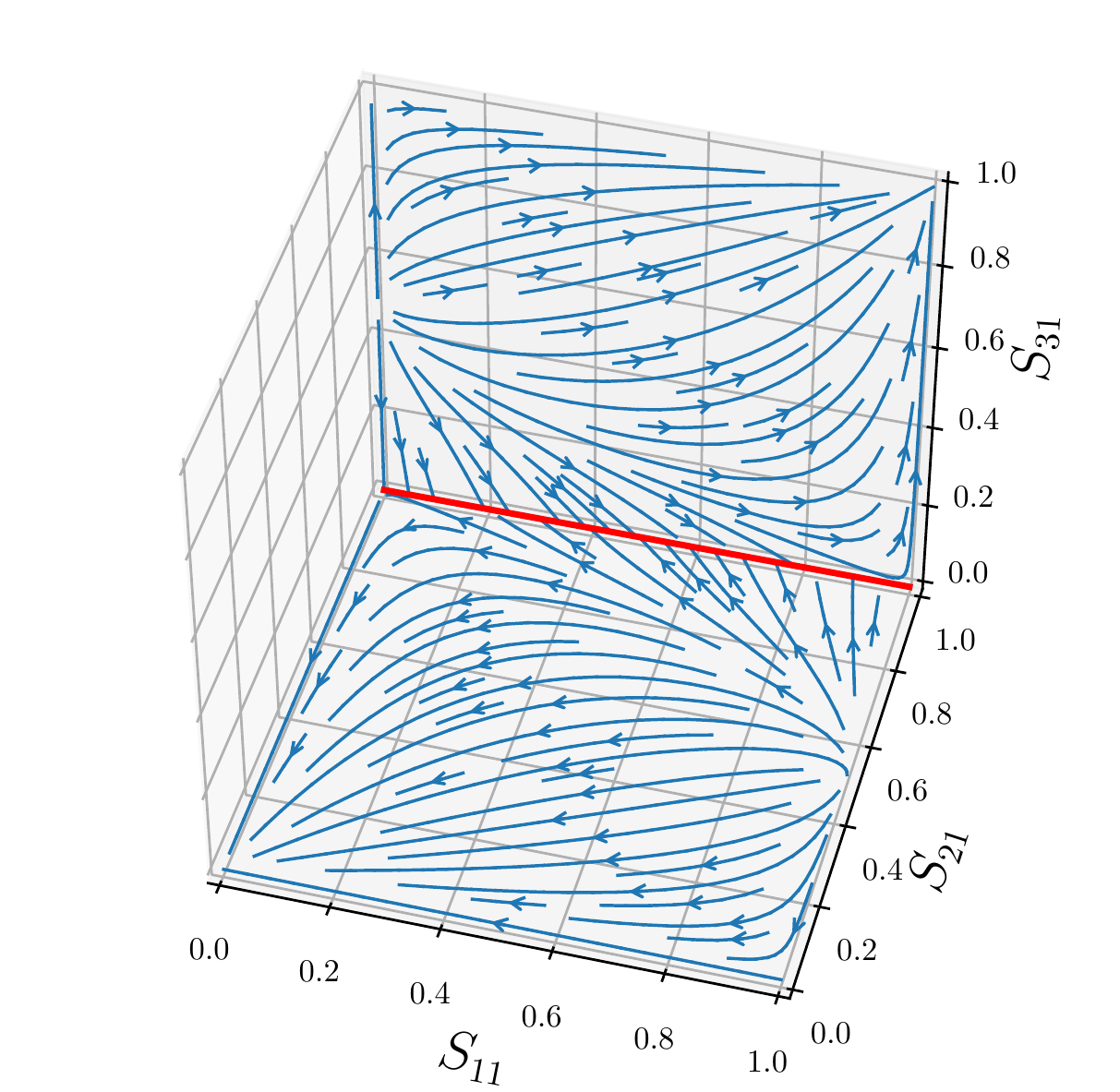} \qquad
\includegraphics[width=0.35\textwidth]{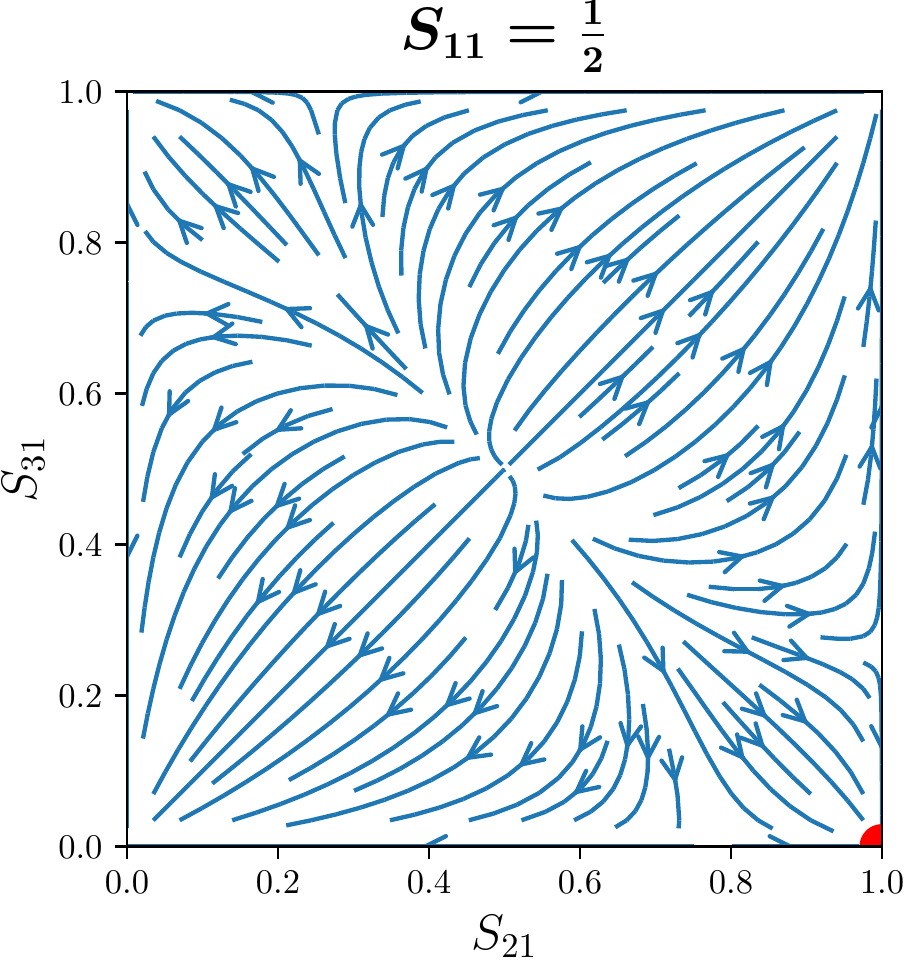}
\end{center}
\caption[Phase portrait for the flow of Example~\ref{ex:nonpos-diag}.]{%
\textbf{Phase portrait for the flow of Example~\ref{ex:nonpos-diag}.} We graphically depict the $S$-flow with $\Omega$ given by~\eqref{eq:ex-nonpos-diag-Omega}, by its first column. This describes the flow completely, since $n=|J|=2$. The left panel shows the phase portrait of the flow within the planes $\{ S_{21} = 1 \}$ and $\{ S_{31} = 0 \}$. The plane $\{ S_{11} = \frac{1}{2} \}$ is depicted by the right panel. The line $\mathcal{L}$ of equilibria given by~\eqref{eq:ex-nonpos-diag-L} is marked red and located in the lower right vertex in the right plot. The phase portrait illustrates that this line attracts the flow within a small neighborhood.} \label{fig:example_nonpos_diag}
\end{figure}
\subsubsection{\texorpdfstring{Constructing $3\times 3$ Systems with Various Asymptotic Properties}{Constructing 3x3 Systems with Various Asymptotic Properties}}\label{sec:3x3counter-examples}
In this section, we construct a family of $S$-flows~\eqref{eq:prop_sflow} in terms of a class of nonnegative parameter matrices $\Omega$, that may violate assumption~\eqref{eq:sym-Omega} which underlies Theorem~\ref{thm:convergence}. Accordingly, for a small problem size $n=3$, we explicitly specify flows that exhibit one of the following behaviors:
\begin{enumerate}
\item $t \mapsto S(t)$ converges towards a point $S^* \in \overline{\mathcal{W}}$ as $t \rightarrow \infty$;
\item $t \mapsto S(t)$ is periodic with some period $t_1 > 0$;
\item $t \mapsto S(t)$ neither converges to a point nor is periodic.
\end{enumerate}
These cases are discussed below as Example~\ref{ex:non-convergent_s-flow} and illustrated by Figure~\ref{fig:non-convergent}. They demonstrate that assumption~\eqref{eq:sym-Omega} is not too strong, because violation may easily imply that the flow fails to converge to an equilibrium.

Let $\mc{D}$ denote the set of \textit{doubly stochastic, circulant} matrices. We consider the case $m=|I|=|J|=n$ and therefore have $\mc{D} \subset\ol{\mathcal{W}}$. Let
\begin{equation}
P \in \{0,1\}^{n \times n},\qquad
P_{ij} = \begin{cases} 1, &\text{if $i-j \equiv 1\ (\operatorname{mod} n)$,} \\ 0, &\text{else} \end{cases}
\end{equation}
denote the permutation matrix that represents the $n$-cycle $(1,\dotsc,n)$. Then $\mc{D}$ is the convex hull of the matrices $\{P, P^2, \dots, P^n \}$ with $P^n = I_n$, and any element $M \in \mathcal{D}$ admits the representation
\begin{equation}
M = \sum_{k \in [n]} \mu_{k} P^k \quad \text{with} \quad \mu \in \Delta_{n}.
\end{equation}
Since the matrices $P, P^2, \dots, P^n \in \R^{n \times n}$ are linearly independent, the vector ${\mu \in \Delta_n}$ is uniquely determined. We will call $\mu$ the \textit{representative} of $M \in \mathcal{D}$. The following Lemma characterizes two matrix products on $\mc{D}$ in terms of the corresponding matrix representatives.
\begin{lemma} \label{lem:dot_in_D}
Let $\mu^{(1)}, \mu^{(2)} \in \Delta_{n}$ be the representatives of any two matrices $M^{(1)}, M^{(2)} \in \mathcal{D}$. Then the element-wise Hadamard product and the ordinary matrix product, respectively, are given by
\begin{align}
M^{(1)} \odot M^{(2)} &= \sum_{k \in [n]} \eta_{k} P^k \quad \text{with} \quad \eta = \mu^{(1)} \odot \mu^{(2)} \in \R_{\geq 0}^n,
\label{eq:D-Hadamard-product} \\ \label{eq:D-ordinary-product}
M^{(1)} M^{(2)} &= \sum_{k \in [n]} \mu_{k} P^k \quad \text{with} \quad \mu = M^{(1)} \mu^{(2)} \in \Delta_n.
\end{align}
\end{lemma}
\begin{proof}
We note that the $k$th power of $P$ is given by
\begin{equation} \label{eq:Pkij}
(P^k)_{ij} = \begin{cases} 1, &\text{if $i-j \equiv k\ (\operatorname{mod} n)$,} \\ 0, &\text{else.} \end{cases}
\end{equation}
This implies $P^k \odot P^l = \delta_{kl} P^k$ for $k,l \in [n]$, with $\delta_{kl}$ denoting the Kronecker delta, and
\begin{align}
\begin{split}
M^{(1)} \odot M^{(2)} 
	&= \bigg( \sum_{k \in [n]} \mu_k^{(1)} P^k \bigg) \odot \bigg( \sum_{l \in [n]} \mu_l^{(2)} P^l \bigg)
	= \sum_{k,l \in [n]} \mu_k^{(1)} \mu_l^{(2)} P^k \odot P^l \\
	&= \sum_{k \in [n]} \mu_k^{(1)} \mu_k^{(2)} P^k.
\end{split}
\end{align}
As for~\eqref{eq:D-ordinary-product}, we compute
\begin{subequations}
\begin{align}
M^{(1)} M^{(2)}
	&\stackrel{\hphantom{(\ref{eq:Pkij})}}{=} \sum_{k,j \in [n]} \mu_{k}^{(1)} \mu_{j}^{(2)} P^{k+j}
	= \sum_{i \in [n]}~\sum_{k+j \equiv i (\operatorname{mod} n)} \mu_{k}^{(1)} \mu_{j}^{(2)} P^{i} \\
	&\stackrel{(\ref{eq:Pkij})}{=} \sum_{i \in [n]} \sum_{k \in [n]} \mu_{k}^{(1)} ( P^{k} \mu^{(2)})_i \, P^i
	= \sum_{i \in [n]} ( M^{(1)} \mu^{(2)} )_i \, P^i.
\end{align}
\end{subequations} \qed
\end{proof}
The following proposition shows that the $S$-flow on $\mc{D}$ can be expressed by the evolution of the corresponding representative.
\begin{proposition}
Let $\Omega \in \mathcal{D}$ and suppose the $S$-flow~\eqref{eq:prop_sflow} is initialized at $S(0) \in \mathcal{D}$. Then the solution $S(t) \in \mathcal{D}$ evolves on $\mc{D}$ for all $t \in \R$. In addition, the corresponding representative $p(t) \in \Delta_n$ of $S(t) = \sum_{k \in [n]} p_{k}(t) P^{k}$ satisfies the replicator equation
\begin{equation}\label{eq:p-representative-S-flow}
\dot{p} = R_{p}(\Omega p).
\end{equation}
\end{proposition}
\begin{proof}
Let $S = \sum_{k \in [n]} p_k P^k \in \mathcal{D}$ with $p \in \Delta_n$.
Lemma~\ref{lem:dot_in_D} implies
\begin{equation}\label{eq:Hadamard-S-Omega-S}
S \odot \Omega S = \sum_{k \in [n]} p_k (\Omega p)_k \, P^k.
\end{equation}
Therefore, for any $i \in [n]$,
\begin{align}
\begin{split}
\langle S_i, (\Omega S)_i \rangle 
	&= \langle \eins_{n}, S_i \odot (\Omega S)_i \rangle = \big\langle \eins_{n}, \big( S \odot (\Omega S) \big)_i \big\rangle \\
	&= \sum_{k \in [n]} p_k (\Omega p)_k \underbrace{\la\eins_{n},(P^{k})_{i}\ra}_{=1} = \langle p, \Omega p \rangle.
\end{split}
\end{align}
Since this equation holds for any $i \in [n]$, the right-hand side of the $S$-flow~\eqref{eq:prop_sflow} can be rewritten as
\begin{subequations}\label{eq:ReplS_Replp}
\begin{align}
R_{S}(\Omega S)
&= S \odot (\Omega S) - \langle p, \Omega p \rangle S
\overset{\eqref{eq:Hadamard-S-Omega-S}}{=}
\sum_{k \in [n]}\Big(p_k (\Omega p)_k \, P^k  - \langle p, \Omega p \rangle p_{k} P^{k}\Big)
\\
&= \sum_{k \in [n]} v_k P^k \quad \text{with} \quad v = p \odot(\Omega p) - \langle p, \Omega p \rangle p = R_{p}(\Omega p).
\end{align}
\end{subequations}
Since $p \in \Delta_{n}$, we have $\la v,\eins_{n}\ra = 0$, that is $v$ is tangent to $\Delta_{n}$. Hence, by~\eqref{eq:ReplS_Replp}, $\dot S = \sum_{k \in [n]} \dot p_{k} P^{k} = R_{S}(\Omega S) $ is determined by $\dot p = v = R_{p}(\Omega p)$, whose solution $p(t)$ evolves on $\Delta_{n}$. \qed
\end{proof}
The following proposition introduces a restriction of parameter matrices $\Omega \in \mc{D}$ that ensures, for any such $\Omega$, that the product $\prod_{j \in [n]} p_{j}$ changes monotonously depending on the flow~\eqref{eq:p-representative-S-flow}.
\begin{proposition} \label{prop:prod_monotone}
Let $\Omega  = \sum_{k \in [n]} \mu_{k} P^{k} \in \mathcal{D}$ be parametrized by
\begin{equation} \label{eq:mu_restricted}
\mu = \alpha e_{n} + \frac{\beta}{n} \eins_n + \sum_{k < \big\lfloor \tfrac{n}{2} \big\rfloor} \gamma_k (e_{k} - e_{n-k}) \in \Delta_n,\qquad
\alpha,\beta,\gamma_{1},\dotsc,\gamma_{\lfloor \tfrac{n}{2}\rfloor-1} \in \R.
\end{equation}
Suppose $p(t) \in \mathcal{S} = \rint(\Delta_n)$ solves~\eqref{eq:p-representative-S-flow}. Then
\begin{equation}\label{eq:prop-dt-pi-p}
\frac{\D}{\D t} \prod_{j \in [n]} p_j(t)
\left\{ \begin{aligned}
&< 0, \;\text{if}\; \alpha > 0
\\
&= 0, \;\text{if}\; \alpha = 0
\\
&> 0, \;\text{if}\; \alpha < 0
\end{aligned}\right\}, \quad \text{for} \quad p(t) \neq \tfrac{1}{n} \eins_n.
\end{equation}
\end{proposition}
\begin{proof}
Set $\pi_{p} \coloneqq \prod_{j \in [n]} p_j$. By virtue of~\eqref{eq:p-representative-S-flow} and $\la\eins_{n},\Omega p\ra=\la \Omega^{\T}\eins_{n},p\ra=1$ ($\Omega$ is doubly stochastic and $p \in \Delta_{n}$), we have
\begin{equation}\label{eq:dt-pi-p-proof}
\frac{\D}{\D t} \pi_{p} = \pi_{p}\sum_{j \in [n]} \big( (\Omega p)_j - \langle p, \Omega p \rangle \big) = \pi_{p} \big( 1 - n \langle p, \Omega p \rangle \big).
\end{equation}
Hence, since $\pi_{p} > 0$ for $p \in \mathcal{S}$, $\frac{\D}{\D t} \pi_{p}$ has the same sign as $\tfrac{1}{n} - \langle p, \Omega p \rangle$.
Regarding the term
\begin{equation} \label{eq:pOmegap}
\langle p, \Omega p \rangle = \sum_{k \in [n]} \mu_k \langle p, P^k p \rangle,
\end{equation}
we have the following three cases:
\begin{enumerate}
\item[($\alpha$)] for all $k < n$, the inequality $\langle p, P^k p \rangle \leq \langle p, p \rangle = \langle p, P^n p \rangle$ holds, with equality if and only if $p = \tfrac 1n \eins_n$;
\item[($\beta$)] $\sum_{k \in [n]} \langle p, P^k p \rangle = \langle p, \eins_{n \times n} \, p \rangle = 1$;
\item[($\gamma$)] for all $k \in [n]$, $\langle p, P^k p \rangle = \langle p, P^{n-k} p \rangle$, since $P^{-1} = P^\top$.
\end{enumerate}
Inserting~\eqref{eq:mu_restricted} into~\eqref{eq:pOmegap} and applying $(\alpha),(\beta),(\gamma)$ gives
\begin{equation}
\langle p, \Omega p \rangle = \alpha \langle p, p \rangle + \beta \frac{1}{n} \quad \text{and} \quad \langle p, p \rangle > \frac{1}{n} \sum_{k \in [n]} \langle p, P^k p \rangle = \frac{1}{n} \quad \text{for} \quad p \neq \tfrac{1}{n} \eins_n.
\end{equation}
Since $\langle \mu, \eins_n \rangle = \alpha + \beta = 1$, we further obtain
\begin{equation}\label{eq:ineq-p-Omega-p}
\langle p, \Omega p \rangle
\left\{\begin{aligned}
&> \tfrac 1n,\;\text{if}\; \alpha > 0
\\
&= \tfrac 1n,\;\text{if}\; \alpha = 0
\\
&< \tfrac 1n,\;\text{if}\; \alpha < 0
\end{aligned}\right\},\quad \text{for all} \quad p \in \Delta_n \setminus \{ \tfrac{1}{n} \eins_{n} \}.
\end{equation}
Combining~\eqref{eq:ineq-p-Omega-p} and~\eqref{eq:dt-pi-p-proof} yields~\eqref{eq:prop-dt-pi-p}. \qed
\end{proof}
\begin{remark}\label{rem:Omega-circulant}
Based on Proposition~\ref{prop:prod_monotone}, we observe: If $\alpha > 0$, then $p(t)$ moves towards the (relative) boundary of the simplex $\Delta_n$, for any $p(0) \neq \tfrac 1n \eins_n$. If $\alpha < 0$, then $p(t)$ converges towards the barycenter $\tfrac 1n \eins_n$.
For $\alpha = 0$, the product $\prod_{j \in [n]} p_j(t)$ is constant over time.
\end{remark}
The scalars $\gamma_k$ in~\eqref{eq:mu_restricted} steer the skew-symmetric part of $\Omega$. Consequently,
if $\gamma_k = 0$ for all $k$, then $\Omega$ is symmetric and the $S$-flow converges to a single point by Theorem~\ref{thm:convergence}.
Depending on the skew-symmetric part, the $S$-flow may not converge to a point, as Example~\ref{ex:non-convergent_s-flow} below will demonstrate for few explicit instances and $n=3$. Note that, in this case $n=3$,~\eqref{eq:mu_restricted} describes a parametrization rather than a restriction of $\Omega \in \mathcal{D}$.
\begin{example} \label{ex:non-convergent_s-flow}
Let $n=3$. The matrix $\Omega \in \mathcal{D}$ take the form
\begin{align}
\mu &= \alpha e_3 + \tfrac{\beta}{3}\eins_3 + \gamma (e_1 - e_2), \\
\Omega &= \begin{pmatrix} \mu_3 & \mu_2 & \mu_1 \\ \mu_1 & \mu_3 & \mu_2 \\ \mu_2 & \mu_1 & \mu_3 \end{pmatrix}
	= \alpha \begin{pmatrix} 1 & 0 & 0 \\ 0 & 1 & 0 \\ 0 & 0 & 1 \end{pmatrix} + \frac{\beta}{3} \begin{pmatrix} 1 & 1 & 1 \\ 1 & 1 & 1 \\ 1 & 1 & 1 \end{pmatrix} + \gamma \begin{pmatrix} 0 & -1 & 1 \\ 1 & 0 & -1 \\ -1 & 1 & 0 \end{pmatrix}
\label{eq:Omega3-decomp}
\end{align}
with the constraint $\mu \in \Delta_3$, i.e.
\begin{equation}
\alpha + \beta = 1, \quad \alpha+\frac{\beta}{3} \geq 0, \quad \frac{\beta}{3} \geq |\gamma|.
\end{equation}
We examine the behavior of the flow~\eqref{eq:p-representative-S-flow}, depending on the parameters $\alpha$ and $\gamma$.
Note, that the flow does not depend on the parameter $\beta$ that merely ensures $\Omega$ to be row-stochastic. \par
\textbf{Case $\alpha < 0$.}
As already discussed (Remark~\ref{rem:Omega-circulant}), $p(t)$ converges to the barycenter in this case.
Depending on $\gamma$, this may happen with ($\gamma \neq 0$) or without ($\gamma = 0$) a spiral as depicted by Figure~\ref{fig:non-convergent} (a) and (b). \par
\textbf{Case $\alpha = 0$.} We distinguish the two cases $\gamma = 0$ and $\gamma \neq 0$.
If $\gamma = 0$, then we have $\Omega = \tfrac 13 \eins_{3\times 3}$ and therefore $\dot{p} = R_{p} \Omega p \equiv 0$, i.e., each point $p^* \in \Delta_3$ is an equilibrium.
In contrast, if $\gamma \neq 0$, then we have the (standard) rock-paper-scissors dynamics~\cite[Chapter 10]{schecter2016game}:
\begin{equation}
\dot{p} = \gamma \begin{pmatrix} p_1 (p_3 - p_2) \\ p_2 (p_1 - p_3) \\ p_3 (p_2 - p_1) \end{pmatrix} \neq 0, \quad \text{for } p \in \Delta_3 \setminus \{ e_1, e_2, e_3, \tfrac 13 \eins_3 \}.
\end{equation}
Starting at a point $p_0 \in \rint(\Delta_3) \setminus \{ \tfrac 13 \eins_3 \}$, the curve $t \mapsto p(t)$ moves along the closed curve $\big\{ p \in \Delta_3 \colon \prod_{j} p_j = \prod_{j} p_{0,j} \big\}$, i.e., the curve $t \mapsto p(t)$ is periodic; see Figure~\ref{fig:non-convergent} (c). \par
\textbf{Case $\alpha > 0$.}
We distinguish again the two cases $\gamma = 0$ and $\gamma \neq 0$.
If $\gamma = 0$, then the flow reduces to $\dot{p} = \alpha R_p p$ whose solution converges to
\begin{equation}
\lim_{t \rightarrow \infty} p(t) = \tfrac{1}{|J^*|} \sum_{j \in J^*} e_j \in \Delta_3, \quad \text{with} \quad J^* = \operatorname*{arg\, max}_{j \in [3]}~p_j(0).
\end{equation}
As for the remaining case $\alpha > 0$ and $\gamma \neq 0$, we distinguish $\alpha > |\gamma|$ and $\alpha \leq |\gamma|$ as illustrated by Figure~\ref{fig:non-convergent} (e), (f) and (g).
If $\alpha \leq |\gamma|$, then we have a generalized rock-paper-scissors game~\cite[Chapter 10]{schecter2016game}.
The curve $t \mapsto p(t)$ spirals towards the boundary of the simplex $\Delta_3$ and does not converge to a single point.
In contrast, if $\alpha > |\gamma|$, then the flow converges to a point on the boundary. In fact, the vertices of the simplex are attractors.
\end{example}

\begin{figure}
\begin{center}
\begin{tabular}{ccc}
\includegraphics[width=0.2\textwidth]{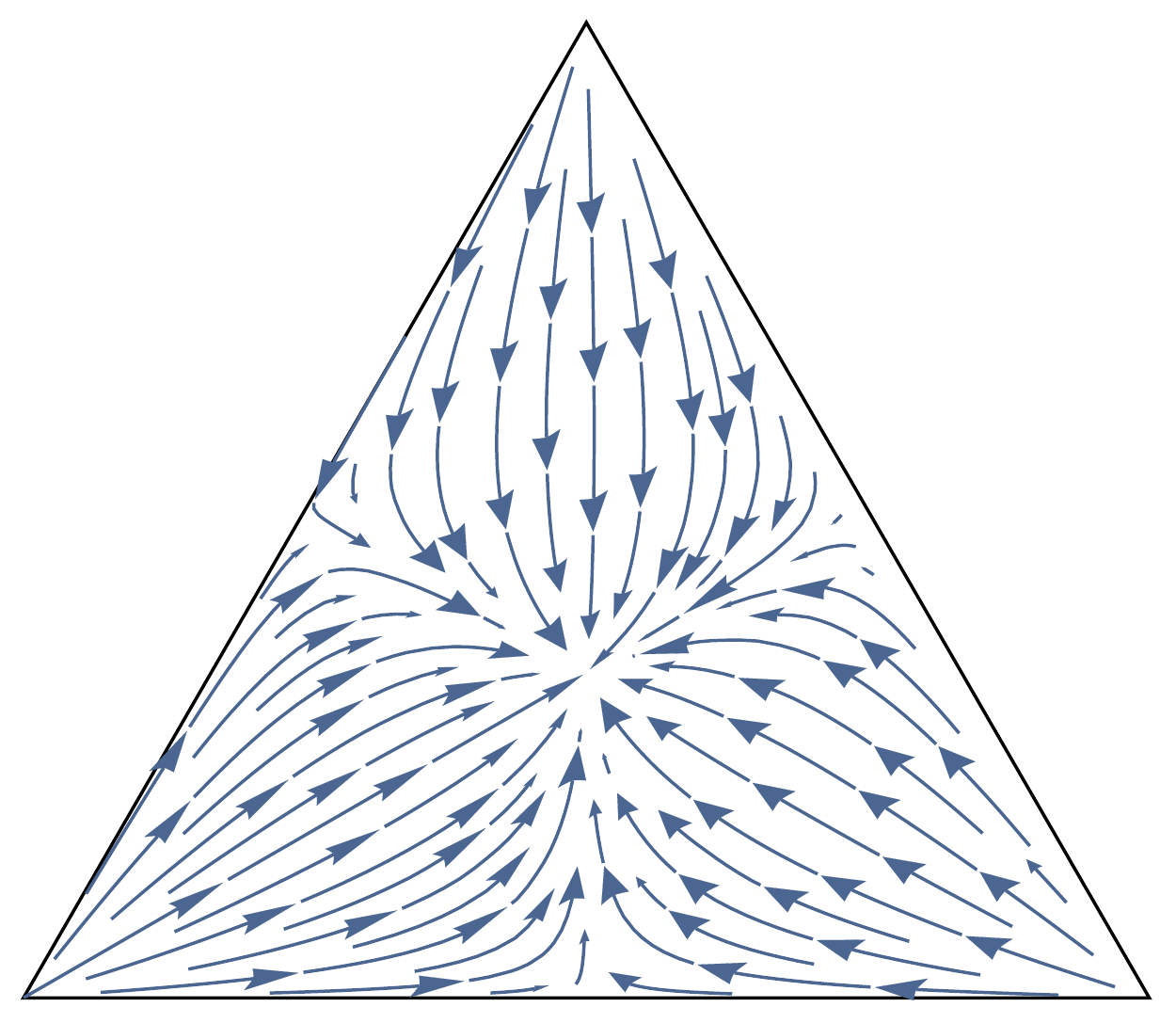} &
\includegraphics[width=0.2\textwidth]{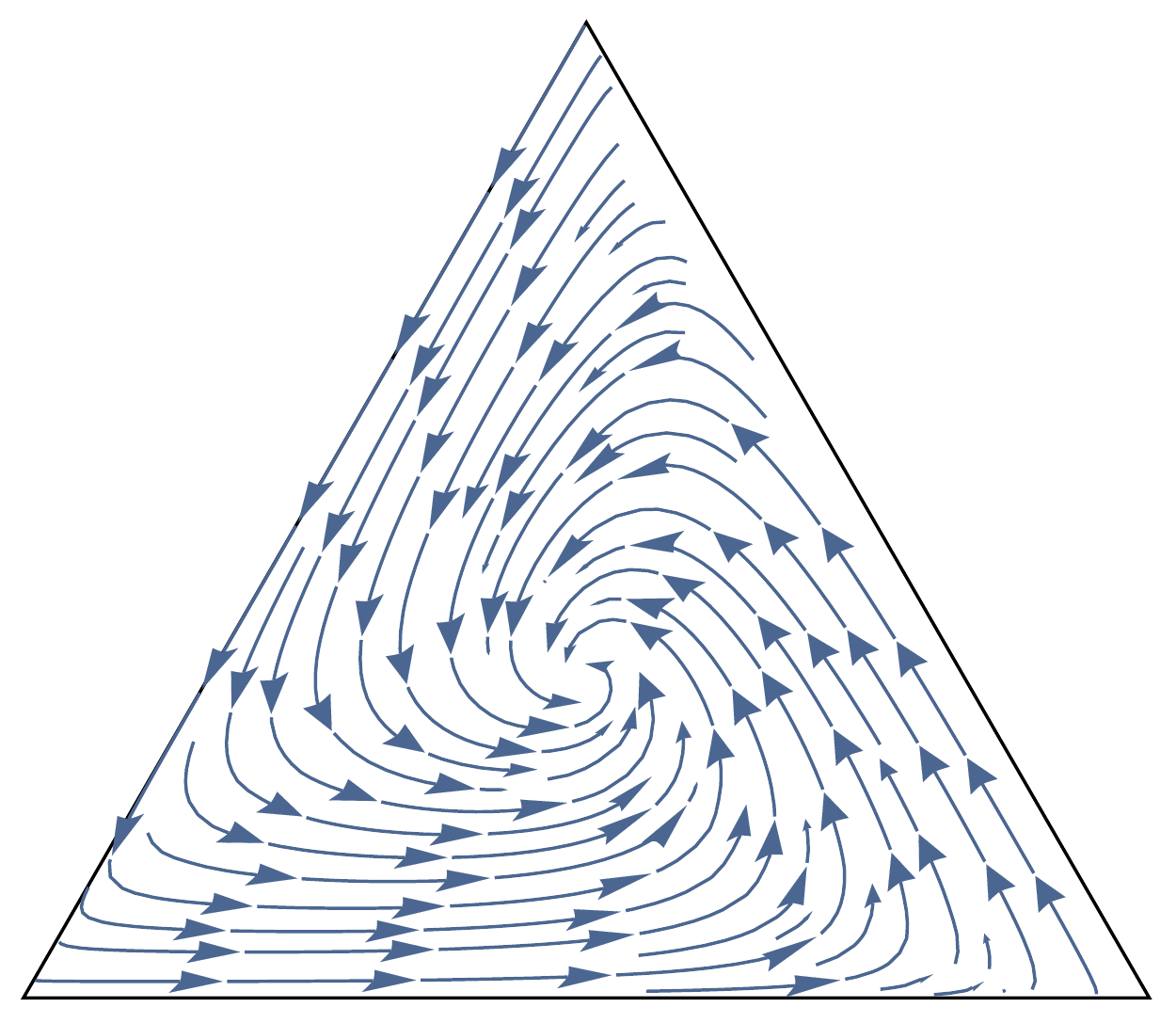} &
\includegraphics[width=0.2\textwidth]{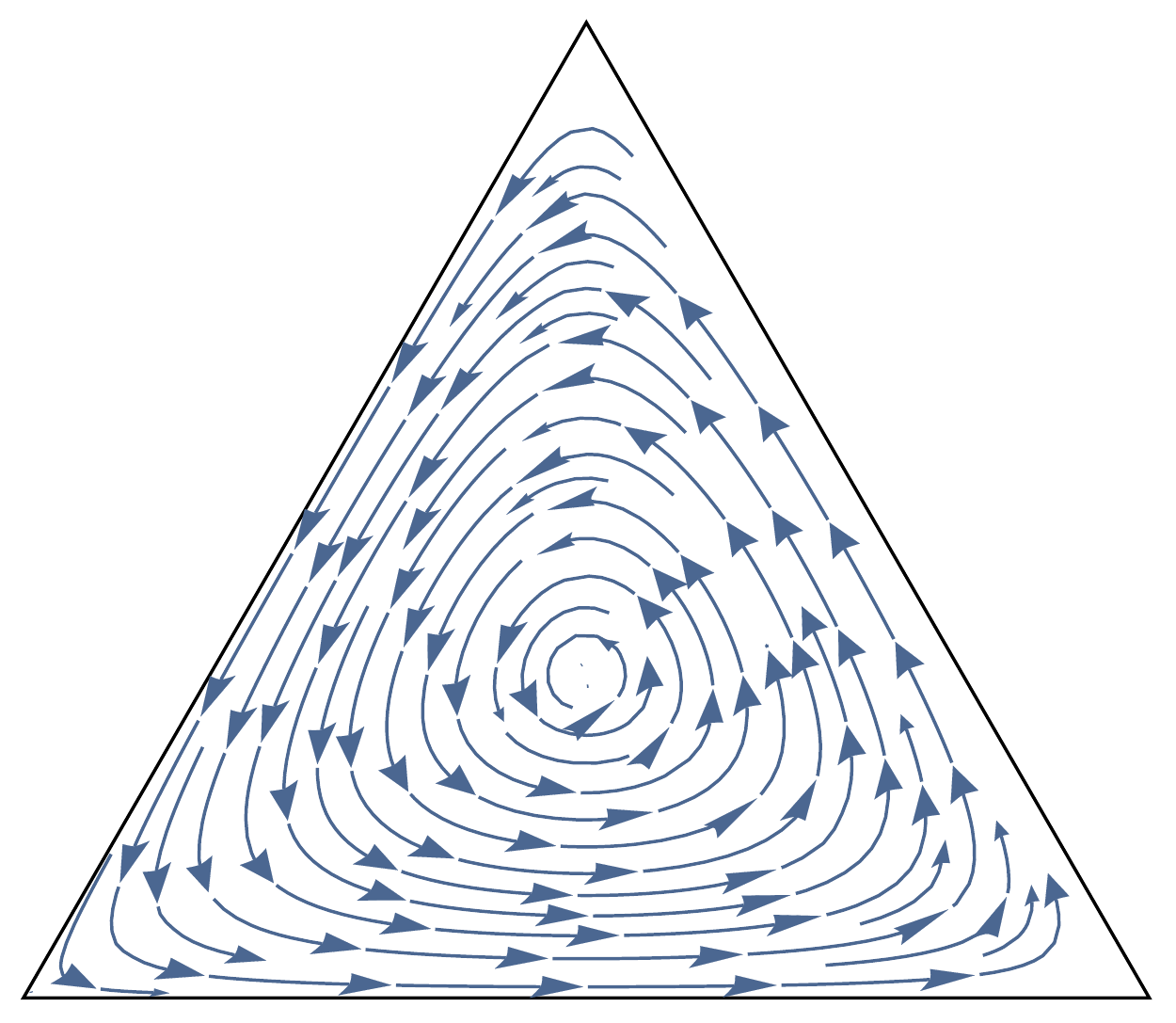} \\
\textbf{(a)} $\alpha < 0$, $\gamma = 0$ &
\textbf{(b)} $\alpha < 0$, $\gamma \neq 0$ &
\textbf{(c)} $\alpha = 0$, $\gamma \neq 0$
\end{tabular}
\begin{tabular}{cccc}
\includegraphics[width=0.2\textwidth]{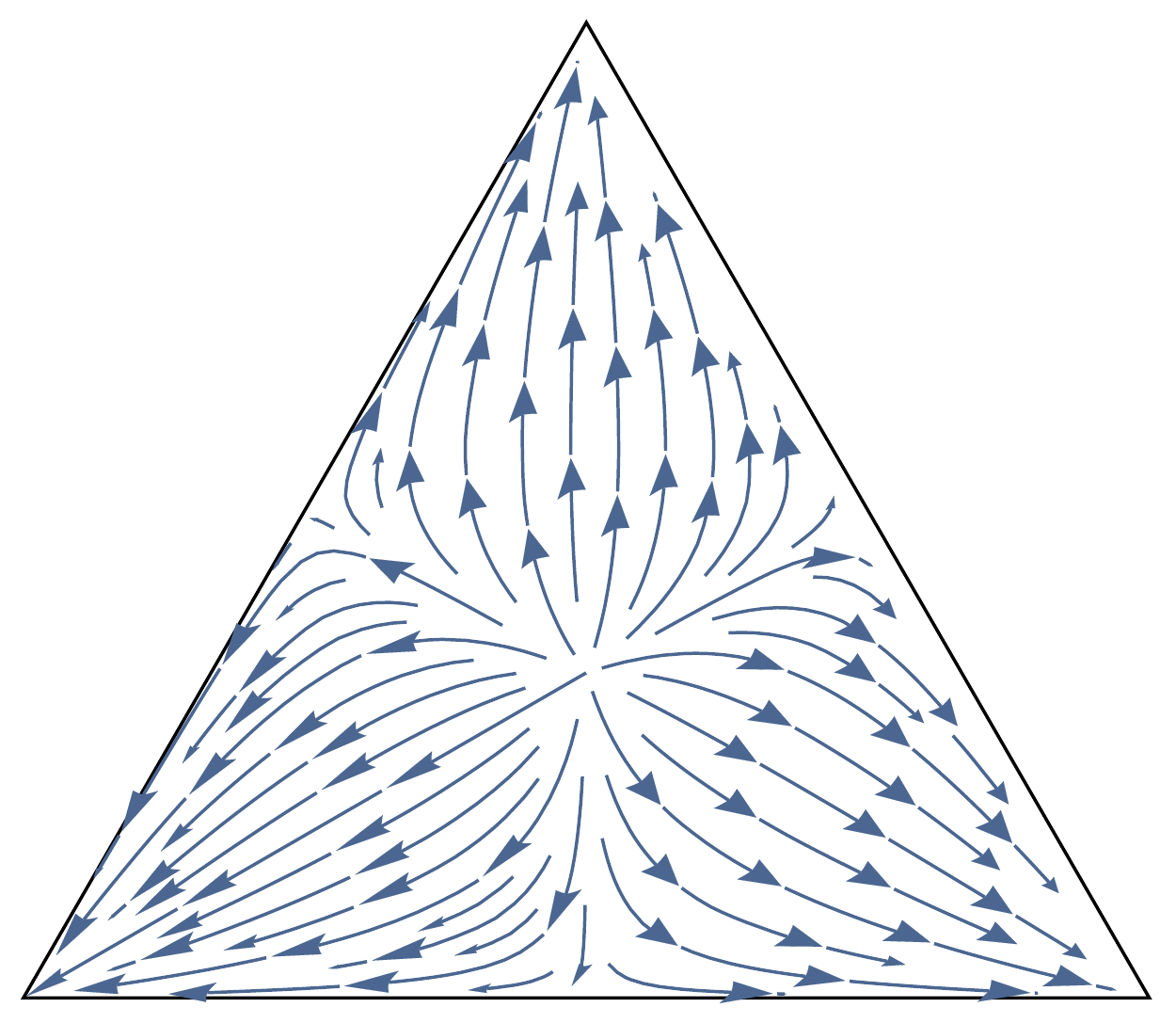} &
\includegraphics[width=0.2\textwidth]{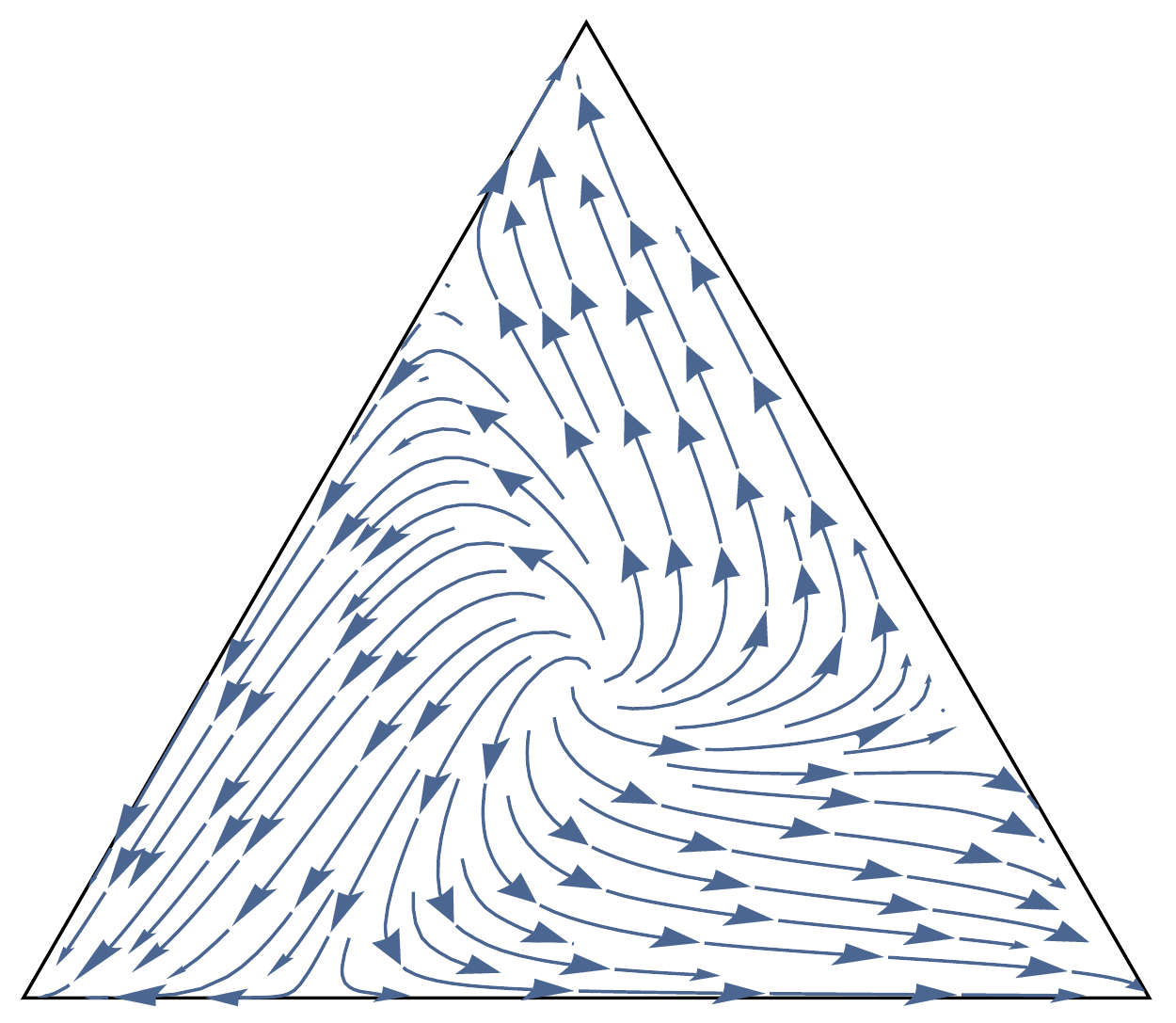} &
\includegraphics[width=0.2\textwidth]{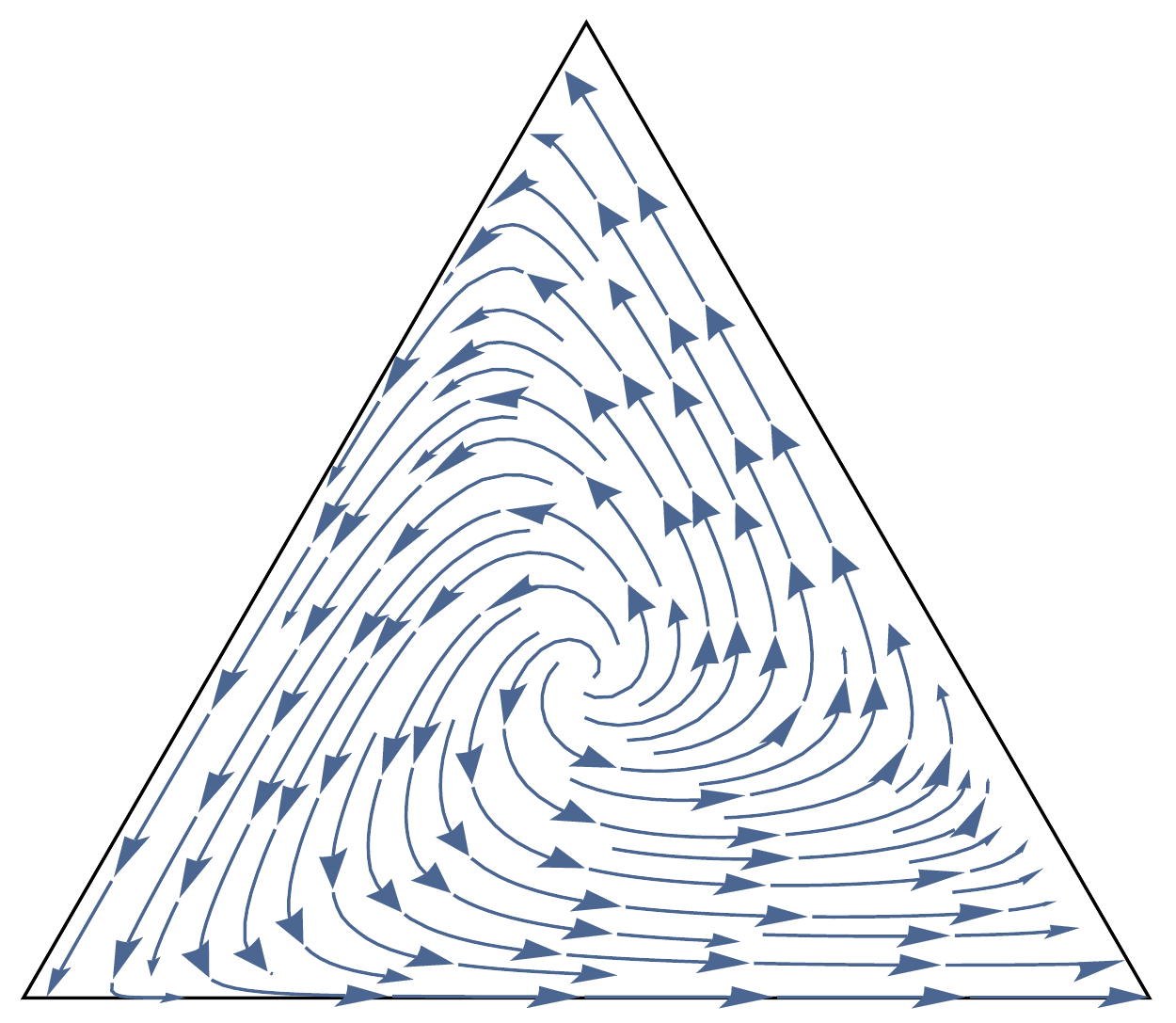} &
\includegraphics[width=0.2\textwidth]{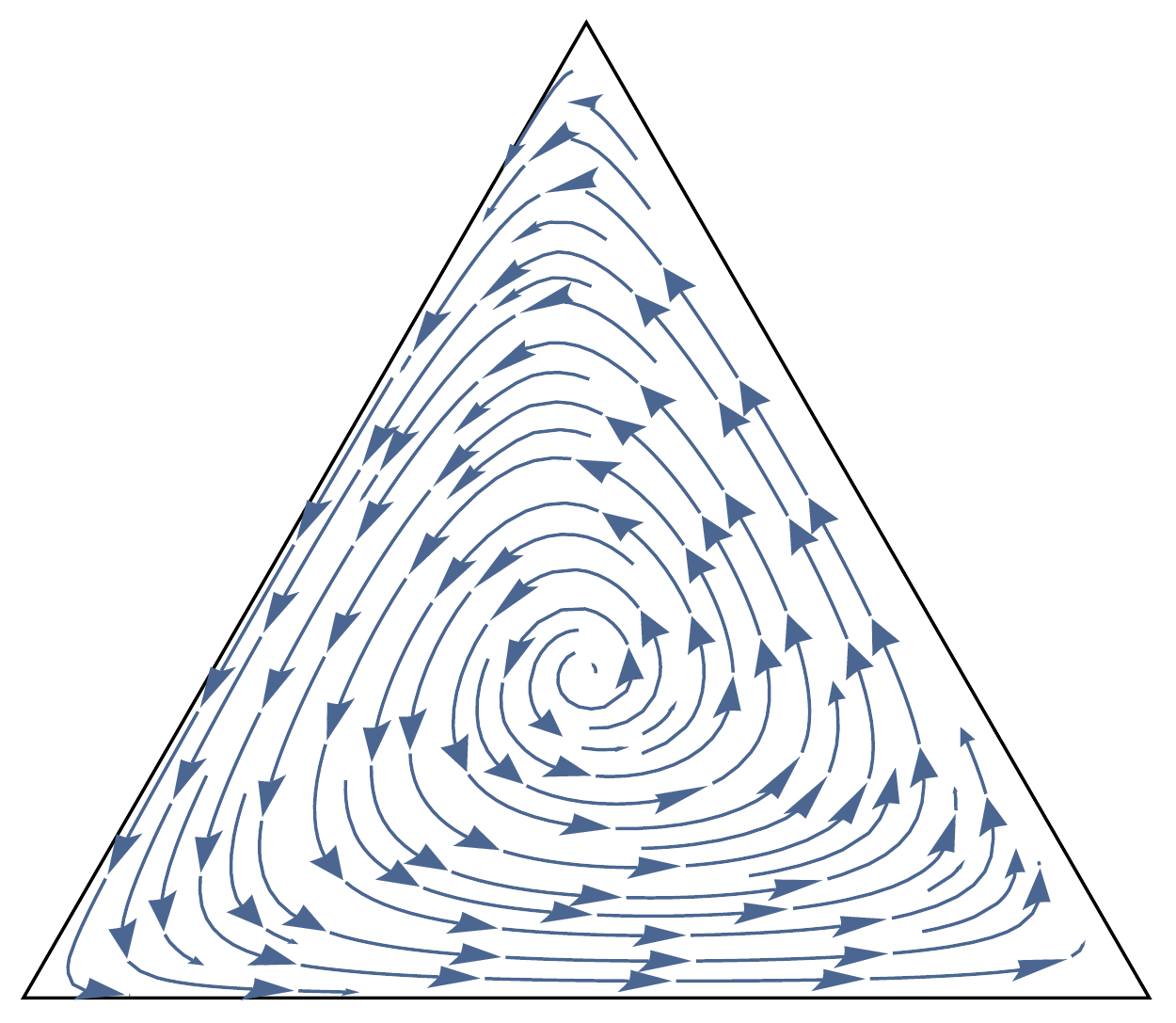} \\
\textbf{(d)} $\alpha > 0$, $\gamma = 0$ &
\textbf{(e)} $\alpha > |\gamma| > 0$ &
\textbf{(f)} $\alpha = |\gamma| > 0$ &
\textbf{(g)} $0 < \alpha < |\gamma|$
\end{tabular}
\end{center}
\caption[Phase portraits for the flows $\dot{p} = R_p(\Omega p)$ of Example~\ref{ex:non-convergent_s-flow}.]{%
\textbf{Phase portraits for the flows \boldmath$\dot{p} = R_p(\Omega p)$ of Example~\ref{ex:non-convergent_s-flow}.}
$\Omega$ is parameterized as specified by~\eqref{eq:Omega3-decomp}.
Parameter $\alpha$ controls whether the flow evolves towards the barycenter ($\alpha < 0$) as in \textbf{(a)} and \textbf{(b)}, or towards the boundary of the simplex ($\alpha > 0$) as in \textbf{(d)}-\textbf{(g)}.
Parameter $\gamma$ controls the rotational component of the flow.
In \textbf{(c)}, the flow neither evolves towards the barycenter nor towards the boundary, and the rotational component of the flow causes periodic orbits.
If $\alpha > 0$, then the convergence of the flow depends on the size of $\gamma$.
If $0 \leq |\gamma| < \alpha$ as in \textbf{(d)} and \textbf{(e)}, then the flow converges to a point on the boundary.
If $|\gamma| \geq \alpha$ as in \textbf{(f)} and \textbf{(g)}, then the flow spirals towards the boundary without converging to a single point.} \label{fig:non-convergent}
\end{figure}
Example~\ref{ex:non-convergent_s-flow} is devoted to the $S$-flow~\eqref{eq:prop_sflow} that parametrizes the assignment flow~\eqref{eq:prop_wflow}, as specified by Proposition~\ref{prop:W-from-S}. The following examples illustrate how the assignment flow may behave if the $S$-flow does not converge to an equilibrium point.
\begin{example} \label{ex:non-convergent_w-flow}
This example continues Example~\ref{ex:non-convergent_s-flow}. Accordingly, we consider the case $n=3$ and assume $\Omega \in \mc{D}$. Let the distance matrix $D$, whose row vectors define the mappings~\eqref{eq:def-Si} corresponding to the assignment flow, be given by
\begin{equation} \label{eq:input_nonconvergent_D}
D = \begin{pmatrix} 0 & 1 & 1 \\ 1 & 0 & 1 \\ 1 & 1 & 0 \end{pmatrix}.
\end{equation}
Then, if $\Omega \in \mathcal{D}$, the initial value $S(0) = \exp_{\baryW}(-\Omega D)$ of the $S$-flow~\eqref{eq:prop_sflow} lies in $\mathcal{D}$ as well.
Hence, the above observations of Example~\ref{ex:non-convergent_s-flow} for the $S$-flow hold.
The resulting assignment flow $t \mapsto W(t)$ then also evolves in $\mathcal{D}$ which can be verified using~\eqref{eq:W-from-S}.
As for the averaging parameters $\Omega$, we consider the following three matrices in $\mathcal{D}$:
\begin{equation} \label{eq:input_nonconvergent_Omega}
\Omega_{\text{center}} = \begin{pmatrix} 0 & 0 & 1 \\ 1 & 0 & 0 \\ 0 & 1 & 0 \end{pmatrix}, \quad
\Omega_{\text{cycle}} = \frac{1}{3} \begin{pmatrix} 1 & 0 & 2 \\ 2 & 1 & 0 \\ 0 & 2 & 1 \end{pmatrix}, \quad
\Omega_{\text{spiral}} = \frac{1}{5} \begin{pmatrix} 2 & 0 & 3 \\ 3 & 2 & 0 \\ 0 & 3 & 2 \end{pmatrix}.
\end{equation}
Figure~\ref{fig:non-convergentW} displays the trajectories of the assignment flow for these averaging matrices.
The symmetry of these plots results from $W(t) \in \mathcal{D}$.

Matrix $\Omega_{\text{center}}$ corresponds to the parameters $(\alpha,\beta,\gamma) = (-\tfrac{1}{2},\tfrac{3}{2},\tfrac{1}{2})$ of~\eqref{eq:Omega3-decomp}, for which the $S$-flow converges to the barycenter.
As a consequence, $W(t)$ converges to a point in $\mathcal{W} \setminus \{ \baryW \}$.

Matrix $\Omega_{\text{cycle}}$ corresponds to the parameters $(\alpha,\beta,\gamma) = (0,1,\tfrac{1}{3})$, for which the $S$-flow has periodic orbits.
Since these orbits are symmetric around the barycenter, i.e.\ $\int_{0}^{t_1} \big( S(t) - \baryW \big) \D t = 0$ with $t_1$ being the period of the trajectory, the trajectory $t \mapsto W(t)$ is also periodic as a consequence of equation~\eqref{eq:W-from-S}.

Finally, matrix $\Omega_{\text{spiral}}$ corresponds to the parameters $(\alpha,\beta,\gamma) = (0.1,0.9,0.3)$, for which
the $S$-flow spirals towards the boundary of the simplex. It is not clear a priori if $t \mapsto W(t)$ does not converge to a single point either.
The trajectory of $W(t)$ shown by Figure~\ref{fig:non-convergentW} suggests that the assignment flow also spirals towards the boundary of the simplex without converging to a single point.
\end{example}
\begin{figure}
\begin{center}
\begin{tabular}{ccc}
\includegraphics[width=0.29\textwidth]{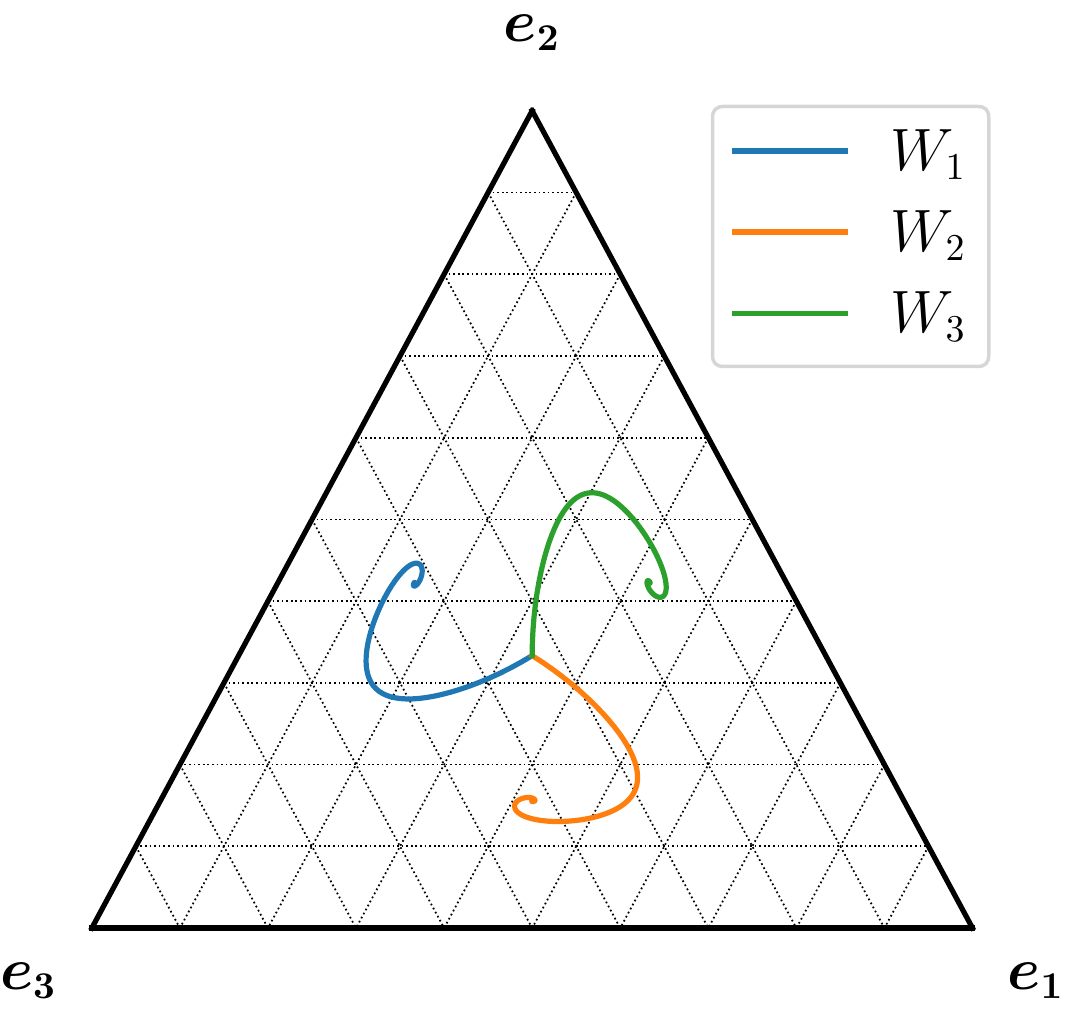} &
\includegraphics[width=0.29\textwidth]{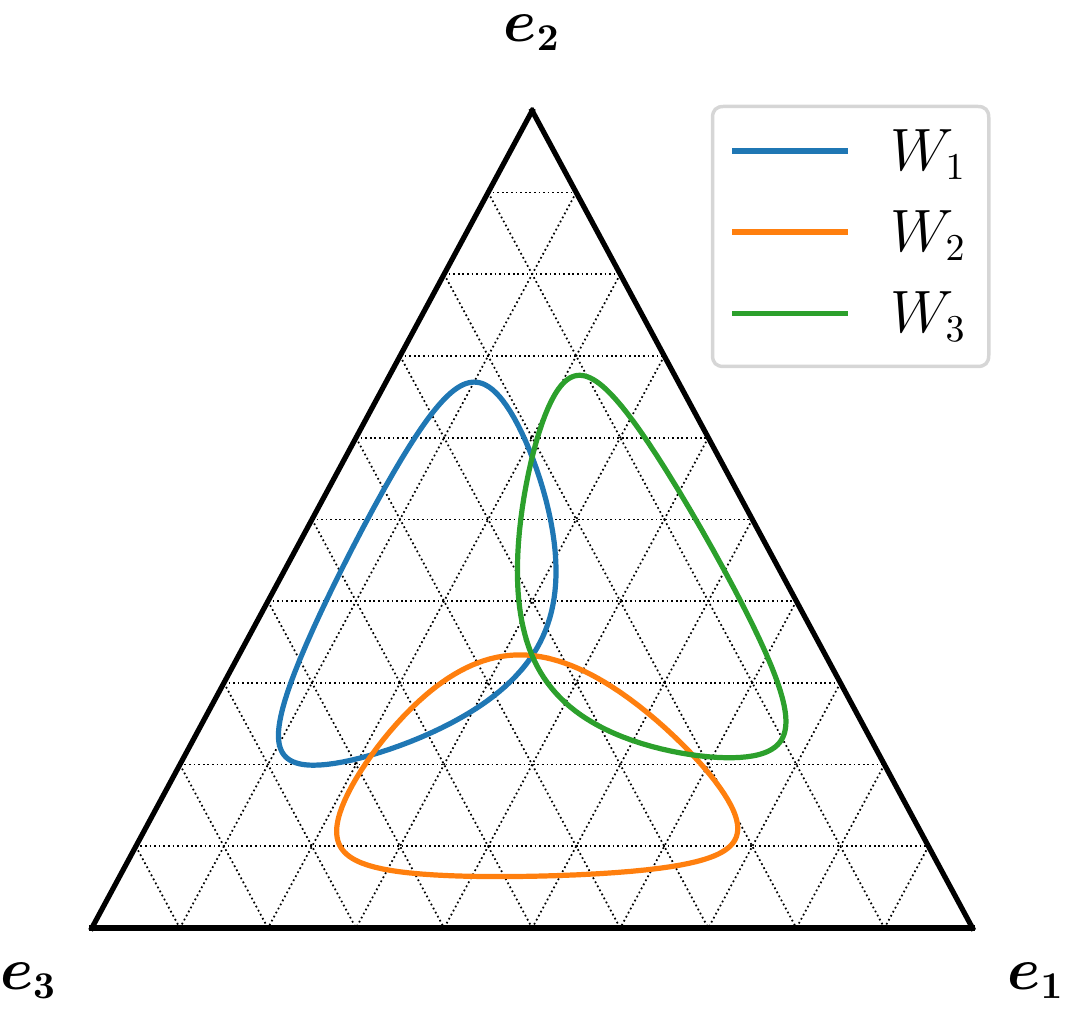} &
\includegraphics[width=0.29\textwidth]{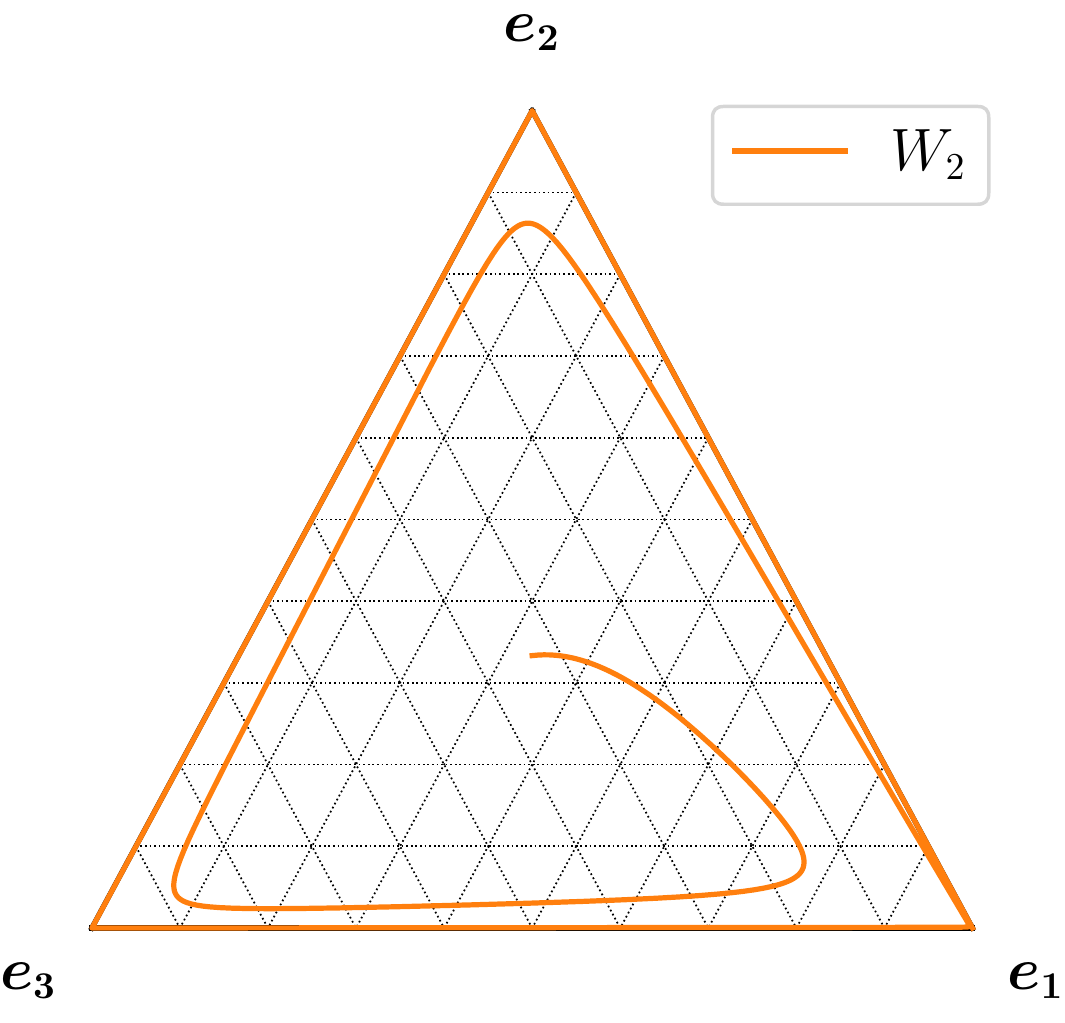} \\
$\Omega_{\text{center}}$ & $\Omega_{\text{cycle}}$ & $\Omega_{\text{spiral}}$
\end{tabular}
\end{center}
\caption[Trajectories of the assignment flows in example~\ref{ex:non-convergent_w-flow}.]{%
\textbf{Trajectories of the assignment flows in example~\ref{ex:non-convergent_w-flow}.}
The input data is given in~\eqref{eq:input_nonconvergent_D} and~\eqref{eq:input_nonconvergent_Omega}.
The flow for the matrix $\Omega_{\text{center}}$ converges to a point in the interior of the assignment manifold.
This limit point differs from the barycenter.
The trajectory for the averaging matrix $\Omega_{\text{cycle}}$ is a closed curve.
The trajectory for $\Omega_{\text{spiral}}$ is spiraling towards the boundary of the simplex.
For the sake of clarity, the trajectory of only one data point is plotted for $\Omega_{\text{spiral}}$.
The trajectories for the other data points can be obtained from that one by permuting the label indices.} \label{fig:non-convergentW}
\end{figure}

\begin{remark}\label{rem:discussion-convergence}
Examples~\ref{ex:non-convergent_s-flow} and~\ref{ex:non-convergent_w-flow} considered the special case $m=|I| = |J|=n=3$. We observed in further experiments similar behaviors also in the case $|J| < |I|$.
For example, it can be verified, for $|J|=2$ and $\Omega = \Omega_{\text{cycle}}$ from Example~\ref{ex:non-convergent_w-flow}, that the $S$-flow possesses a (unstable) limit cycle, i.e.\ a periodic orbit.

The above examples also demonstrate that several symmetries in the input data are required, e.g.\ $\Omega \in \mathcal{D}$ and $S_0 \in \mathcal{D}$, in order to obtain nonconvergent orbits. However, small perturbations like numerical errors or the omnipresent noise in \textit{real} data will break these symmetries. Therefore, it is very unlikely to observe such behavior of the $S$-flow and the assignment flow, respectively, in practice.
\end{remark}
\subsubsection{Geometric Averaging and Spatial Shape}\label{sec:ex-termination}
We design and construct a small academical example that, despite its simplicity, illustrates the following important points:
\begin{itemize}
\item
the region of attraction due to Corollary~\ref{prop:eps_unif}, here for the special case of uniform averaging parameters $\Omega$ (and likewise more generally for nonuniform $\Omega$ (Proposition~\ref{prop:eps_est})), that enables to terminate the numerical scheme and rounding to the \textit{correct} labeling;
\item
the influence of $\Omega$ on the spatial shape of patterns created through data labeling, which provides the basis for pixel-accurate `semantic' image labeling;
\item
undesired asymptotic behavior of the \textit{numerically} integrated assignment
flow---cf.~Remark~\ref{rem:Steidl} below---cannot occur when using proper geometric numerical integration, like the scheme~\eqref{eq:Euler-scheme} or any scheme devised by~\cite{Zeilmann:2018aa}.
\end{itemize}
\begin{example}\label{ex:stability}
We consider a $12 \times 12$ RGB image $u \colon I \rightarrow [0,1]^3$ shown by Figure~\ref{fig:stability-example}.
\begin{figure}[htbp]
\begin{center}
\begin{tabular}{cc}
\includegraphics[width=0.2\textwidth]{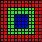} &
\includegraphics[width=0.2\textwidth]{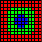} \\
input image & $S^*$
\end{tabular}
\end{center}
\caption[Illustration of input and output of example~\ref{ex:stability}.]{%
\textbf{Illustration of input and output of Example~\ref{ex:stability}.} The input image consisting of three colors, which was used for computing the distance matrix $D$, is shown on the left.
This distance matrix was used to initialize the $S$-flow, whose limit is illustrated by the image on the right. This is a minimal example that demonstrates how stability conditions~\eqref{eq:stable} constrain spatial shape.

} \label{fig:stability-example}
\end{figure}
The three unit vectors $e_{j},\,j \in J=[3]$ define the labels that are marked by the colors red, green and blue.
For spatial regularization we used $3 \times 3$ neighborhoods $\mc{N}_{i},\,i \in I$ with uniform weights $\omega_{ik} = \tfrac{1}{|\mathcal{N}_i|},\, k \in \mc{N}_{i}$, with shrunken neighborhoods if they intersect the boundary of the underlying quadratic domain.
The distance matrix $D$ that initializes the $S$-flow by $S_0 = \exp_{\baryW}(-\Omega D)$, was set to $D_{ij} = 10 \cdot \| u_i - e_j \|_2,\,i \in I,\,j \in J$.

Adopting the termination criterion from~\cite{Astroem2017}, we numerically integrated the $S$-flow using the scheme~\eqref{eq:Euler-scheme}, until iteration $T$ when the average entropy dropped below $10^{-3}$, i.e.
\begin{equation}\label{eq:termination-entropy}
-\frac{1}{|I| \log |J|} \sum_{i\in I,j\in J} S_{ij}^{(T)} \log S_{ij}^{(T)} < 10^{-3}.
\end{equation}
The resulting assignment $S^{(T)}$ was rounded to the integral assignment $S^* \in \Wstar$ depicted by the right panel of Figure~\ref{fig:stability-example}. We observe the following.
\begin{enumerate}[(i)]
\item
The resulting labeling $S^{\ast}$ differs from the input image although \textit{exact} (integral) input data are used.

This conforms to Corollary~\ref{cor:stability-S}(b), which enables to recognize the input data as \textit{unstable}. As a consequence, the green and blue labels at the corners of the corresponding quadrilateral shapes in the input data are replaced by the flow. The resulting labeling $S^{\ast}$ is stable, as one easily verifies using Corollary~\ref{cor:stability-S}(a).

This simple example and the corresponding observation points to a fundamental question to be investigated in future work: how can $\Omega$ be used for `storing' prior knowledge about the shape of labeling patterns?
\item
Using the estimate~\eqref{eq:eps_unif} which is the special case of~\eqref{eq:eps_est} in the case of uniform weights,
we computed
\begin{equation}
\varepsilon_{\mathrm{est}} = \varepsilon_{\mathrm{unif}} = 0.2.
\end{equation}
Since the distance between $S^*$ and the assignment $S^{(T)}$ obtained after terminating numerical integration due to~\eqref{eq:termination-entropy}, satisfied
\begin{equation}
\max_{i \in I}~\| S_i^{(T)} - S_i^* \|_1 \approx 0.00196 < \varepsilon_{\mathrm{est}},
\end{equation}
we had the \textit{guarantee} due to Proposition~\ref{prop:region-attraction-Euler} that $S^{(t)}$ converges for $t > T$ to $S^*$, i.e.\ that no label indicated by $S^{(T)}$ can change anymore.
With regard to Proposition~\ref{prop:approx-discr-euler}, the estimate \eqref{eq:approx-discr-euler} implies for sufficiently small step size ${h > 0}$ that the continuous $S$-flow $S(hT)$ also lies in the attracting region $B_{\varepsilon}(S^*)$. 
Proposition~\ref{prop:attraction} then states the convergence of the $S$-flow to $S^*$. 
Eventually, the continuous assignment flow~\eqref{eq:AF} converge to $S^*$ by Proposition~\ref{prop:limit-W}.
\end{enumerate}
\end{example}
\begin{remark}[numerical integration and asymptotic behavior]\label{rem:Steidl}
The authors of~\cite{Astroem2017} adopted a numerical scheme from~\cite{losert1983dynamics} which, when adapted and applied to~\eqref{eq:AF-i}, was shown in~\cite{bergmann2017iterative} to always converge to a constant solution as $t \to \infty$, i.e.\ a \textit{single} label is assigned to every pixel, which clearly is an unfavorable property. Irrespective of the fact that uniform positive weights were used by~\cite{Astroem2017}, that satisfy assumption~\eqref{eq:sym-Omega}, this strange asymptotic behavior resulted from the fact that the adaption of the discrete scheme of~\cite{losert1983dynamics} implicitly uses \textit{different} step sizes for updating the flow $S_{i}$ at different locations $i \in I$.

Our results in this paper show that the continuous-time assignment flow does not exhibit this asymptotic behavior, under appropriate assumptions on the parameter matrix $\Omega$. In addition, point (ii) above and Proposition~\ref{prop:region-attraction-Euler} show that using a proper geometric scheme from~\cite{Zeilmann:2018aa} turns condition~\eqref{eq:termination-entropy} into a sound criterion for terminating the numerical scheme, followed by safe rounding to an integral labeling.
\end{remark}
%


%% file: Conclusion.tex

\section{Conclusion}\label{sec:Conclusion}
We established in this paper that under reasonable assumptions on the weight parameters $\Omega$, the assignment flow approach is a sound method for contextual data classification on graphs. Favourable properties like convergence to integral assignments and existence of corresponding basins of attraction extend to sequences generated by discrete-time schemes for geometric integration. This shows that geometric numerical integration of the assignment flow yields sound numerical algorithms. A range of counter-examples demonstrate that these conditions are not too strong, since violating them may quickly lead to unfavorable behavior of the assignment flow regarding classification.

The results provide a proper basis and justify recent work on learning the assignment flow parameters $\Omega$ from data \cite{Huhnerbein:2021th,Zeilmann:2021wt,Zeilmann:2021ul}, on extending the approach to unsupervised data classification on graphs \cite{Zern:2020ab,Zisler:2020aa} or taking additional spatial constraints into account \cite{Sitenko:2021vu}. Our future work will focus on deeper parametrizations of assignment flows using the same mathematical framework and on studying their properties and performance for statistical data classification on graphs.

%% file: Appendix.tex

\section{Proofs}\label{sec:Proofs}

\subsection{Proof of Proposition~\ref{prop:limit-W}}
\begin{proof}
\begin{enumerate}[(a)]
\item Let $\beta_i \coloneqq \tfrac{1}{2} \min \{S_{i j^*(i)}^* - S_{ij}^*\}_{j \neq j^*(i)} > 0$. Since
\begin{equation}
\lim_{t\to\infty} S_{i j^*(i)}(t) - S_{ij}(t) = S_{i j^*(i)}^* - S_{ij}^* \geq 2 \beta_i > 0, \quad \forall j \in J \setminus \{ j^*(i) \},
\end{equation}
there exists $t_1 \geq 0$ such that
\begin{equation}
S_{i j^*(i)}(t) - S_{ij}(t) > \beta_i, \quad \forall t \geq t_1, \quad \forall j \in J \setminus \{ j^*(i) \}.
\end{equation}
We estimate
\begin{subequations}
\begin{align}
\| W_i(t) &- e_{j^*(i)}\|_1
\\
	&= 1 - W_{i j^*(i)} + \sum_{j \neq j^*(i)} W_{ij}
	= 2 - 2 W_{i j^*(i)}
	\overset{\eqref{eq:W-from-S}}{=}
2 - 2 \frac{\exp\Big( \int_{0}^{t} S_{i j^*(i)}(\tau) \D\tau \Big)}{\sum_{j \in J} \exp\Big( \int_{0}^{t} S_{ij}(\tau) \D\tau \Big)} \\
	&= 2 \frac{\sum_{j \neq j^*(i)} \exp\Big( \int_{0}^{t} S_{ij}(\tau) \D\tau \Big)}{\sum_{j \in J} \exp\Big( \int_{0}^{t} S_{ij}(\tau) \D\tau \Big)} \\
	&= 2 \frac{\sum_{j \neq j^*(i)} \exp\Big( \int_{0}^{t} \big( S_{ij}(\tau) - S_{i j^*(i)}(\tau) \big) \D\tau \Big)}{1 + \sum_{j \neq j^*(i)} \exp\Big( \int_{0}^{t} \big( S_{ij}(\tau) - S_{i j^*(i)}(\tau) \big) \D\tau \Big)} \\
	&\leq 2 \sum_{j \neq j^*(i)} \exp\Big( \int_{0}^{t} \big( S_{ij}(\tau) - S_{i j^*(i)}(\tau) \big) \D\tau \Big) \\
	&= 2 \sum_{j \neq j^*(i)} \exp\Big( \int_{0}^{t_1} \big( S_{ij}(\tau) - S_{i j^*(i)}(\tau) \big) \D\tau + \int_{t_1}^{t} \big( \overbrace{S_{ij}(\tau) - S_{i j^*(i)}(\tau)}^{< - \beta_i} \big) \D\tau \Big) \\
	&\leq 2 \sum_{j \neq j^*(i)} \exp\Big( \int_{0}^{t_1} \big( S_{ij}(\tau) - S_{i j^*(i)}(\tau) \big) \D\tau \Big) \cdot e^{-\beta_i (t - t_1)} \\
	&= \underbrace{2 e^{\beta_i t_1} \sum_{j \neq j^*(i)} \exp\Big( \int_{0}^{t_1} \big( S_{ij}(\tau) - S_{i j^*(i)}(\tau) \big) \D\tau \Big)}_{\eqqcolon \alpha_i > 0} \cdot e^{-\beta_i t},
\end{align}
\end{subequations}
which proves~\eqref{eq:prop-limit-W-a}.
\item Let $J^*(i) \coloneqq \argmax_{j \in J}~S_{ij}^*$. For any $j, l \in J^*(i)$, we have
\begin{align}
\begin{split}
\int_{0}^{\infty} \big| S_{ij}(t) - S_{il}(t) \big| \D t 
	&\leq \int_{0}^{\infty} \big| S_{ij}(t) - S_{ij}^* \big| \D t + \int_{0}^{\infty} \big| S_{il}(t) - S_{il}^* \big| \D t \\
	&\leq 2 \int_{0}^{\infty} \| S_i(t) - S_i^* \|_1 \D t < \infty,
\end{split}
\end{align}
where the last inequality follows from the hypothesis of~\eqref{eq:S-flow-fast-concergence}.
Thus, the improper integral \linebreak $\int_{0}^{\infty} {\big( S_{ij}(t) - S_{il}(t) \big)} \D t \in \R$ exists.

If $j \in J^*(i)$, we obtain
\begin{subequations}
\begin{align}
W_{ij}(t)
	&\overset{\eqref{eq:W-from-S}}{=} \frac{\exp\Big( \int_{0}^{t} S_{ij}(\tau) \D\tau \Big)}{\sum_{l \in J} \exp\Big( \int_{0}^{t} S_{il}(\tau) \D\tau \Big)} \\
\begin{split}
	&= \Bigg( 1 + \sum_{l \in J \setminus J^*(i)} \exp\Big( \overbrace{ \int_{0}^{t} \big( S_{il}(\tau) - S_{ij}(\tau) \big) \D\tau }^{\rightarrow -\infty} \Big) \\
	&\qquad\qquad + \sum_{l \in J^*(i) \setminus \{ j \}} \exp\Big( \int_{0}^{t} \big( S_{il}(\tau) - S_{ij}(\tau) \big) \D\tau \Big) \Bigg)^{-1}
\end{split} \\
	&\longrightarrow \Bigg( 1 + \sum_{l \in J^*(i) \setminus \{ j \}} \exp\Big( \int_{0}^{\infty} \big( S_{il}(\tau) - S_{ij}(\tau) \big) \D\tau \Big) \Bigg)^{-1} \in (0,1] \quad \text{for} \quad t \rightarrow \infty,
\end{align}
\end{subequations}
whereas for any $j \in J \setminus J^*(i)$
\begin{subequations}
\begin{align}
W_{ij}(t)
	&= \frac{\exp\Big( \int_{0}^{t} S_{ij}(\tau) \D\tau \Big)}{\sum_{l \in J} \exp\Big( \int_{0}^{t} S_{il}(\tau) \D\tau \Big)} \\
\begin{split}
	&= \Bigg( \overbrace{\sum_{l \in J \setminus J^*(i)} \exp\Big( \int_{0}^{t} \big( S_{il}(\tau) - S_{ij}(\tau) \big) \D\tau \Big)}^{\geq 0} \\
	&\qquad\qquad + \sum_{l \in J^*(i)} \exp\Big( \underbrace{\int_{0}^{t} \big( S_{il}(\tau) - S_{ij}(\tau) \big) \D\tau}_{\rightarrow \infty} \Big) \Bigg)^{-1} 
\end{split} \\
	&\longrightarrow 0 \quad \text{for} \quad t \rightarrow \infty.
\end{align}
\end{subequations} \qed
\end{enumerate}
\end{proof}
\subsection{Proof of Proposition~\ref{prop:Jacobian-eigen}}
\begin{proof}
\begin{enumerate}[(a)]
\item Since $\sigma\big( \tfrac{\partial F}{\partial S}(S^*)^\top \big) = \sigma\big( \tfrac{\partial F}{\partial S}(S^*) \big)$, we may alternatively regard the transpose of the Jacobian
\begin{equation}
\tfrac{\partial F}{\partial S}(S^*)^\top = \begin{pmatrix} B_1^\top & & \\ & \ddots & \\ & & B_{m}^\top \end{pmatrix} + \Omega^\top \otimes I_{n} \cdot \begin{pmatrix} R_{S_{1}^*} & & \\ & \ddots & \\ & & R_{S_{m}^*} \end{pmatrix}
\end{equation}
with $B_i^\top = \Diag\big( (\Omega S^*)_i \big) - \langle S_i^*, (\Omega S^*)_i \rangle I_{n} - (\Omega S^*)_i {S_i^*}^\top$. We have for each $i \in I$,
\begin{subequations}
\begin{alignat}{3}
B_i^\top \eins_{n} &= -\langle S_i^*, (\Omega S^*)_i \rangle \eins_{n}, \quad & R_{S_i^*} \eins_{n} &= 0, \\
B_i^\top e_j &= \big( (\Omega S^*)_{ij} - \langle S_i^*, (\Omega S^*)_i \rangle \big) e_j, \quad & R_{S_i^*} e_j &= 0, \quad \forall j &\in J \setminus \supp(S_i^*).
\end{alignat}
\end{subequations}
Hence, the transposed Jacobian possesses the following eigenpairs:
\begin{subequations} \label{eq:eigenpairsJT-corner}
\begin{align}
\tfrac{\partial F}{\partial S}(S^*)^\top \cdot e_i \otimes \eins_{n} &= -\langle S_i^*, (\Omega S^*)_i \rangle \cdot e_i \otimes \eins_{n}, \quad \forall i \in I, \\
\tfrac{\partial F}{\partial S}(S^*)^\top \cdot e_i \otimes e_j &= \big( (\Omega S^*)_{i j} -\langle S_i^*, (\Omega S^*)_i \rangle \big) \cdot e_i \otimes e_j, \quad \forall j \in J \setminus \supp(S_i^*), \quad \forall i \in I.
\end{align}
\end{subequations}
If $S^* \in \Wstar$, then $|\supp (S_i^*)| = 1$ for each $i \in I$ and therefore~\eqref{eq:eigenpairsJT-corner} specifies all $m n$ eigenpairs and the entire spectrum, which proves~\eqref{eq:spectrum-a}. In this case, the eigenvectors of $\tfrac{\partial F}{\partial S}(S^*)$ can also be stated explicitly: Since $R_{S^*} = 0$, we have
\begin{equation}
\tfrac{\partial F}{\partial S}(S^*) = \begin{pmatrix} B_1 & & \\ & \ddots & \\ & & B_{m} \end{pmatrix}.
\end{equation}
Each block $B_i$ fulfills
\begin{subequations}
\begin{align}
B_i S_i^* &= - \langle S_i^*, (\Omega S^*)_i \rangle S_i^*, \\
B_i (S_i^* - e_j) &= \big( (\Omega S^*)_{ij} - \langle S_i^*, (\Omega S^*)_i \rangle \big) (S_i^* - e_j) \quad \forall j \in J \setminus \supp(S_i^*).
\end{align}
\end{subequations}
Hence, the corresponding eigenvectors of $\tfrac{\partial F}{\partial S}(S^*)$ are
\begin{equation} 
e_i \otimes S_i^*, \quad e_i \otimes (S_i^* - e_j), \quad \forall j \in J \setminus \supp(S_i^*), \quad \forall i \in I.
\end{equation}
\item Since $\Omega S^* = \tfrac{1}{|J_{+}|} (\Omega \eins_{m}) \eins_{J_{+}}^\top$ for $S^*= \tfrac{1}{|J_{+}|} \eins_{m} \eins_{J_{+}}^\top$, we have
\begin{align}
B_i &= ( \Omega \eins_{m} )_{i} \cdot \Big( \tfrac{1}{|J_{+}|} \Diag(\eins_{J_{+}}) - \tfrac{1}{|J_{+}|} I_{n} -\tfrac{1}{|J_{+}|^2} \eins_{J_{+}} \eins_{J_{+}}^\top \Big), \\
R_{S_i^*} &= \tfrac{1}{|J_{+}|} \Diag(\eins_{J_{+}}) - \tfrac{1}{|J_{+}|^2} \eins_{J_{+}} \eins_{J_{+}}^\top
\end{align}
for all $i \in I$, i.e., the Jacobian matrix simplifies to
\begin{multline}
\tfrac{\partial F}{\partial S}(S^*) = \Diag(\Omega \eins_{m}) \otimes B_0 + \Omega \otimes R_{S_1} \\
	\text{with} \quad B_0 = \tfrac{1}{|J_{+}|} \Diag(\eins_{J_{+}}) - \tfrac{1}{|J_{+}|} I_{n} -\tfrac{1}{|J_{+}|^2} \eins_{J_{+}} \eins_{J_{+}}^\top.
\end{multline}
Let $\{ (\lambda_i, w_i) \}_{i \in \tilde{I}} \subset \C \times \C^{m}$ be the set of all eigenpairs of $\Omega$ indexed by $\widetilde I$, and let $\{ v_{1}, \dots, v_{|J_{+}|-1} \}$ be a basis of $\big\{ v \in \R^{n} \colon \langle v, \eins_{J_{+}} \rangle = 0,\ \supp(v) \subseteq J_{+} \big\}$. Note that $|\widetilde{I}| < m$ if and only if $\Omega$ is not diagonalizable. A short calculation shows
\begin{subequations}
\begin{alignat}{3}
B_0 e_j &= -\tfrac{1}{|J_{+}|} e_j, \quad &R_{S_1} e_j &= 0, \quad &&\forall j \in J \setminus J_{+}, \\
B_0 \eins_{J_{+}} &= -\tfrac{1}{|J_{+}|} \eins_{J_{+}}, \quad &R_{S_1} \eins_{J_{+}} &= 0, \quad && \\
B_0 v_j &=  0, \quad &R_{S_1} v_j &= \tfrac{1}{|J_{+}|} v_j, \quad &&\forall j \in \{ 1, \dots, |J_{+}|-1 \}.
\end{alignat}
\end{subequations}
Hence, the Jacobian has the following $m n - (m-|\widetilde{I}|)(|J_{+}|-1)$ eigenpairs:
\begin{subequations}
\begin{align}
\Big( -\tfrac{( \Omega \eins_{m} )_{i}}{|J_{+}|}, e_i \otimes e_j \Big),& \quad \forall j \in J \setminus J_{+}, \quad \forall i \in I, \\
\Big( -\tfrac{( \Omega \eins_{m} )_{i}}{|J_{+}|}, e_i \otimes \eins_{J_{+}} \Big),& \quad \forall i \in I, \\
\Big( \tfrac{\lambda_i}{|J_{+}|}, w_i \otimes v_j \Big),& \quad \forall j \in \{1, \dots, |J_{+}|-1 \}, \quad \forall i \in \widetilde{I}.
\end{align}
\end{subequations}
If $|\widetilde{I}| = m$, we thus have a complete set of $m n$ eigenpairs. If $|\tilde{I}| < m$, we may consider a diagonalizable perturbation $\widetilde{\Omega}$ of $\Omega$. 
By the same argument, we get a complete set of eigenpairs for the perturbed Jacobian matrix. Consequently, we obtain~\eqref{eq:spectrum-bary} by continuity of the spectrum.
\item We show that the real and imaginary parts of the corresponding eigenvector lie in the linear subspace
\begin{equation}
\mathcal{T}_{+} = \mathcal{T}_+(S^*) = \big\{ V \in \R^{m n} \colon \langle V_i, \eins_{n} \rangle = 0,\ \supp(V_i) \subseteq \supp(S_i^*), \; \forall i \in I \big\}
\end{equation}
To this end, we show the two inclusions
\begin{equation}\label{eq:inclusion_R-T-B}
\im R_{S^*} \subseteq \mathcal{T}_{+} \subseteq \ker B,
\end{equation}
where $R_{S^*}$ and $B$ denote the block diagonal matrices
\begin{equation}
B = \begin{pmatrix} B_1 & & \\ & \ddots & \\ & & B_{m} \end{pmatrix}, \quad R_{S^*} = \begin{pmatrix} R_{S_1^*} & & \\ & \ddots & \\ & & R_{S_{m}^*} \end{pmatrix}.
\end{equation}
As for the first inclusion, we use the orthogonal projection onto $\mathcal{T}_{+}$ given by
\begin{multline}
\Pi_{\mathcal{T}_{+}} = \begin{pmatrix} \Pi_{\mathcal{T}_{+},1} & & \\ & \ddots & \\ & & \Pi_{\mathcal{T}_{+},m} \end{pmatrix} \\ \text{with} \quad \Pi_{\mathcal{T}_{+},i} = \Diag( \eins_{J_i} ) - \tfrac{1}{|J_i|} \eins_{J_i} \eins_{J_i}^\top, \quad J_i = \supp(S_i^*) \quad \forall i \in I.
\end{multline}
One can verify that $\Pi_{\mathcal{T}_{+}} R_{S^*} = R_{S^*}$ which implies $\im R_{S^*} \subseteq \im \Pi_{\mathcal{T}_{+}} = \mathcal{T}_{+}$, i.e.\ the first inclusion of~\eqref{eq:inclusion_R-T-B}.

As for the second inclusion, we have to take into account that $S^*$ is an equilibrium point, i.e.\ by~\eqref{eq:equilibrium}
\begin{equation}
(\Omega S^*)_{ij} = \langle S_i^*, (\Omega S^*)_i \rangle \quad \forall j \in \supp(S_i^*) \quad \forall i \in I.
\end{equation}
Since $B$ is a block diagonal matrix, it suffices to examine each block
\begin{equation}
B_i = \Diag\big( (\Omega S^*)_i \big) - \langle S_i^*, (\Omega S^*)_i \rangle I_{n} - S^*_i (\Omega S^*)_i^\top
\end{equation}
separately. Since
\begin{equation}
B_i e_j = (\Omega S^*)_{ij} e_j - \langle S_i^*, (\Omega S^*)_i \rangle e_j - (\Omega S^*)_{ij} S_i^* = -\langle S_i^*, (\Omega S^*)_i \rangle S_i^*, \quad \forall j \in \supp(S_i^*)
\end{equation}
is independent of $j\in \supp(S_i^*)$, we get $B_i v = 0$ for any $v \in \R^{n}$ with $\langle v, \eins_{n} \rangle = 0$ and $\supp(v) \subseteq \supp(S_i^*)$. This verifies the second inclusion of~\eqref{eq:inclusion_R-T-B}.

As a consequence of the two inclusions (\ref{eq:inclusion_R-T-B}), any eigenvector $V$ of $R_{S^*} (\Omega \otimes I_{n})$ corresponding to a nonvanishing eigenvalue $\lambda \neq 0$ has a real and imaginary part lying in $\im R_{S^*} \subseteq \mathcal{T}_{+} \subseteq \ker B$. Therefore, $(\lambda, V)$ is also an eigenpair of $\tfrac{\partial F}{\partial S}(S^*) = B + R_{S^*} (\Omega \otimes I_{n})$.
It remains to show that
\begin{equation}
R_{S^*} (\Omega \otimes I_{n}) = \begin{pmatrix} \omega_{11} R_{S_1^*} & \cdots & \omega_{1 m} R_{S_1^*} \\ \vdots & & \vdots \\ \omega_{m 1} R_{S_{m}^*} & \cdots & \omega_{m m} R_{S_{m}^*} \end{pmatrix}
\end{equation}
has at least one eigenvalue with positive real part. Since the trace
\begin{equation}\label{eq:ProofAtLeastOnePositiveEigenvalue}
\tr\big( R_{S^*} (\Omega \otimes I_{n}) \big) = \sum_{i \in I} \omega_{ii} \tr\big( R_{S_i^*} \big) = \sum_{i \in I} \underbrace{\omega_{ii} \sum_{j \in J} (S_{ij}^* - {S_{ij}^*}^2)}_{\begin{cases} \geq 0, & \forall i \in I \\ > 0, & \text{for some $i \in I$} \end{cases}} > 0
\end{equation}
is positive by assumption, the existence of such an eigenvalue is guaranteed.
\end{enumerate} \qed
\end{proof}
\subsection{Proof of Theorem~\ref{thm:convergence}}

The proof follows after two preparatory Lemmata.
Let $\Lambda \subset \ol{\mathcal{W}}$ be the limit set of the orbit $\{ S(t) \colon t \geq 0 \}$, i.e.
\begin{equation}\label{eq:def-Lambda-critical}
\Lambda = \Lambda(S_0) = \Big\{ S^* \in \ol{\mathcal{W}} \colon \exists (t_k)_{k \in \N} \subset \R_{\geq 0}\;\text{with}\; t_k \rightarrow \infty,\; S(t_k) \rightarrow S^* \Big\}.
\end{equation}
The set $\Lambda\neq\emptyset$ is non-empty since $\ol{\mathcal{W}}$ is compact.
\begin{lemma}
Every point $S^{\ast} \in \Lambda$ of the limit set~\eqref{eq:def-Lambda-critical} is an equilibrium point satisfying the condition of Proposition~\ref{prop:crit-points}(a), which under assumption~\eqref{eq:sym-Omega} reads
\begin{equation}\label{eq:equilibrium-hat}
(\widehat{\Omega} S^*)_{ij} = \langle S_i^*, (\widehat{\Omega} S^*)_i \rangle \quad \forall j \in \supp S_i^* \quad \forall i \in I.
\end{equation}
\end{lemma}
\begin{proof}
The assertion follows from~\cite[Proposition 1]{losert1983dynamics} if the flow $\dot{S} = F(S)$ on $\ol{\mathcal{W}}$ admits a Lyapunov function $f \colon \ol{\mathcal{W}} \rightarrow \R$, i.e.\ $\frac{\D}{\D t}f(S(t)) = \langle \nabla f(S(t)), F(S(t)) \rangle \geq 0$, with equality only at an equilibrium.

The function $f \colon \ol{\mathcal{W}} \rightarrow \R$,
\begin{equation}\label{eq:f-Lyapunov}
f(S) = \langle S, \widehat\Omega S \rangle
\end{equation}
is a Lyapunov function for the $S$-flow~\eqref{eq:def-S-flow-F}, since
\begin{subequations}\label{eq:f-increases}
\begin{align}
\frac{\D}{\D t} f\big(S(t)\big)
&= 2 \langle \widehat\Omega S, \dot{S} \rangle \label{eq:dfSt_a}
	= 2 \sum_{i \in I} \langle (\widehat{\Omega} S)_i, \dot{S}_i \rangle
	\overset{\la\eins_{n},\dot S_{i}\ra = 0}{=}
	2 \sum_{i \in I} \big\langle (\widehat{\Omega} S)_i - \langle S_i, (\widehat{\Omega} S)_i \rangle \eins_{n}, \dot{S}_i \big\rangle \\
	&\overset{\eqref{eq:def-S-flow-F},\eqref{eq:sym-Omega}}{=} \sum_{i \in I} \frac{2}{w_i} \sum_{j \in J} S_{ij} \Big( (\widehat{\Omega} S)_{ij} - \langle S_i, (\widehat{\Omega} S)_i \rangle \Big)^2
	\geq 0,
\end{align}
\end{subequations}
with equality only if $S$ satisfies the equilibrium criterion~\eqref{eq:equilibrium-hat}. \qed
\end{proof}
Next, we introduce some additional notation.
Let $S^* \in \Lambda$ be an equilibrium with $S(t_k) \rightarrow S^*$. The weighted Kullback-Leibler divergence is defined by
\begin{subequations}\label{eq:def-DKL}
\begin{align}
D_{\mathrm{KL}}^w(S^*, S)
	&= \begin{cases} - \sum_{i \in I} w_i \sum_{j \in \supp(S_i^*)} S_{ij}^* \log \tfrac{S_{ij}}{S_{ij}^*}, &\text{if $\supp(S^*) \subseteq \supp(S)$,} \\ \infty, & \text{else,} \end{cases} \\
	&= \sum_{i \in I} w_i D_{\mathrm{KL}}(S_i^*, S_i),
\end{align}
\end{subequations}
with weights $w \in \R_{>0}^{m}$ from~\eqref{eq:sym-Omega} and the supports
\begin{subequations}
\begin{align}
\supp S &= \{ (i,j) \in I \times J \colon S_{ij} \neq 0 \}, \\
\supp S_i &= \{ j \in J \colon S_{ij} \neq 0 \}.
\end{align}
\end{subequations}
Analogously to~\cite{losert1983dynamics}, we consider the index sets
\begin{subequations} \label{eq:J0pm}
\begin{align}
J_0(i) &= \big\{ j \in J \colon (\widehat{\Omega} S^*)_{ij} = \langle S_i^*, (\widehat{\Omega} S^*)_i \rangle \big\}, \\
J_-(i) &= \big\{ j \in J \colon (\widehat{\Omega} S^*)_{ij} > \langle S_i^*, (\widehat{\Omega} S^*)_i \rangle \big\}, \\
J_+(i) &= \big\{ j \in J \colon (\widehat{\Omega} S^*)_{ij} < \langle S_i^*, (\widehat{\Omega} S^*)_i \rangle \big\}
\end{align}
\end{subequations}
and define the continuous functions $Q \colon \ol{\mathcal{W}} \rightarrow \R_{\geq 0}$ and $V \colon \ol{\mathcal{W}} \rightarrow \R_{\geq 0} \cup \{ \infty \}$ by
\begin{subequations}\label{eq:def-Q-V}
\begin{alignat}{2}
Q \colon \ol{\mathcal{W}} &\rightarrow \R_{\geq 0},&\qquad
Q(S) &= \sum_{i \in I} w_i \sum_{j \in J_+(i)} S_{ij},
\label{eq:def-Q-V-Q} \\ \label{eq:def-Q-V-V}
V \colon \ol{\mathcal{W}} &\rightarrow \R_{\geq 0} \cup \{ \infty \},&\qquad
V(S) &= D_{\mathrm{KL}}^w(S^*, S) + 2 Q(S).
\end{alignat}
\end{subequations}
The equilibrium criterion~\eqref{eq:equilibrium-hat} implies
\begin{equation}
\supp(S_i^*) \subseteq J_0(i) \qquad \text{and} \qquad J_-(i), J_+(i) \subseteq J \setminus \supp(S_i^*) \qquad \forall i \in I,
\end{equation}
i.e.\ $V(S^*) = Q(S^*) = 0$. Using the Lyapunov function~\eqref{eq:f-Lyapunov}, we have the following.
\begin{lemma}[cf. {\cite[Proposition 2]{losert1983dynamics}}] \label{lem:convergence2}
There exists $\varepsilon > 0$ such that, if $\| S(t) - S^* \| < \varepsilon$ and $f(S(t)) < f(S^*)$ with $f$ given in~\eqref{eq:f-Lyapunov}, then $\frac{\D}{\D t} V(S(t)) < 0$.
\end{lemma}
\begin{proof}
Since $S(t) \in \mathcal{W}$ for all $t \geq 0$, we have $D_{\mathrm{KL}}^w(S^*, S(t)) < \infty$. Hence
\begin{subequations}\label{eq:dt-DKL}
\begin{align}
\begin{split}
\frac{\D}{\D t}& D_{\mathrm{KL}}^w(S^*, S(t)) \\
	&\stackrel{\substack{\hphantom{\sum_{j \in J} S_{ij}=1} \\ \eqref{eq:def-DKL}}}{=} -\sum_{i \in I} w_i \sum_{j \in \supp(S_i^*)} S_{ij}^* \frac{\dot{S}_{ij}}{S_{ij}} 
\end{split} \\
	&\stackrel{\substack{\hphantom{\sum_{j \in J} S_{ij}=1} \\ \eqref{eq:def-S-flow-F}}}{=} -\sum_{i \in I} \sum_{j \in \supp(S_i^*)} S_{ij}^* \big( (\widehat{\Omega} S)_{ij} - \langle S_i, (\widehat{\Omega} S)_i \rangle \big) \\
	&\stackrel{\hphantom{\sum_{j \in J} S_{ij}=1}}{=} \langle S, \widehat{\Omega} S \rangle - \langle S^*, \widehat{\Omega} S \rangle
	\stackrel{\eqref{eq:sym-Omega}}{=} \langle S, \widehat{\Omega} S \rangle - \langle S, \widehat{\Omega} S^* \rangle \label{eq:dDSt_d}\\
	&\stackrel{\hphantom{\sum_{j \in J} S_{ij}=1}}{=} \langle S, \widehat{\Omega} S \rangle - \langle S^*, \widehat{\Omega} S^* \rangle + \langle S^*, \widehat{\Omega} S^* \rangle - \langle S, \widehat{\Omega} S^* \rangle \\
\begin{split}
	&\stackrel{\sum_{j \in J} S_{ij}=1}{=} \langle S, \widehat{\Omega} S \rangle - \langle S^*, \widehat{\Omega} S^* \rangle + \sum_{i \in I} \sum_{j \in J} S_{ij} \Big( \langle S_i^*, (\widehat{\Omega} S^*)_i \rangle - (\widehat{\Omega} S^*)_{ij} \Big)
\end{split} \\
\begin{split}
	&\stackrel{\substack{\hphantom{\sum_{j \in J} S_{ij}=1} \\ \eqref{eq:J0pm}}}{=} \underbrace{\langle S, \widehat{\Omega} S \rangle - \langle S^*, \widehat{\Omega} S^* \rangle}_{=f(S)-f(S^*)< 0} + \sum_{i \in I} \sum_{j \in J_-(i)} S_{ij} \Big( \underbrace{ \langle S_i^*, (\widehat{\Omega} S^*)_i \rangle - (\widehat{\Omega} S^*)_{ij} }_{< 0} \Big) \\
	&\hphantom{\stackrel{\sum_{j \in J} S_{ij}=1}{=}} \qquad + \sum_{i \in I} \sum_{j \in J_+(i)} S_{ij} \Big( \underbrace{ \langle S_i^*, (\widehat{\Omega} S^*)_i \rangle - (\widehat{\Omega} S^*)_{ij} }_{> 0} \Big).
\end{split}
\end{align}
\end{subequations}
We now focus on $Q(S)$~\eqref{eq:def-Q-V-Q} that is added to the KL-divergence to define $V(S)$ in~\eqref{eq:def-Q-V-V}. We have for each $j \in J_+(i)$
\begin{equation}
\langle S_i, (\widehat{\Omega} S)_i \rangle - (\widehat{\Omega} S)_{ij} \quad\longrightarrow\quad \langle S_i^*, (\widehat{\Omega} S^*)_i \rangle - (\widehat{\Omega} S^*)_{ij} > 0 \qquad \text{as} \quad S \rightarrow S^*.
\end{equation}
Since the limit is positive, there exists $\varepsilon > 0$ such that $\| S - S^* \| < \varepsilon$ implies
\begin{equation}
\langle S_i, (\widehat{\Omega} S)_i \rangle - (\widehat{\Omega} S)_{ij} \geq \frac{3}{4} \Big( \langle S_i^*, (\widehat{\Omega} S^*)_i \rangle - (\widehat{\Omega} S^*)_{ij} \Big), \quad \forall j \in J_+(i), \quad \forall i \in I.
\end{equation}
Consequently,
\begin{subequations}\label{eq:dt-Q(S)}
\begin{align}
\frac{\D}{\D t} Q\big(S(t)\big)
	&\overset{\substack{\eqref{eq:def-S-flow-F} \\ \eqref{eq:def-Q-V-Q}}}{=} \sum_{i \in I} \sum_{j \in J_+(i)} S_{ij} \Big( (\widehat{\Omega} S)_{ij} - \langle S_i, (\widehat{\Omega} S)_i \rangle \Big) \\
	&\overset{\hphantom{\eqref{eq:def-Q-V-Q}}}{\leq} - \frac{3}{4} \sum_{i \in I} \sum_{j \in J_+(i)} S_{ij} \Big( \langle S_i^*, (\widehat{\Omega} S^*)_i \rangle - (\widehat{\Omega} S^*)_{ij}  \Big).
\end{align}
\end{subequations}
Substituting~\eqref{eq:dt-DKL} and~\eqref{eq:dt-Q(S)} into~\eqref{eq:def-Q-V-V}, we finally obtain
\begin{subequations}
\begin{align}
\begin{split}
\frac{\D}{\D t} V(S(t))
	&= \frac{\D}{\D t} D_{\mathrm{KL}}^w(S^*, S(t)) + 2 \frac{\D}{\D t} Q(S(t)) \\
	&\leq \langle S, \widehat{\Omega} S \rangle - \langle S^*, \widehat{\Omega} S^* \rangle + \sum_{i \in I} \sum_{j \in J_-(i)} S_{ij} \Big( \langle S_i^*, (\widehat{\Omega} S^*)_i \rangle - (\widehat{\Omega} S^*)_{ij} \Big) \\
	&\hskip0.25\textwidth - \frac{1}{2} \sum_{i \in I} \sum_{j \in J_+(i)} S_{ij} \Big( \langle S_i^*, (\widehat{\Omega} S^*)_i \rangle - (\widehat{\Omega} S^*)_{ij} \Big)
\end{split} \\
	&< 0.
\end{align}
\end{subequations} \qed
\end{proof}
\begin{proof}[Proof of Theorem~\ref{thm:convergence}]
Let $S^{\ast} \in \Lambda$ be any equilibrium point and $(t_{k})_{k \in \N}$ a corresponding sequence due to~\eqref{eq:def-Lambda-critical}.
We show that $D_{\mathrm{KL}}^w(S^*, S(t)) \rightarrow 0$ for $t \rightarrow \infty$, which is equivalent to the assertion $S(t) \rightarrow S^*$ to be shown.

Choose $\varepsilon > 0$ according to the Lemma~\ref{lem:convergence2}.
There exists $\varepsilon_1 > 0$ such that the (relatively) open set $U = \{ S \in \ol{\mathcal{W}} \colon V(S) < \varepsilon_1 \}$ is contained in $\{ S \in \ol{\mathcal{W}} \colon \| S - S^* \| < \varepsilon \}$.
The function $t \mapsto f(S(t))$ is strictly increasing unless the orbit $\{ S(t) \colon t \geq 0 \}$ consists of an equilibrium.
Hence, $f(S(t)) < f(S^*)$ for all $t \geq 0$.
Since ${S(t_k) \rightarrow S^*}$, we get $S(t_{k_0}) \in U$ for some $k_0 \in \N$. Since then $t \mapsto V(S(t))$ is decreasing, i.e.\ ${V(S(t)) < V(S(t_{k_0})) < \varepsilon_1}$ for all $t > t_{k_0}$, and because $V(S(t))$ is decreasing and $V(S(t_k)) \rightarrow V(S^*) = 0$, we get
\begin{equation}
0 \leq D_{\mathrm{KL}}^w (S^*, S(t)) \leq V(S(t)) \rightarrow 0 \quad \text{for} \quad t \rightarrow \infty,
\end{equation}
which implies $S(t) \rightarrow S^*$ for $t \rightarrow \infty$. \qed
\end{proof}

\subsection{Proof of Proposition~\ref{prop:approx-discr-euler}}
\begin{proof}
For any $t \in \N_{0}$, we set
\begin{subequations}\label{eq:def-Yt}
\begin{align}
Y^{(t)}(\tau)
&= F_{\tau}(S^{(t)})
\overset{\eqref{eq:Euler-scheme}}{=}
\exp_{S^{(t)}}(\tau\Omega S^{(t)})
\intertext{and thus have}\label{eq:Yt-relations}
Y^{(t)}(h) &= S^{(t+1)},\qquad
Y^{(t)}(0) = S^{(t)},\qquad
Y^{(0)}(0)= S_{0}.
\end{align}
\end{subequations}
Formula~\eqref{eq:dexp} implies
\begin{equation}\label{eq:def-dot-Y-t-G}
\dot Y^{(t)}(\tau)
= \frac{\D}{\D\tau}\exp_{S^{(t)}}(\tau\Omega S^{(t)})
= R_{Y^{(t)}(\tau)}(\Omega S^{(t)}) = G(Y^{(t)}),
\end{equation}
where we defined the shorthand $G(Y^{(t)})$.

Now, with $S(t)$ solving the $S$-flow~\eqref{eq:def-S-flow-F}, we estimate with any $T \geq t h$,
\begin{subequations}
\begin{align}
S(T) &- Y^{(t)}(T-th)
\\
&\stackrel{\hphantom{\eqref{eq:def-dot-Y-t-G},\eqref{eq:def-S-flow-F}}}{=} S(t h)-Y^{(t)}(0) + \int_{0}^{T-t h}\frac{\D}{\D\tau}\big(S(th+\tau)-Y^{(t)}(\tau)\big)\D\tau
\\
&\stackrel{\eqref{eq:def-dot-Y-t-G},\eqref{eq:def-S-flow-F}}{=} S(t h)-Y^{(t)}(0) + \int_{0}^{T-t h} \Big(F\big(S(t h+\tau)\big) - G\big(Y^{(t)}(\tau)\big)\Big)\D \tau
\\
\begin{split}
&\stackrel{\hphantom{\eqref{eq:def-dot-Y-t-G},\eqref{eq:def-S-flow-F}}}{=} S(th)-Y^{(t)}(0) + \int_{0}^{T-t h} \Big(F\big(S(t h+\tau)\big) - F\big(Y^{(t)}(\tau)\big)\Big)\D \tau
	\\ &\qquad\qquad
	+ \int_{0}^{T-t h} \Big(F\big(Y^{(t)}(\tau)\big) - G\big(Y^{(t)}(\tau)\big)\Big)\D \tau
\end{split} \\
\begin{split}
&\stackrel{\hphantom{\eqref{eq:def-dot-Y-t-G},\eqref{eq:def-S-flow-F}}}{=} S(th)-Y^{(t)}(0) + \int_{0}^{T-t h} \Big(F\big(S(t h+\tau)\big) - F\big(Y^{(t)}(\tau)\big)\Big)\D \tau
\\ &\qquad\qquad
	+ \int_{0}^{T-t h} \int_{0}^{\tau} \frac{\D}{\D\tau}\Big(F\big(Y^{(t)}(\tau)\big) - G\big(Y^{(t)}(\tau)\big)\Big)\Big|_{\tau=\lambda}\D\lambda\D\tau
\end{split} \\
\begin{split}
	&\stackrel{\hphantom{\eqref{eq:def-dot-Y-t-G},\eqref{eq:def-S-flow-F}}}{=}  S(th) - Y^{(t)}(0) + \int_{0}^{T-t h} \Big(F\big(S(t h+\tau)\big) - F\big(Y^{(t)}(\tau)\big)\Big)\D \tau
\\ &\qquad\qquad
+ \int_{0}^{T-t h} \int_{0}^{\tau} \Big(dF\big(Y^{(t)}(\lambda)\big)\big[G\big(Y^{(t)}(\lambda)\big)\big] \\
	&\qquad\qquad \hphantom{+ \int_{0}^{T-t h} \int_{0}^{\tau} \Big(}\qquad - dG\big(Y^{(t)}(\lambda)\big)\big[G\big(Y^{(t)}(\lambda)\big)\big]\Big)\D\lambda\D\tau.
\end{split}
\end{align}
\end{subequations}
By assumption, $F$ given by~\eqref{eq:def-S-flow-F} is $C^{1}$, as is $G$ given by~\eqref{eq:def-dot-Y-t-G} which has the same form. Consequently, regarding the integrand of the last integral, since $\ol{\mc{W}}$ is compact there exists a constant $C$ such that
\begin{equation}
\Big\|
dF\big(Y^{(t)}(\lambda)\big)\big[G\big(Y^{(t)}(\lambda)\big)\big] -  dG\big(Y^{(t)}(\lambda)\big)\big[G\big(Y^{(t)}(\lambda)\big)\big]
\Big\| \leq C,\qquad \forall Y^{(t)} \in \ol{\mathcal{W}}.
\end{equation}
Hence,
\begin{subequations}
\begin{align}
\|S(T) &- Y^{(t)}(T-th)\|
\\
\begin{split}
&\leq \|S(th) - Y^{(t)}(0)\| + \int_{0}^{T-t h} \big\|F\big(S(t h+\tau)\big) - F\big(Y^{(t)}(\tau)\big)\big\|\D \tau \\
&\qquad + C \int_{0}^{T-t h} \int_{0}^{\tau}\D\lambda\D\tau
\end{split} \\
\begin{split}
&\leq \|S(th) - Y^{(t)}(0)\| + L \int_{0}^{T-t h} \|S(t h+\tau)-Y^{(t)}(\tau)\|\D\tau \\
&\qquad + \frac{C}{2}(T-t h)^{2}
\end{split}
\end{align}
\end{subequations}
Applying Gronwall's inequality~\cite[Lemma~2.7]{Teschl:2012aa} yields
\begin{equation}
\|S(T) - Y^{(t)}(T-th)\|
\leq \Big(\|S(th) - Y^{(t)}(0)\| + \frac{C}{2}(T-t h)^{2}\Big) e^{L(T-t h)}
\end{equation}
and setting $T=(t+1) h$
\begin{subequations}
\begin{align}
\big\|S\big((t+1)h\big)-Y^{(t)}(h)\big\|
&\overset{\eqref{eq:Yt-relations}}{=}
\big\|S\big((t+1)h\big)-S^{(t+1)}\big\| \\
&\overset{\eqref{eq:Yt-relations}}{\leq}
\Big(\|S(th) - S^{(t)}\| + \frac{C h^{2}}{2}\Big) e^{L h}.
\end{align}
\end{subequations}
Thus,
\begin{subequations}
\begin{align}
\|S(t h)-S^{(t)}\|
&\leq \Big(\|S\big((t-1)h\big) - S^{(t-1)}\| + \frac{Ch^{2}}{2}\Big) e^{L h}
\\
&\leq \Big(\Big(\|S\big((t-2)h\big) - S^{(t-2)}\| + \frac{Ch^{2}}{2}\Big) e^{L h} + \frac{Ch^{2}}{2}\Big) e^{L h}
\\
&= \big\|S\big((t-2)h\big) - S^{(t-2)}\big\|e^{2 L h}
+ \frac{Ch^{2}}{2}\big(e^{2 L h} + e^{L y}\big)
\\
&= \underbrace{\|S(0) - S^{(0)}\|}_{=0\;\text{by~\eqref{eq:Yt-relations}}} e^{t L h} + \frac{Ch^{2}}{2}\sum_{k \in [t]}e^{k L h}
\\
&= \frac{Ch^{2}}{2}\Big(\frac{e^{(t+1) L h}-1}{e^{L h}-1}-1\Big)
= \frac{Ch^{2}}{2}e^{L h}\frac{e^{t L h}-1}{e^{L h}-1}
\intertext{and using $e^{L h}\geq 1+L h$}
&\leq \frac{C}{2 L} h e^{(t+1) L h},\qquad \forall t \in \N.
\end{align}
\end{subequations} \qed
\end{proof}

\section{Stability Statements for Dynamical Systems}\label{sec:Appendix}
We state basic results from the literature which are used to analyze the stability of the equilibria of the $S$-flow in Section~\ref{sec:convergence-stability}. 
\begin{theorem} \label{thm:stability}
Let $x^*$ be an equilibrium point of the system $\dot{x}(t) = F(x(t))$ with $F \in C^1(U, \R^n)$.
\begin{enumerate}[(a)]
\item If all eigenvalues of the Jacobian matrix $\tfrac{\partial F}{\partial x}(x^*)$ have negative real part, then $x^*$ is exponentially stable.
\item If the Jacobian matrix $\tfrac{\partial F}{\partial x}(x^*)$ has an eigenvalue with positive real part, then $x^*$ is unstable.
\end{enumerate}
\end{theorem}
Statement (a) can be found in~\cite[Theorem~6.10]{Teschl:2012aa}. For statement (b) we refer to~\cite[Proposition~6.2.1]{Schaeffer:2016aa}. These stability criteria concern flows $\dot{x}(t) = F(x(t))$ on an open subset $U \subseteq \R^n$.

Since we regard the $S$-flow as a flow on the compact set $\ol{\mathcal{W}}$, we need a few additional arguments. In~\cite[Section 6.8.4]{Schaeffer:2016aa}, a direct proof of theorem~\ref{thm:stability}(b) is sketched. Since we employ techniques that are used in that proof for our own analysis, we summarize the main statements in the following proposition for the reader's convenience. Informally, the proposition states that, if $\tfrac{\partial F}{\partial x}(x^*)$ has an eigenvalue with positive real part, then there exists an open truncated cone at $x^*$ where the flow $\dot{x} = F(x)$ is repelled from $x^*$.
\begin{proposition} \label{prop:unstable-cone}
Let $x^*$ be an equilibrium point of $\dot{x}(t) = F(x(t))$ with $F \in C^1(U, \R^n)$. Then
\begin{enumerate}[(a)]
\item There exist a sufficiently small $\varepsilon_1 > 0$ and a (real) similarity transform
\begin{equation}
V^{-1} \tfrac{\partial F}{\partial x}(x^*) V = \begin{pmatrix} A_{\text{sc}} & 0 \\ 0 & A_{\text{u}} \end{pmatrix} = A,
\end{equation}
such that
\begin{enumerate}[(i)]
\item $\Re(\lambda) \leq 0$ for all eigenvalues $\lambda$ of $A_{\text{sc}}$,
\item $\Re(\lambda) > 0$ for all eigenvalues $\lambda$ of $A_{\text{u}}$,
\item $\langle y_{\text{sc}}, A_{\text{sc}} y_{\text{sc}} \rangle \leq \tfrac{\varepsilon_1}{4} \| y_{\text{sc}} \|_2^2$,
\item $\langle y_{\text{u}}, A_{\text{u}} y_{\text{u}} \rangle \geq \varepsilon_1 \| y_{\text{u}} \|_2^2$.
\end{enumerate}
\item Suppose $\tfrac{\partial F}{\partial x}(x^*)$ has at least one eigenvalue $\lambda$ with $\Re(\lambda) > 0$.
Considering an affine coordinate transform $y=V^{-1}(x - x^*)$ with $V \in \GL_n(\R)$ from (a), the resulting flow
\begin{equation}
\dot{y} = G(y) = V^{-1} F(Vy + x^*),
\end{equation}
which has the equilibrium $y^* = 0$ with $\tfrac{\partial G}{\partial y}(0) = V^{-1} \tfrac{\partial F}{\partial x}(x^*) V = A$, has the following property.
There exist $\eta > 0$, $\delta > 0$ and $\varepsilon > 0$, such that if the flow starts at some point in the open truncated cone
\begin{equation}
U_{\eta, \delta} = \Big\{ y = \begin{psmallmatrix} y_{\text{sc}} \\ y_{\text{u}} \end{psmallmatrix} \in \R^n \colon \| y_{\text{sc}} \|_2^2 < \eta \| y_{\text{u}} \|_2^2,\ \| y \|_2 < \delta \Big\}
\subset B_{\delta}(0) = \big\{ y \in \R^n \colon \| y \|_2 < \delta \big\},
\end{equation}
then the solution will not cross the conical portion of $\partial U_{\eta, \delta}$, i.e.
\begin{equation}
\big\{ y \in \R^{n} \colon \| y_{\text{sc}} \|_2^2 = \eta \| y_{\text{u}} \|_2^2,\ \| y \|_2 < \delta \big\},
\end{equation}
and it fulfills $\| y(t) \| \geq \| y(0) \| e^{\varepsilon t}$ as long as $y(t) \in U_{\eta,\delta}$, i.e., $y(t)$ leaves the ball $B_{\delta}(0)$ at some time point. Especially, the equilibrium $y^* = 0$ is unstable. This property is accordingly transferred to the equilibrium $x^*$ of $\dot{x}(t) = F(x(t))$ using $x(t) = V y(t) + x^*$.
\end{enumerate}
\end{proposition}
We note that if $\tfrac{\partial F}{\partial x}(x^*)$ is diagonalizable with real eigenvalues then the similarity transform in proposition~\ref{prop:unstable-cone}(a) is just the diagonalization.
In general, if $v \in \C^n$ is an eigenvector of $\tfrac{\partial F}{\partial x}(x^*)$ corresponding to an eigenvalue $\lambda \in \C$ with $\Re(\lambda) > 0$ and $V \in \GL_n(\R)$ is given by Proposition~\ref{prop:unstable-cone}(a), then $V^{-1} \Re(v) = \binom{0}{y_{\text{u}}}$ and $V^{-1} \Im(v) = \binom{0}{\tilde{y}_{\text{u}}}$.